\documentclass[oneside,11pt,reqno]{amsart}
\usepackage{amsmath,amssymb,amsthm,xypic,mathrsfs} 
\input xy
\xyoption{all}

\usepackage{hyperref}

\usepackage[totalheight=22.5 true cm, totalwidth=16 true cm]{geometry}

\usepackage{color}

\usepackage[shortcuts]{extdash}

\title{Almost-simple affine difference algebraic groups}

\author{Michael Wibmer}
\address{Michael Wibmer, Institute of Analysis and Number Therory, Graz University of Technology, Kopernikusgasse~24, 8010 Graz, Austria, \url{https://sites.google.com/view/wibmer}}
\email{wibmer@math.tugraz.at}

\thanks{The author was supported by the NSF grants DMS-1760212, DMS-1760413, DMS-1760448 and the Lise Meitner grant M-2582-N32 of the Austrian Science Fund FWF}

\newtheorem{theo}{Theorem}[section]
\newtheorem{lemma}[theo]{Lemma}
\newtheorem{prop}[theo]{Proposition}
\newtheorem{cor}[theo]{Corollary}
\newtheorem{defi}[theo]{Definition}
\newtheorem{rem}[theo]{Remark}

\newtheorem*{theoremA}{Theorem A}
\newtheorem*{theoremB}{Theorem B}

\theoremstyle{definition}
\newtheorem{ex}[theo]{Example}

\newcommand{\f}{\phi}

\newcommand{\ida}{\mathfrak{a}}

\newcommand{\p}{\mathfrak{p}}
\newcommand{\q}{\mathfrak{q}}
\newcommand{\m}{\mathfrak{m}}

\newcommand{\spec}{\operatorname{Spec}}

\newcommand{\G}{\mathcal{G}}

\newcommand{\Gl}{\operatorname{GL}}

\newcommand{\Aut}{\operatorname{Aut}}
\newcommand{\Hom}{\operatorname{Hom}}

\newcommand{\V}{\mathbb{V}}

\newcommand{\id}{\operatorname{id}}

\newcommand{\sdim}{\sigma\text{-}\dim}
\newcommand{\ld}{\operatorname{ld}}

\newcommand{\ord}{\operatorname{ord}}

\newcommand{\X}{\mathcal{X}}

\newcommand{\N}{\mathcal{N}}

\newcommand{\s}{\sigma}

\newcommand{\nn}{\mathbb{N}}

\newcommand{\I}{\mathbb{I}}
\newcommand{\vv}{\mathcal{V}}

\newcommand{\hs}{{{}^\sigma\!}}
\newcommand{\VV}{\mathcal{V}}

\newcommand{\ks}{$k$\=/$\s$}

\newcommand{\Gm}{\mathbb{G}_m}

\newcommand{\hsi}{^{\sigma^i}\!}

\newcommand{\iar}{\hookrightarrow}
\newcommand{\Ga}{\mathbb{G}_a}

\def\H{\mathcal{H}}
\newcommand{\sar}{\twoheadrightarrow}
\newcommand{\red}{{\operatorname{red}}}
\newcommand{\wm}{{\operatorname{wm}}}
\newcommand{\sred}{{\sigma\text{-}\operatorname{red}}}
\newcommand{\per}{{\operatorname{per}}}
\newcommand{\GG}{\mathsf{G}}

\newcommand{\NN}{\mathsf{N}}

\newcommand{\Emb}{\operatorname{Emb}}

\begin{document}



\date{\today}

\begin{abstract}
Affine difference algebraic groups are a generalization of affine algebraic groups obtained by replacing algebraic equations with algebraic difference equations.
We show that the isomorphism theorems from abstract group theory have meaningful analogs for these groups and we establish a Jordan-H\"{o}lder type theorem that allows us to decompose any affine difference algebraic group into almost-simple affine difference algebraic groups. We also characterize almost-simple affine difference algebraic groups via almost-simple affine algebraic groups.
\end{abstract}

\maketitle


\section*{Introduction}

Affine algebraic groups can be described as subgroups of a general linear group defined by polynomials in the matrix entries. In a similar spirit, affine difference algebraic groups can be described as subgroups of a general linear group defined by difference polynomials in the matrix entries, i.e., the defining equations involve a formal symbol $\s$ that has to be interpreted as a ring endomorphism. Many concepts and results from the theory of algebraic groups have meaningful analogs for affine difference algebraic groups, e.g., the \emph{$\s$-dimension} is a measure for the size of an affine difference algebraic group analogous to the dimension of algebraic varieties. For example, the full general linear group $\Gl_n$, considered as a difference algebraic group, has $\s$\=/dimension $n^2$, while the difference algebraic subgroup $G=\{g\in\Gl_n\ |\ \s(g)^Tg=g\s(g)^T=I_n\}$ of $\Gl_n$ has $\s$-dimension zero. 

An affine difference algebraic group is \emph{strongly connected} if it has positive $\s$-dimension and no proper difference algebraic subgroup of the same $\s$\=/dimension. The eponymous protagonists of this article, the \emph{almost-simple affine difference algebraic groups}, are the strongly connected affine difference algebraic groups with the property that every proper normal difference algebraic subgroup has $\s$-dimension zero.
For example, as we show, an almost-simple affine algebraic group, considered as a difference algebraic group, is an almost-simple affine difference algebraic group. 

The main goal of this paper is to elucidate the structure of affine difference algebraic groups ``up to $\s$-dimension zero''. Another crucial notion for this objective, besides the notion of almost-simple affine difference algebraic groups, is the concept of isogeny: Two affine difference algebraic groups $G_1$ and $G_2$ are \emph{isogenous} if there exists an affine difference algebraic group $H$ and surjective morphisms $H\to G_1$ and $H\to G_2$ with kernels of $\s$-dimension zero. Our first main result is a Jordan-H\"{o}lder type theorem for affine difference algebraic groups (Theorem~\ref{theo: Jordan Hoelder}):

\begin{theoremA} \label{theo: Theorem A}
 Let $G$ be a strongly connected affine difference algebraic group. Then there exists a subnormal series $$G=G_0\supseteq G_1\supseteq\cdots\supseteq G_n=1$$
of strongly connected difference algebraic subgroups of $G$ such that $G_i/G_{i+1}$ is almost-simple for $i=0,\ldots,n-1$. If
	$$G=H_0\supseteq H_1\supseteq\cdots\supseteq H_m=1$$
	is another such subnormal series, then $m=n$ and there exists a permutation $\pi$ such that
	$G_{i}/G_{i+1}$ and $H_{\pi(i)}/H_{\pi(i)+1}$ are isogenous for $i=0,\ldots,n-1$.
\end{theoremA}

We also show that any affine difference algebraic group of positive $\s$-dimension has a strong identity component that is strongly connected, i.e., a (unique) minimal difference algebraic subgroup of the same $\s$\=/dimension. Therefore, the above theorem yields a decomposition result for arbitrary affine difference algebraic groups. 

Theorem A prompts us to determine the structure of the almost-simple affine difference algebraic groups. This is the content of our second main result (Theorem \ref{theo: main almost-simple}):

\begin{theoremB}
	A strongly connected affine difference algebraic group is almost-simple if and only if it is isogenous to an almost-simple affine algebraic group, considered as an affine difference algebraic group.	
\end{theoremB}

Difference algebraic groups are the discrete analog of differential algebraic groups and the latter have always played an important role in differential algebra. See, e.g., the textbooks \cite{Kolchin:differentialalgebraicgroups} and \cite{Buium:DifferntialAlgebraicGroupsOfFiniteDimension} on differential algebraic groups. The last couple of years have seen an exciting and fruitful interaction between the theory of differential algebraic groups and the Galois theory of linear differential or difference equations depending on a differential parameter, also known as parameterized Picard-Vessiot theory (\cite{Landesman:GeneralizedDifferentialGaloisTheory}, \cite{CassidySinger:GaloisTheoryofParameterizedDifferentialEquations}, \cite{HardouinSinger:DifferentialGaloisTheoryofLinearDifferenceEquations}). 
In this Galois theory the Galois groups are differential algebraic groups and in this capacity they measure the differential algebraic relations (with respect to an auxiliary derivation) among the solutions of linear differential or difference equations. The structure theory of differential algebraic groups has facilitated the development of very strong hypertranscendence criteria that have been applied to various special functions (\cite{Arreche:AGaloisTheoreticProofOfTheDifferentialTranscendenceOfTheIncompleteGammaFunction}, \cite{DiVizio:ApprocheGaloisienneDeLaTranscendanceDifferentielle}, \cite{HardouinOvchinnikov:CalculatingGaloisGroups} \cite{DreyfusHardouinRoques:HypertranscendenceOfSolutionsOfMahlerFunctions}, \cite{HardouinMinchenkoOvchinnikov:CalculatingDifferentialGaloisGroupsOfParameterizedDifferentialEqautionsWithApplications}, \cite{Hardouin:GaloisianApproachToDifferentialTranscendence}, \cite{ArrecheSinger:GaloisGroupsForIntegrableAndProjectivelyIntegrableLinearDifferenceequations}, \cite{ArrecheDreyfusRoques:DifferentialTranscendenceCriteriaForSecondOrderLinearDifferenceEquationsAndEllipticHypergeometricFunctions}, \cite{AdamczewskiDreyfusHardouin:HypertranscendenceAndLinearDifferenceEquations}) and the development of algorithms for computing these Galois groups (\cite{Dreyfus:ComputingTheGaloisGroupOfSomeParameterizedLinearDifferentialEquationOrderTwo}, \cite{Arreche:ComputingTheDifferentialGaloisGroupOfaParameterizedSecondOrderLinearDifferentialEquation}, \cite{Arreche:OnTheComputationOfTheParameterizedDifferentialGaloisGroupForASecondOrderLinearDifferentialEquationwithDifferentialParameters}, \cite{Arreche:ComputationOfTheDifferenceDifferentialGaloisGroupAndDifferentialRelationsAmongSolutionsForASecondorderLinearDifferenceEquation}, \cite{MinchenkoOvchinnikovSinger:UnipotentDifferentialAlgebraicGroupsAsParameterizedDifferentialGaloisGroups}, \cite{MinchenkoOvchinnikovSinger:ReductiveLinearDifferentialAlgebraicGroupsAndTheGaloisGroups}, \cite{MinchenkoOvchinnikov:CalculatingGaloisGroupsOfThirdOrderLinearDifferentialEquationsWithParameters}).

Similar Galois theories exist for linear differential or difference equations depending on a discrete parameter (\cite{DiVizioHardouinWibmer:DifferenceGaloisTheoryOfLinearDifferentialEquations}, \cite{OvchinnikovWibmer:SGaloisTheoryOfLinearDifferenceEquations}). In these Galois theories the Galois groups are affine difference algebraic groups and they measure the difference algebraic relations among the solutions. While there has been some progress (\cite{DiVizioHardouinWibmer:DifferenceAlgebraicRel}, \cite{DreyfusHardouinRoques:FunctionalRelationsOfSolutionsOfqdifferenceEquations}, \cite{BachmayrWibmer:AlgebraicGroupsAsDifferenceGaloisGroupsOfLinearDifferentialEquations}), the case of discrete parameters is far less developed than the case of differential parameters and the difference analogs of
current results at the interface of differential algebraic groups and parameterized Picard-Vessiot theory, such as \cite{SanchezPillay:DifferentialGaloisCohomologyAndParameterizedPicardVessiotExtensions}, are far beyond reach at the moment. This is mainly due to the fact that the theory of difference algebraic groups is practically non-existent.
In sharp contrast to the situation in differential algebra, difference algebraic groups have long played no role at all in difference algebra. At present, only a few scattered results on difference algebraic groups are available in the literature: Some results relating to cohomology of difference algebraic groups are in \cite{BachmayrWibmer:TorsorsForDifferenceAlgebraicGroups}, \cite{ChalupnikKowalski:DifferenceModulesAndDifferenceCohomology}, \cite{ChalupnikKowalski:DifferenceSheavesAndTorsors}, \cite{Tomasic:AToposTheoreticViewOfDifferenceAlgebra}. Groups definable in ACFA, the model companion of difference fields, played a crucial role in Hrushovski's proof of the Manin-Mumford conjecture in \cite{Hrushovski:TheManinMumfordConjectureAndTheModelTheorOfDifferenceFields}. The relation between affine difference algebraic groups and groups definable in $\operatorname{ACFA}$ is somewhat analogous to the relation between affine group schemes of finite type over a field and groups definable in $\operatorname{ACF}$, the theory of algebraically closed fields.
Further results related to the Manin-Mumford conjecture and groups definable in ACFA are in
\cite{Scanlon:DifferenceAlgebraicSubgroupsOfCommutativeAlgebraicGroups}, \cite{KowalskiPillay:ANoteonGroupsDefinableInDifferenceFields} \cite{Scanlon:APositiveCharacteristicManinMumfordTheorem}, \cite{KowalskiPillay:OnAlgebraicSigmaGroups}, \cite{ChatzidakisHrushovski:OnSubgroupsOfSemiabelianVarietiesDefinedByDifferenceEquations}.

Because of the present infantile state of the theory of difference algebraic groups, the purpose of this article is also to lay the groundwork for a further comprehensive study of affine difference algebraic groups.\footnote{Some further steps in this direction can be found in the author's habilitation thesis \cite{Wibmer:Habil}, which encompasses the first seven sections of this article.} To this end we build on \cite{Wibmer:FinitenessPropertiesOfAffineDifferenceAlgebraicGroups}, where some basic finiteness properties of affine difference algebraic groups have been established and numerical invariants for affine difference algebraic groups, such as the $\s$-dimension $\sdim(G)$ and the limit degree $\ld(G)$, have been introduced. 
On our path to Theorems A and B above we encounter several basic results and constructions that we deem fundamental for the further development of the theory of affine difference algebraic groups:
\begin{itemize}
	\item We introduce four different difference algebraic subgroups of an affine difference algebraic group that are in a certain sense analogous to the maximal reduced subgroup of an affine algebraic group. 
	\item We establish the existence of the quotient $G/N$ of an affine difference algebraic group $G$ by a normal difference algebraic subgroup $N$ and show that it is well-behaved, e.g., $\sdim(G/N)=\sdim(G)-\sdim(N)$ and $\ld(G/N)=\frac{\ld(G)}{\ld(N)}$.
	\item We show that every morphism $G\to H$ of affine difference algebraic groups factors uniquely as a quotient map followed by an embedding.
	\item We establish the analogs of the isomorphism theorems from abstract group theory. This, in particular, includes formulas such as $H/(H\cap N)\simeq HN/N$ or $(G/N)/(H/N)\simeq G/H$ and the correspondence between difference algebraic subgroups of $G/N$ and difference algebraic subgroups of $G$ containing $N$. 
	\item We introduce and study the identity component $G^o$ and the strong identity component $G^{so}$ of an affine difference algebraic group. In particular, we show that $G^o$ is a characteristic subgroup of $G$ (in the sense that it is stable under automorphisms even after base change) and we isolate conditions that guarantee that $G^{so}$ is normal in $G$. 
\end{itemize}


Let us elaborate a little more on the first point in the above list: It is well-recognized (see e.g., \cite{Milne:AlgebraicGroupsTheTheoryOfGroupSchemesOfFiniteTypeOverAField}) that allowing nilpotent elements in the coordinate rings of affine algebraic groups has its benefits. The situation for affine difference algebraic groups is similar. Without allowing ``$\s$-nilpotent'' elements, the Galois correspondences in \cite{DiVizioHardouinWibmer:DifferenceGaloisTheoryOfLinearDifferentialEquations} and \cite{OvchinnikovWibmer:SGaloisTheoryOfLinearDifferenceEquations} would not be complete and the isomorphism theorems for affine difference algebraic groups would not hold. In difference algebraic geometry there is a whole zoo of elements playing a role analogous to nilpotent elements in algebraic geometry. They roughly correspond to the following
assertions valid for elements in a difference field but not generally valid
for elements in a difference ring:
\begin{itemize}
	\item $a^n = 0$ implies $a = 0$.
	\item $\s(a) = 0$ implies $a = 0$.
	\item $ab = 0$ implies $a\s(b) = 0$.
	\item  $a\s(a) = 0$ implies $a = 0$.
\end{itemize}
In this spirit we obtain four difference algebraic subgroups of an affine difference algebraic
group that play a role analogous to the maximal reduced subgroup of an affine algebraic
group.

\medskip

We conclude this introduction with an outline of the article: In Section 1 we go through the details of the definition of affine difference algebraic groups and we recall the necessary results from \cite{Wibmer:FinitenessPropertiesOfAffineDifferenceAlgebraicGroups}. Sections 2 to 6 roughly correspond to the five bullet points in the above list. In Section 7 we establish our Jordan H\"{o}lder type theorem (Theorem A) and in the final section on almost-simple affine difference algebraic groups we prove Theorem B.

\section{Notation and preliminaries}

In this section we introduce some notation that will be used throughout the text. 
We also recall the required constructions and results from \cite{Wibmer:FinitenessPropertiesOfAffineDifferenceAlgebraicGroups}.
The reader familiar with \cite{Wibmer:FinitenessPropertiesOfAffineDifferenceAlgebraicGroups} may safely skip this section.

All rings are assumed to be commutative and unital. The natural numbers $\nn$ include $0$. We begin by recalling the jargon of difference algebra. Standard references for difference algebra are \cite{Cohn:difference} and \cite{Levin:difference}. However, note that in these references the transforming operators are always assumed to be injective. Here we do not make this assumption.

A \emph{difference ring} (or \emph{$\s$-ring} for short) is a ring $R$ together with a ring endomorphism $\s\colon R\to R$. We usually omit $\s$ from the notation and simply refer to $R$ as a $\s$-ring. Moreover, as customary, we use the same symbol $\s$ for various different endomorphisms. A morphism between $\s$-rings $R$ and $S$ is a morphism $\psi\colon R\to S$ of rings such that
$$
\xymatrix{
R \ar^-\psi[r] \ar_\s[d] & S \ar^\s[d]\\
R \ar^-\psi[r] & S 	
}
$$
commutes. A \emph{$\s$-subring} of a $\s$-ring is a subring that is stable under $\s$. A $\s$-ring $R$ is \emph{inversive} if $\s\colon R\to R$ is bijective. A \emph{$\s$-field} is a $\s$-ring whose underlying ring is a field. If $K$ is a $\s$-subring of a $\s$-ring $L$ such that $K$ and $L$ are $\s$-fields, then $L$ is a \emph{$\s$-field extension} of $K$.

A \emph{$\s$-ideal} of a $\s$-ring $R$ is an ideal $\ida$ of $R$ such that $\s(\ida)\subseteq \ida$. In this case $R/\ida$ is naturally a $\s$-ring such that the canonical map $R\to R/\ida$ is a morphism of $\s$-rings. For a subset $F$ of $R$ the smallest $\s$-ideal containing $F$ is denoted by $[F]$. It is called the \emph{$\s$-ideal $\s$-generated by $F$} and agrees with the ideal generated by $\s^i(f)$ ($i\in\nn,\ f\in F$). A $\s$-ideal $\ida$ is \emph{finitely $\s$-generated} if there exists a finite subset $F$ of $\ida$ such that $\ida=[F]$.

Let $k$ be a $\s$-ring. A \emph{\ks-algebra} is a $\s$-ring $R$ together with a morphism $k\to R$ of $\s$-rings. A morphism of \ks-algebras is a morphism of $\s$-rings that is a morphism of $k$-algebras. 
The tensor product $R\otimes_k S$ of two \ks-algebras is a \ks-algebra via $\s(r\otimes s)=\s(r)\otimes \s(s)$. A \emph{\ks-subalgebra} of a \ks-algebra is a $\s$-subring that is a $k$-subalgebra. For a subset $F$ of a \ks-algebra $R$, the smallest \ks-subalgebra of $R$ containing $F$ is denoted by $k\{F\}$. It is called the \emph{\ks-subalgebra $\s$-generated by $F$} and agrees with the $k$-subalgebra of $R$ generated by $\s^i(f)$ ($i\in\nn,\ f\in F)$. A \ks-algebra $R$ is \emph{finitely $\s$-generated} (over $k$) if there exists a finite subset $F$ of $R$ with $R=k\{F\}$. The \emph{$\s$-polynomial ring} in the $\s$-variables $y_1,\ldots,y_n$ over $k$ is 
$$k\{y_1,\ldots,y_n\}=k\left[\s^i(y_j)|\ i\in\nn,\ 1\leq j\leq n\right],$$
where the action of $\s$ on $k\{y_1,\ldots,y_n\}$ extends the action of $\s$ on $k$ and $\s$ acts on the variables $\s^i(y_j)$ as suggested by their names. The order a $\s$-polynomial $f\in k\{y_1,\ldots,y_n\}$ is the largest power of $\s$ that occurs in $f$. For a \ks-algebra $R$, a $\s$-polynomial $f\in k\{y_1,\ldots,y_n\}$ and $x=(x_1,\ldots,x_n)\in R^n$, we denote with $f(x)$ the element of $R$ obtained from $f$ by specializing $\s^i(y_j)$ to $\s^i(x_j)$. For $F\subseteq k\{y_1,\ldots,y_n\}$ the set of $R$-valued solutions of $F$ is
$$\V_R(F)=\{x\in R^n|\ f(x)=0 \ \forall \ f\in F\}.$$ Note that   
$R\rightsquigarrow\V_R(F)$ is naturally a functor from the category of \ks-algebras to the category of sets.

\begin{defi} \label{defi: svariety}
	A \emph{$\s$-variety} over $k$ is a functor from the category of \ks-algebras to the category of sets that is isomorphic to a functor of the form $R\rightsquigarrow \V_R(F)$ for some $n\geq 1$ and $F\subseteq k\{y_1,\ldots,y_n\}$.
\end{defi}
It would be more accurate to add the word ``affine'' into the above definition. We chose not to do so because we have no need to consider non-affine $\s$-varieties in this article and to avoid countless repetitions of the word ``affine''.
A \emph{morphism of $\s$-varieties} over $k$ is a natural transformation of functors.

If $X=\V(F)$ is the $\s$-variety defined by $F\subseteq k\{y_1,\ldots,y_n\}$, i.e., $X(R)=\V_R(F)$ for all \ks-algebras $R$, then $$\I(X)=\left\{f\in k\{y_1,\ldots,y_n\}|\ f(x)=0 \ \forall\ x\in X(R) \ \forall\ \text{\ks-algebras } R\right\}$$
is a $\s$-ideal of $k\{y_1,\ldots,y_n\}$ that agrees with $[F]$ (choose $R=k\{y_1,\ldots,y_n\}/[F]$). The \ks\=/algebra
$k\{X\}=k\{y_1,\ldots,y_n\}/\I(X)$ is called the \emph{coordinate ring} of $X$. For every \ks-algebra $R$ we have a bijection $\Hom(k\{X\},R)\simeq X(R)$ that assigns to a morphism $\psi\colon k\{X\}\to R$ of \ks-algebras the tuple $(\psi(\overline{y_1}),\ldots,\psi(\overline{y_n}))\in R^n$. As these bijections are functorial in $R$, we see that $X$ is represented by $k\{X\}$. It follows that a functor from the category of \ks-algebras to the category of sets is a $\s$-variety if and only if it is representable by a finitely $\s$-generated \ks-algebra. By the Yoneda Lemma the \ks-algebra representing a $\s$-variety $X$ is uniquely determined up to a unique isomorphism. As above, it is called the coordinate ring of $X$ and denoted by $k\{X\}$. 
Moreover, $X\rightsquigarrow k\{X\}$ is an equivalence of categories between the category of $\s$-varieties over $k$ and the category of finitely $\s$-generated \ks-algebras. We will usually identify $X$ with $\Hom(k\{X\},-)$.

For a morphism $\f\colon X\to Y$ of $\s$-variety the corresponding morphism $\f^*\colon k\{Y\}\to k\{X\}$ of \ks-algebras is called the \emph{dual} of $\f$. For $\s$-varieties $X$ and $Y$ the functor $X\times Y$ given by  $R\rightsquigarrow X(R)\times Y(R)$ is a product in the category of $\s$-varieties over $k$. In fact, $k\{X\times Y\}=k\{X\}\otimes_k k\{Y\}$.

Let $X$ be a $\s$-variety. An element $f\in k\{X\}$ defines for every \ks-algebra $R$ a map $f\colon X(R)\to R,\  \psi\mapsto \psi(f)$. For $F\subseteq k\{X\}$ the subfunctor $Y=\V(F)$ of $X$ given by
$Y(R)=\{x\in X(R)|\ f(x)=0 \ \forall f\in F\}$ for all \ks-algebras $R$ is called the \emph{$\s$\=/closed $\s$\=/subvariety} of $X$ defined by $F$. Note that $Y$ is a $\s$-variety with coordinate ring $k\{Y\}=k\{X\}/[F]$.
The map $\ida\mapsto \V(\ida)$ is a bijection between the $\s$-ideals of $k\{X\}$ and the $\s$-closed $\s$-subvarieties of $X$ (\cite[Lemma 1.4]{Wibmer:FinitenessPropertiesOfAffineDifferenceAlgebraicGroups}). The $\s$-ideal corresponding to a $\s$-closed $\s$-subvariety $Y$ of $X$ is denoted by $\I(Y)\subseteq k\{X\}$ and called the \emph{defining ideal} of $Y$. We use the notation ``$Y\subseteq X$'' to indicate that $Y$ is a $\s$-closed $\s$-subvariety of $X$.

The intersection $Y_1\cap Y_2$ of two $\s$-closed $\s$-subvarieties of $X$ is defined by $(Y_1\cap Y_2)(R)=Y_1(R)\cap Y_2(R)$ for any \ks-algebra $R$. It is a $\s$-closed $\s$-subvariety of $X$ corresponding to the sum of $\s$-ideals. 

A morphism $\f\colon X\to Y$ of $\s$-varieties is a \emph{$\s$-closed embedding} if it induces an isomorphism between $X$ and a $\s$-closed $\s$-subvariety of $Y$. This is equivalent to $\f^*\colon k\{Y\}\to k\{X\}$ being surjective (\cite[Lemma 1.6]{Wibmer:FinitenessPropertiesOfAffineDifferenceAlgebraicGroups}). We will usually indicate a $\s$-closed embedding as $X\hookrightarrow Y$.

For a morphism $\f\colon X\to Y$ of $\s$-varieties there exists a unique $\s$-closed $\s$-subvariety $\f(X)$ of $Y$
such that $\f$ factors through the inclusion $\f(X)\subseteq Y$ and for any other $\s$-closed $\s$-subvariety $Z$ of $Y$ such that $\f$ factors through $Z\subseteq Y$, one has $\f(X)\subseteq Z$ (\cite[Lemma 1.5]{Wibmer:FinitenessPropertiesOfAffineDifferenceAlgebraicGroups}). In fact, $\f(X)$ is the $\s$-closed $\s$-subvariety of $Y$ defined by the kernel of $\f^*\colon k\{Y\}\to k\{X\}$.
For a $\s$-closed $\s$-subvariety $V$ of $X$, we define $\f(V)$ as $\f_V(V)$, where $\f_V\colon V\to X\xrightarrow{\f}Y$.
Thus $\f(V)$ is the $\s$-closed $\s$-subvariety of $Y$ defined by the kernel of $k\{Y\}\to k\{X\}\to k\{V\}$.

For a morphism $\f\colon X\to Y$ of $\s$-varieties and a $\s$-closed $\s$-variety $Z$ of $Y$, we can define a subfunctor $\f^{-1}(Z)$ of $X$ by $\f^{-1}(Z)(R)=\f_R^{-1}(Z(R))$ for any \ks-algebra $R$. If $Z=\V(\ida)$,
 then
\begin{align*}
\f^{-1}(Z)(R) & =\{\psi\in\Hom(k\{X\},R)|\ \ida\subseteq\ker(\psi\f^*)\} \\
& = \{\psi\in\Hom(k\{X\},R)|\ \f^*(\ida)\subseteq\ker(\psi)\}=\V(\f^*(\ida))(R).
\end{align*}
Therefore $\f^{-1}(Z)=\V(\f^*(\ida))$ is a $\s$-closed $\s$-subvariety of $X$.

Let $k\to K$ be a morphism of $\s$-rings. For a $\s$-variety $X$ over $k$ the functor $X_{K}$ defined by $X_{K}(R')=X(R')$ for every $K$-$\s$-algebra $R'$, is a $\s$-variety over $K$. Indeed, $K\{X_{K}\}=k\{X\}\otimes _k K$. 

Note that for a $\s$-ring $k$ the map $\s\colon k\to k$ is a morphism of $\s$-rings. For an object $X$ over $k$ (e.g., a $\s$-variety or a \ks-algebra) we denote the new object over $k$ obtained by base change via $\s\colon k\to k$ by ${\hs X}$. A similar notation applies for morphisms and higher powers of $\s$.

\medskip

{\bf From now on and throughout the article we assume that $k$ is a $\s$-field.} All schemes and $\s$-varieties are assumed to be over $k$ unless indicated otherwise.

\begin{defi}
	A \emph{$\s$-algebraic group} $G$ over $k$ is a group object in the category of $\s$-varieties over $k$. 
\end{defi}
In particular, $G(R)$ is a group for every \ks-algebra $R$. A list of examples of $\s$-algebraic groups can be found in \cite[Section 2]{Wibmer:FinitenessPropertiesOfAffineDifferenceAlgebraicGroups}. A $\s$-closed $\s$-subvariety $H$ of a $\s$-algebraic group $G$ is a \emph{$\s$-closed} subgroup if $H(R)$ is a subgroup of $G(R)$ for any \ks-algebra $R$.
In symbols, we express this as $H\leq G$. A $\s$-closed subgroup $N$ of a $\s$-algebraic group $G$ is \emph{normal} if $N(R)$ is a normal subgroup of $G(R)$ for every \ks-algebra $R$. We indicate this as $N\unlhd G$. 

A \emph{morphism $\f\colon G\to H$ of $\s$-algebraic groups} is a morphism of $\s$-varieties such that the map $\f_R\colon G(R)\to H(R)$ is a morphism of groups for every \ks-algebra $R$. A morphism of $\s$\=/algebraic groups is a \emph{$\s$-closed embedding} if it is a $\s$-closed embedding of $\s$-varieties.

A \emph{\ks-Hopf algebra} is a \ks-algebra $R$ equipped with the structure of a Hopf-algebra over $k$ such that the Hopf algebra structure maps, (i.e., the comultiplication $\Delta\colon R\to R\otimes_k R$, the counit $\varepsilon\colon R\to k$ and the antipode $S\colon R\to R$) are morphisms of \ks-algebras. A \emph{\ks-Hopf subalgebra} of a \ks-Hopf algebra is a \ks-subalgebra that is a Hopf subalgebra.

The category of $\s$-algebraic groups over $k$ is anti-equivalent to the category of \ks-Hopf algebras that are finitely $\s$-generated over $k$ (\cite[Rem. 2.3]{Wibmer:FinitenessPropertiesOfAffineDifferenceAlgebraicGroups}). A $\s$-closed $\s$-subvariety $H$ of a $\s$-algebraic group $G$ is a $\s$-closed subgroup if and only if $\I(H)\subseteq k\{G\}$ is a Hopf-ideal (\cite[Lemma 2.4]{Wibmer:FinitenessPropertiesOfAffineDifferenceAlgebraicGroups}).
A  $\s$-ideal that is also a Hopf ideal will be called a \emph{$\s$-Hopf ideal}.

 For a $\s$-algebraic group $G$, we denote the kernel of the counit $k\{G\}\to k$ by $\m_G$. Note that $\m_G$ is the $\s$-ideal of $k\{G\}$ that defines the trivial subgroup $1$ of $G$.

For a \ks-algebra $R$ we denote with $R^\sharp$ the $k$-algebra obtained from $R$ by forgetting $\s$. The functor $R\rightsquigarrow R^\sharp$ from the category of \ks-algebras to the category of $k$-algebras has a left adjoint $A\rightsquigarrow [\s]_k A$ (\cite[Lemma 1.7]{Wibmer:FinitenessPropertiesOfAffineDifferenceAlgebraicGroups}).  Explicitly, for a $k$-algebra $A$ the \ks-algebra $[\s]_k A$ is given as follows: For $i\in\nn$ let ${\hsi A}=A\otimes _k k$ denote the $k$-algebra obtained from $A$ by base change via $\s^i\colon k\to k$ and set $A[i]=A\otimes_k {\hs A}\otimes_k\ldots\otimes_k{\hsi A}$. Then $[\s]_k A$ is the union of the $A[i]$'s. 

\begin{lemma}[{\cite[Lemma 1.7]{Wibmer:FinitenessPropertiesOfAffineDifferenceAlgebraicGroups}}] \label{lemma: unicersal property for skA}
	The inclusion $A=A[0]\hookrightarrow [\s]_kA$ satisfies the following universal property: If $R$ is a \ks-algebra and $A\to R$ a morphism of $k$-algebras, then there exists a unique morphism $[\s]_kA\to R$ of \ks-algebras such that 
	$$
	\xymatrix{
A \ar[rr] \ar[rd] & & [\s]_kA \ar@{..>}[ld] \\
& R & 	
	}
$$
commutes.
\end{lemma}

Let $\X$ be an affine scheme of finite type over $k$. Then the functor $R\rightsquigarrow ([\s]_k\X)(R)=X(R^\sharp)$ from the category of \ks-algebras to the category of sets is a \ks-variety. Indeed, $[\s]_k\X$ is represented by $[\s]_kk[\X]$, where $k[\X]$ is the coordinate ring of $\X$, (i.e., $\X=\spec(k[\X])$ respectively $\X=\Hom(k[\X],-)$).
To simplify the notation we will sometimes write $k\{\X\}$ instead of $k\{[\s]_k\X\}=[\s]_kk[\X]$.

\medskip

{\bf Notation for algebraic groups:}
For the purposes of this article, an algebraic group (over $k$) is, by definition, an affine group scheme of finite type (over $k$). In particular, in positive characteristic, an algebraic group need not be reduced. 
For an algebraic group $\G$, we denote with $|\G|$ the dimension of $k[\G]$ as a $k$\=/vector space. (If it is not finite this is simply $\infty$ and we employ the usual rules for calculating with this symbol.)
A \emph{closed subgroup} of an algebraic group is, by definition, a closed subgroup scheme. With $\G_\red$ we denote the underlying reduced scheme of an algebraic group $\G$. (If $k$ is perfect, $\G_\red$ is a closed subgroup of $\G$ by \cite[Cor.~1.39]{Milne:AlgebraicGroupsTheTheoryOfGroupSchemesOfFiniteTypeOverAField}.) The identity component of $\G$ is denoted with $\G^o$. A morphism $\pi\colon\G\to \H$ of algebraic groups or affine group schemes is a \emph{quotient map} if the dual map $\pi^*\colon k[\H]\to k[\G]$ is injective (equivalently faithfully flat). This is the appropriate analog of a surjective morphism of smooth algebraic groups over an algebraically closed field (\cite[Prop. 5.47]{Milne:AlgebraicGroupsTheTheoryOfGroupSchemesOfFiniteTypeOverAField}).
The image $\pi(\G)$ of a morphism $\pi\colon\G\to\H$ of algebraic groups or affine group schemes is the scheme-theoretic image (as in \cite[Def. 1.73]{Milne:AlgebraicGroupsTheTheoryOfGroupSchemesOfFiniteTypeOverAField}).

\medskip

Note that if $\G$ is an algebraic group over $k$, then $[\s]_k\G$ is a $\s$-algebraic group over $k$. A $\s$-closed subgroup of $\G$ is, by definition, a $\s$-closed subgroup of $[\s]_k\G$. If there is no danger of confusion we may sometimes write $\G$ instead of $[\s]_k\G$ also in other places.

  When working with examples of $\s$-algebraic groups we sometimes take the liberty to drop the \ks-algebra $R$ in the notation. For example, we may simply write \mbox{$G=\{g\in \Gm |\ \s(g)^5g^3=1\}$}, instead of cumbersomely saying that $G$ is the $\s$-closed subgroup of the multiplicative group $\Gm$ given by $G(R)=\{g\in R^\times|\ \s(g)^5g^3=1\}$ for any \ks-algebra $R$.

\begin{prop}[{\cite[Prop. 2.16]{Wibmer:FinitenessPropertiesOfAffineDifferenceAlgebraicGroups}}] \label{prop: exists sclosed embedding}
	For every $\s$-algebraic group $G$, there exist exists an algebraic group $\G$ and a $\s$-closed embedding $G\hookrightarrow [\s]_k\G$. In particular, every $\s$-algebraic group is isomorphic to a $\s$-closed subgroup of some general linear group.
\end{prop}

For a $\s$-variety $X$ we denote with $X^\sharp$ the (affine) scheme obtained from $X$ by forgetting $\s$, i.e., $X^\sharp=\spec(k\{X\}^\sharp)$ or, equivalently, $X=\Hom(k\{X\}^\sharp,-)$ as a functor from the category of $k$-algebras to the category of sets. For example, for an algebraic group $\G$, we have $([\s]_k\G)^\sharp=\G\times{\hs\G}\times{^{\s^2}\!\G}\times\ldots$. The following lemma is a geometric reformulation for groups of Lemma \ref{lemma: unicersal property for skA}.
\begin{lemma} \label{lemma: universal property with Hopf}
	Let $\G$ be an algebraic group. The projection $([\s]_k\G)^\sharp\to \G$ onto the first factor satisfies the following universal property: If $G$ is a $\s$-algebraic group and $G^\sharp\to\G$ a morphism of group schemes, then there exists a unique morphism $\f\colon G\to [\s]_k \G$ of $\s$-algebraic groups such that
	$$
	\xymatrix{
	([\s]_k\G)^\sharp  \ar[rr] & & \G  \\
	 & G^\sharp \ar[ru] \ar@{..>}^-{\f^\sharp}[lu] &
}
	$$  
	commutes.
\end{lemma}
\begin{proof}
	This is clear from Lemma \ref{lemma: unicersal property for skA} and \cite[Lemma 2.15]{Wibmer:FinitenessPropertiesOfAffineDifferenceAlgebraicGroups}, where the statement is formulated in terms of Hopf-algebras.
\end{proof}

\begin{defi}
	Let $\X$ be an affine scheme of finite type over $k$ and let $Y$ be a $\s$-closed $\s$\=/subvariety of $[\s]_k\X$. Then $Y$ is defined by a $\s$-ideal $\I(Y)\subseteq k\{\X\}=\cup_{i\in\nn}k[\X][i]$. For $i\in\nn$ the closed subscheme $Y[i]$ of $\X\times {\hs \X}\times\ldots\times {\hsi \X}$ defined by $\I(Y[i])=\I(Y)\cap k[X][i]$ is called the \emph{$i$-th order Zariski closure} of $Y$ in $\X$. The $\s$-variety $Y$ is \emph{Zariski dense} in $\X$ if $Y[0]=\X$.
\end{defi}

We may sometimes also refer to $H[0]$ as the Zariski closure of $Y$ in $\X$.
Note that for a $\s$-closed subgroup $G$ of an algebraic group $\G$, the $i$-th order Zariski closure $G[i]$ of $G$ in $\G$ is a closed subgroup of $\G\times {\hs \G}\times\ldots\times {\hsi \G}$. Moreover, the projections $$\pi_i\colon G[i]\to G[i-1],\ (g_0,\ldots, g_i)\mapsto (g_0,\ldots,g_{i-1})$$ are quotient maps of algebraic groups .

\begin{theo}[{\cite[Theorem 3.7]{Wibmer:FinitenessPropertiesOfAffineDifferenceAlgebraicGroups}}] \label{theo: sdimension}
	Let $G$ be a $\s$-algebraic group, considered as a $\s$-closed subgroup of some algebraic group $\G$ via a $\s$-closed embedding $G\hookrightarrow [\s]_k\G$. For $i\in \nn$ let $G[i]$ denote the $i$-th order Zariski closure of $G$ in $\G$. Then there exist $d,e\in \nn$ such that
	$$\dim(G[i])=d(i+1)+e  \text{ for all sufficiently large } i\in\nn.$$ The integer $d$ does not depend on the choice of $\G$ and the $\s$-closed embedding $G\hookrightarrow [\s]_k\G$. If $d=0$, the integer $e$ does not depend on the choice of $\G$ and the $\s$-closed embedding $G\hookrightarrow [\s]_k\G$. 
\end{theo}

The integer $d=\sdim(G)$ from the above theorem is the \emph{$\s$-dimension} of $G$. If $\sdim(G)=0$, the integer $e=\ord(G)$ is the \emph{order} of $G$. If $\sdim(G)>0$, we set $\ord(G)=\infty$.

\begin{ex} \label{ex: sdim=dim}
	For an algebraic group $\G$ one has $\sdim([\s]_k\G)=\dim(\G)$ (\cite[Example~3.10]{Wibmer:FinitenessPropertiesOfAffineDifferenceAlgebraicGroups}).
\end{ex}

\begin{lemma} \label{lemma: sdim and base change}
	Let $G$ be a $\s$-algebraic group and let $K$ be a $\s$-field extension of $k$. Then $\sdim(G_{K})=\sdim(G)$ and $\ord(G_{K})=\ord(G)$.
\end{lemma}
\begin{proof}
	Since the formation of Zariski closures is compatible with base change, the claim follows from the fact that the dimension of a finitely generated $k$-algebra is invariant under base change.
\end{proof}

\begin{lemma} \label{lemma: sdim and ord for product}
	Let $G$ and $H$ be $\s$-algebraic groups. Then $G\times H$ is a $\s$-algebraic group with
	$\sdim(G\times H)=\sdim(G)+\sdim(H)$ and
	$\ord(G\times H)=\ord(G)+\ord(H).$
\end{lemma}
\begin{proof}
	Let $\G$ and $\H$ be algebraic groups containing $G$ and $H$ respectively as $\s$-closed subgroups. Then $G\times H$ is a $\s$-closed subgroup of $\G\times \H$ and the claim reduces to the similar formula for algebraic groups.
\end{proof}

\begin{prop} \label{prop: limit of Gi}
	Let $G$ be a $\s$-algebraic group, considered as a $\s$-closed subgroup of some algebraic group $\G$ via a $\s$-closed embedding $G\hookrightarrow [\s]_k\G$. For $i\in \nn$ let $G[i]$ denote the $i$-th order Zariski closure of $G$ in $\G$. Set $\G_0=G[0]$ and for $i\geq 1$ let $\G_i$ denote the kernel of $\pi_i\colon G[i]\to G[i-1]$.
	Then, for every $i\geq 1$, there is a closed embedding $\G_i\to{\hs(\G_{i-1})}$, which, for sufficiently large $i$, is an isomorphism. Moreover:
	\begin{enumerate}
		\item The sequence $(\dim(\G_i))_{i\in\nn}$ is non-increasing and stabilizes with value $\sdim(G)$.	
		\item The sequence $(|\G_i|)_{i\in\nn}$ is non-increasing and therefore eventually constant. The eventual value $\ell=\lim_{i\to\infty}|\G_i|$ does not depend on the choice of $\G$ and the $\s$-closed embedding $G\hookrightarrow [\s]_k\G$.
	\end{enumerate}
\end{prop}
\begin{proof}
	The first statement is \cite[Prop. 3.1]{Wibmer:FinitenessPropertiesOfAffineDifferenceAlgebraicGroups}. Point (i) is \cite[Cor. 3.16]{Wibmer:FinitenessPropertiesOfAffineDifferenceAlgebraicGroups}. Point (ii) is \cite[Prop. 5.1]{Wibmer:FinitenessPropertiesOfAffineDifferenceAlgebraicGroups}.
\end{proof}

The value $\ell=\ld(G)$ from the above proposition is the \emph{limit degree} of $G$. Note that $\ld(G)$ is finite if and only if $\sdim(G)=0$ (otherwise $\ld(G)=\infty$). 

The $\s$-closed subgroups of a $\s$-algebraic group satisfy a dimension theorem:
\begin{theo}[{\cite[Theorem 4.6]{Wibmer:FinitenessPropertiesOfAffineDifferenceAlgebraicGroups}}] \label{theo: dimension theore}
	Let $H_1$ and $H_2$ be $\s$-closed subgroups of a $\s$-algebraic group $G$. Then
	$$\sdim(H_1\cap H_2)\geq \sdim(H_1)+\sdim(H_2)-\sdim(G).$$
	
\end{theo}

The following finiteness theorem is the combination of Theorem 4.1 and Corollaries 4.2 and 4.3 in \cite{Wibmer:FinitenessPropertiesOfAffineDifferenceAlgebraicGroups}.

\begin{theo}  \label{theo: first finiteness} Every descending chain of $\s$-closed subgroups of a $\s$-algebraic group is finite. In fact, if $G$ is a $\s$-closed subgroup of a $\s$-algebraic group $H$, then $\I(G)\subseteq k\{H\}$ is finitely $\s$-generated. Moreover, if $H=[\s]_k\G$ for some algebraic group $\G$ and $G[i]$ denotes the $i$-th order Zariski closure of $G$ in $\G$, then there exists an $m\in\nn$ such that $$\I(G)[i]=\big(\I(G)[i-1], \s(\I(G)[i-1])\big)$$
	for all $i>m$.
\end{theo}

There is also second finiteness theorem:
\begin{theo}[{\cite[Theorem 4.5]{Wibmer:FinitenessPropertiesOfAffineDifferenceAlgebraicGroups}}] \label{theo: second finiteness}
	Let $R$ be a \ks-Hopf algebra that is finitely $\s$-generated over $k$ and let $S$ be a \ks-Hopf subalgebra of $R$. Then $S$ is finitely $\s$-generated over $k$.
\end{theo}

\begin{rem}
	To specify the structure of a \ks-algebra on a given $k$-algebra $R$, is equivalent to specifying a morphism ${\hs R}\to R$ of $k$-algebras. Moreover, to specify a morphism $R\to S$ of \ks-algebras is equivalent to specifying a morphism $\psi\colon R\to S$ of $k$-algebras such that
	\begin{equation} \label{eq: algebraic sgroups}
	\xymatrix{
	{\hs R} \ar[r] \ar_{\hs\psi}[d] & R \ar^\psi[d] \\
	{\hs S} \ar[r] & S	
	}
	\end{equation}
	commutes (cf. the proof of \cite[Prop. 5.9]{Wibmer:FinitenessPropertiesOfAffineDifferenceAlgebraicGroups}). Similarly, to specify the structure of a \ks-Hopf algebra on a given $k$-Hopf algebra $R$ is equivalent to specifying a morphism ${\hs R}\to R$ of $k$-Hopf algebras and to specify a morphism $R\to S$ of \ks-Hopf algebras is equivalent to specifying a morphism $\psi\colon R\to S$ of $k$-Hopf algebras such that (\ref{eq: algebraic sgroups}) commutes. By dualizing one obtains the category of (affine) difference group schemes over $k$ (\cite{ChalupnikKowalski:DifferenceModulesAndDifferenceCohomology}, \cite{ChalupnikKowalski:DifferenceSheavesAndTorsors}): An (affine) difference group scheme $G$ over $k$ is an affine group scheme over $k$ together with a morphism $\s_G\colon G\to{\hs G}$ of group schemes over $k$. A morphism between difference group schemes over $k$ is a morphism $\f\colon G\to H$ of group schemes such that
	$$
	\xymatrix{
	G \ar^{\s_G}[r] \ar_{\f}[d] & {\hs G} \ar^{{\hs\f}}[d]  \\
	H \ar^{\s_H}[r] & {\hs H}	
	}
	$$ 
	commutes. The category of difference algebraic groups is equivalent to the full subcategory of the category of difference group schemes consisting of those difference group schemes whose coordinate ring is finitely $\s$-generated over $k$. Thus the relation between difference group schemes and difference algebraic groups is similar to the relation between (affine) group schemes and (affine) group schemes of finite type (i.e., (affine) algebraic groups).

	Some constructions, results and proofs of this article (e.g., the identity component and the isomorphism theorems)  could also be performed in the larger category of difference group schemes. On the other hand, the concepts of $\s$-dimension, strong identity component and almost-simplicity only apply to difference algebraic groups. In particular, for our main results, the ``of finite $\s$\=/type'' assumption is indispensable.
	
	Occasionally, a certain construction or proof might in fact be swifter in the category of difference group schemes, since there one can apply results from the theory of group schemes directly. On the other hand, we think it is useful to make available the difference analogs of proof techniques of the theory of algebraic groups and so we prefer to stick with our formalism for difference algebraic groups throughout. 
	\end{rem}

%
%
%
%
%
%

\section{Subgroups defined by ideal closures} \label{sec: Subgroups defined by ideal closures}

If $\G$ is an algebraic group over a perfect field, then $\G_\red$, the associated reduced scheme, is a closed subgroup of $\G$ (\cite[Cor. 1.39]{Milne:AlgebraicGroupsTheTheoryOfGroupSchemesOfFiniteTypeOverAField}). In difference algebra, there are several closure operations one can define on difference ideals that are in some way similar to taking the radical of an ideal. Therefore, as we detail in this section, one obtains several $\s$-closed subgroups of a $\s$-algebraic group that are in some way analogous to $\G_\red$. 

 The results of this section are relevant for the proof of Theorem A from the introduction because they enable us to show that strongly connected $\s$-algebraic groups have certain desirable properties (see e.g., Lemma \ref{lemma: strongly connected is sintegral}), which in turn is needed for establishing the existence part of Theorem A.

Let us recall the relevant properties of $\s$-ideals (cf. \cite[Section 2.3]{Levin:difference}.)
\begin{defi}
	Let $R$ be a $\s$-ring and $\ida\subseteq R$ a $\s$-ideal. Then $\ida$ is called
	\begin{itemize}
		\item \emph{reflexive} if $\s^{-1}(\ida)=\ida$, i.e., $\s(f)\in\ida $ implies $f\in\ida$,
		\item \emph{mixed} if $fg\in\ida$ implies $f\s(g)\in\ida$,
		\item \emph{perfect} if $\s^{\alpha_1}(f)\cdots\s^{\alpha_n}(f)\in\ida$ implies $f\in\ida$ for $\alpha_1,\ldots,\alpha_n\in\nn$,
		\item \emph{$\s$-prime} if it is reflexive and $\ida$ is a prime ideal.
	\end{itemize}
\end{defi}
Among properties of a $\s$-ideal one has the following implications:
$$
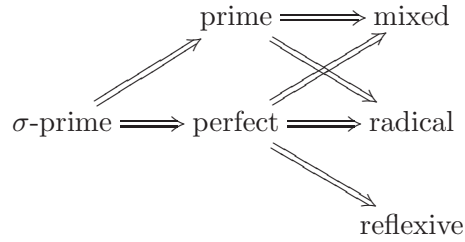
\xymatrix{
& \text{prime} \ar@{=>}[r] \ar@{=>}[rd] & \text{mixed} \\	
\text{$\s$-prime} \ar@{=>}[ru] \ar@{=>}[r]& \text{perfect} \ar@{=>}[r] \ar@{=>}[ru] \ar@{=>}[rd]& \text{radical} \\
& & \text{reflexive}	
}
$$

\begin{defi}
A $\s$-ring whose zero ideal is reflexive / mixed / perfect / $\s$-prime is called \emph{$\s$-reduced} / \emph{well-mixed} \nolinebreak[4] / \emph{perfectly $\s$-reduced} / a \emph{$\s$-domain}.
\end{defi}

Let $\ida$ be a $\s$-ideal of a $\s$-ring $R$. Since the intersection of reflexive / radical mixed \nolinebreak[4] / perfect $\s$-ideals is a reflexive / radical mixed / perfect $\s$-ideal there exists a smallest reflexive / radical mixed / perfect $\s$-ideal of $R$ containing $\ida$. It is called the
\emph{reflexive closure} $\ida^*$/ the \emph{radical mixed closure} $\{\ida\}_\wm$/ the \emph{perfect closure} $\{\ida\}$ of $\ida$. A direct computation shows that
$$\ida^*=\{f\in R \ |\ \exists \ n\in\nn: \ \s^n(f)\in\ida\}.$$
The radical mixed closure and the perfect closure of $\ida$ do not have such a simple elementwise description. Cf. \cite[Section 2.3, p. 121ff]{Levin:difference} and \cite[Lemma 3.1]{Levin:OnTheAscendingChainCondition}.

\begin{defi}
		A $\s$-variety is \emph{reduced} / \emph{$\s$-reduced} / \emph{reduced well-mixed} \nolinebreak[4] / \emph{perfectly $\s$\=/reduced} if its coordinate ring has this property. It is \emph{integral} / \emph{$\s$-integral} if its coordinate ring is an integral domain / $\s$\=/domain. For a $\s$-variety $X$ there exists a unique largest $\s$-closed $\s$\=/subvariety
	$$X_\red\ \ / \ X_{\sred}\ / \ X_\wm\ / \ X_\per$$
	of $X$ that is reduced / $\s$-reduced / reduced well-mixed / perfectly $\s$-reduced. Its defining ideal is the
	radical / reflexive closure / radical mixed closure / perfect closure of the zero ideal of $k\{X\}$.
\end{defi}

We have the following inclusions of $\s$-closed $\s$-subvarieties of $X$:

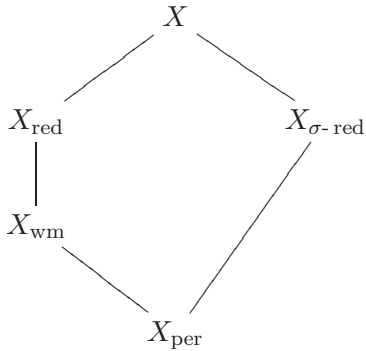
$$
\xymatrix{
	& X  & \\
	X_\red \ar@{-}[ur] & & X_\sred \ar@{-}[ul] \\
	X_\wm \ar@{-}[u]& & \\
	& X_\per \ar@{-}[ul] \ar@{-}[uur] &
}
$$

The importance of perfectly $\s$-reduced $\s$-varieties stems from the fact that they correspond to the classical difference varieties as studied in \cite{Cohn:difference} and \cite{Levin:difference}, where one is only looking for solutions of difference polynomials in $\s$-field extensions of $k$. Mixed $\s$-ideals play a crucial role in the theory of difference schemes as developed by E. Hrushovski in \cite{Hrushovski:elementarytheoryoffrobenius}. Note that for an arbitrary non-empty $\s$-variety $X_\per$ and $X_\wm$ might be empty. Take for example $k\{X\}=k\times k$ with $\s((a,b))=(\s(b),\s(a))$. This pathology does not occur for $\s$-algebraic groups because the kernel $\m_G\subseteq k\{G\}$ of the counit $k\{G\}\to k$ is a $\s$-prime $\s$-ideal.

\begin{ex} \label{ex: smooth connected implies sintegral}
	If $\G$ is a smooth, connected algebraic group, then $[\s]_k\G$ is $\s$-integral and therefore also perfectly $\s$-reduced. However, for a smooth algebraic group $\G$, the $\s$-ring $k\{\G\}$ need not be well-mixed, in particular, $[\s]_k\G$ need not be perfectly $\s$-reduced. 
\end{ex}
\begin{proof}
	For $i\in\nn$ the algebraic groups $\G[i]=\G\times{\hs\G}\times\ldots\times{\hsi\G}$ are smooth and connected. Thus $k[\G[i]]$ is an integral domain and so $k\{\G\}=\bigcup_{i\in\nn}k[\G[i]]$ is also an integral domain. One can check directly from the definition that $\s\colon [\s]_k A\to[\s]_k A$ is injective for any $k$-algebra $A$. From a more geometric perspective, the projection maps $\s_i\colon \G\times\ldots\times{\hsi\G}\to {\hs\G}\times\ldots\times{\hsi\G},\ (g_0,\ldots,g_i)\mapsto (g_1,\ldots,g_i)$ are dominant, so the dual maps are injective.
	
	If $\G$ is not connected, then $[\s]_k$ need not be perfectly $\s$-reduced. For example, consider $\G=\mu_2$, i.e., $\G(A)=\{g\in A^\times|\ g^2=1\}$ for any $k$-algebra $A$. We have $k[\G]=k[y]/(y^2-1)$ and $k\{\G\}=k\{y\}/[y^2-1]$. So $(y-1)(y+1)=0\in k\{\G\}$, however, $(y-1)\s(y+1)=(y-1)(\s(y)+1)$ is not zero in $k\{\G\}$. So $k\{\G\}$ is not well-mixed.
\end{proof}

The following lemma will be needed in Section \ref{sec: components}. It illustrates the general principle that when dealing with perfectly $\s$-reduced $\s$-varieties one can usually restrict to points in $\s$-fields.
 
\begin{lemma} \label{lemma: morphism with perfectly sreduced}
	Let $\f\colon X\to Y$ be a morphism of $\s$-varieties and let $Z\subseteq Y$ be a $\s$-closed $\s$\=/subvariety. Assume that $X$ is perfectly $\s$-reduced. If $\f_K(X(K))\subseteq Z(K)$ for every $\s$-field extension $K$ of $k$, then $\f(X)\subseteq Z$, i.e., $\f$ factors through $Z\hookrightarrow Y$.
\end{lemma}
\begin{proof}
	We have to show that $\I(Z)\subseteq k\{Y\}$ lies in the kernel of $\f^*\colon k\{Y\}\to k\{X\}$. So let $f\in\I(Z)$. We have to show that $\f^*(f)=0$. Since the zero ideal of $k\{X\}$ is perfect, it is the intersection of $\s$-prime $\s$-ideals (\cite[Chapter 3, p. 88]{Cohn:difference}). Therefore, it suffices to show that $\f^*(f)$ lies in every $\s$-prime $\s$-ideal of $k\{X\}$. Let $\p\subseteq k\{X\}$ be a $\s$-prime $\s$-ideal. Then the field of fractions $K$ of $k\{X\}/\p$ naturally is a $\s$-field extension of $k$ and the canonical map $x\colon k\{X\}\to K$ is a morphism of $k$-$\s$-algebras. By assumption, $\f_K(x)\in Z(K)$, i.e., $\I(Z)$ lies in the kernel of $x\circ\f^*$. So $\f^*(f)\in\p$.
\end{proof}

\begin{rem} \label{rem: no field points}
	For a $\s$-algebraic group $G$ the following statements are equivalent:
	\begin{enumerate}
		\item $G(K)=1$ for every $\s$-field extension $K$ of $k$.
		\item $G_\per=1$.
		\item The ideal $\m_G$ is the only $\s$-prime $\s$-ideal of $k\{G\}$.
	\end{enumerate}
\end{rem}
\begin{proof} If $g\in G(K)=\Hom(k\{G\}, K)$, then the kernel of $g$ is a $\s$-prime $\s$-ideal of $k\{G\}$. Conversely, if $\p$ is a $\s$-prime $\s$-ideal of $k\{G\}$, then the field of fractions $K$ of $k\{G\}/\p$ is naturally a $\s$-field and the canonical map $g\colon k\{G\}\to K$ belongs to $G(K)$. Therefore (i) and (iii) are equivalent. The equivalence with (ii) follows from the fact that a perfect $\s$-ideal is the intersection of $\s$-prime $\s$-ideals (\cite[Chapter 3, p. 88]{Cohn:difference}).
\end{proof}

An example of a $\s$-algebraic group satisfying the above three equivalent conditions is the $\s$-closed subgroup $G$ of $\Gl_n$ given by $G(R)=\{g\in\Gl_n(R)|\ \s^d(g)=I_n\}$ for any \ks-algebra $R$. (Here $d\geq 1$ is a fixed integer, $\s$ is applied to $g$ entry-wise and $I_n$ is the $n\times n$-identity matrix.) Another such example, would be $G=[\s]_k\mu_p$ over a $\s$-field of characteristic $p>0$. Here $\mu_p$ is the algebraic group of $p$-th roots of unity, i.e., $\mu_p(A)=\{g\in A^\times|\ g^p=1\}$ for any $k$-algebra $A$.

%
%

To show that for a $\s$-algebraic group $G$ the $\s$-closed $\s$-subvarieties $G_\red\ \ / \ G_{\sred}\ / \ G_\wm\ / \ G_\per$ are $\s$-closed subgroups, we need to know that the corresponding properties are preserved under tensor products:

\begin{lemma} \label{lemma: tensor stays}
	Let $R$ and $S$ be $k$-$\s$-algebras.
	\begin{enumerate}
		\item If $k$ is perfect and $R$ and $S$ are reduced, then $R\otimes_k S$ is reduced.
		\item If $k$ is inversive and $R$ and $S$ are $\s$-reduced, then $R\otimes_k S$ is $\s$-reduced.
		\item If $k$ is algebraically closed and $R$ and $S$ are well-mixed and reduced, then $R\otimes_k S$ is well-mixed and reduced.
		\item If $k$ is inversive and algebraically closed and $R$ and $S$ are perfectly $\s$-reduced, then $R\otimes_k S$ is perfectly $\s$-reduced.
	\end{enumerate}
\end{lemma}
\begin{proof}
	Point (i) is well-known. See e.g., \cite[Theorem 3, Chapter V, \S 15.5, A.V.125]{Bourbaki:Algebra2}. Note that (i) is a special case of (ii) as we may take $\s$ as the Frobenius endomorphism. Point (ii) follows from \cite[Prop. 1.2]{TomasicWibmer:Babbit}.
	
	
	For (iii), note that the zero ideal of a reduced well-mixed $\s$-ring is the intersection of prime $\s$-ideals (\cite[Lemma 2.10]{Hrushovski:elementarytheoryoffrobenius}). If $\p$ is a prime $\s$-ideal of $R$ and $\q$ a prime $\s$-ideal of $S$, then $\p\otimes S+R\otimes \q$ is a prime $\s$-ideal of $R\otimes_k S$ since
	$$(R\otimes_k S)/(\p\otimes S+R\otimes \q)=R/\p\otimes_k S/\q$$
	and the latter is an integral domain, as the tensor product of integral domains over an algebraically closed field is again an integral domain (\cite[Corollary 3, Chapter V, \S 17.5, A.V.143]{Bourbaki:Algebra2}). We see that the zero ideal of $R\otimes_k S$ is the intersection of prime $\s$-ideals of the form $\p\otimes S+R\otimes \q$. This shows that $R\otimes_k S$ is well-mixed and reduced.
	
	To prove (iv) we can proceed as in (iii) by noting that a $\s$-ideal is perfect if and only if it is the intersection of $\s$-prime $\s$-ideals and that the tensor product of $\s$-domains over an inversive algebraically closed $\s$-field is again a $\s$-domain by (ii).
\end{proof}

There are counterexamples showing that the conditions on $k$ in Lemma \ref{lemma: tensor stays} cannot be relaxed. For example, take $k=\mathbb{R}$ with $\s$ the identity map, $R=\mathbb{C}$ with the identity map and $S=\mathbb{C}$ with $\s$ complex conjugation. Then $R$ and $S$ are perfectly $\s$-reduced (hence well-mixed) but $R\otimes_k S$ is not well-mixed (hence not perfectly $\s$-reduced).

If $X$ and $Y$ are $\s$-varieties, then $(X\times Y)_\per$ is contained in $X_\per\times Y_\per\subseteq X\times Y$ but this inclusion might be proper since $X_\per\times Y_\per$ need not be perfectly $\s$-reduced. The situation is similar in the other cases. However, this issue can be circumvented by adding extra assumptions on the base $\s$-field.

\begin{cor} \label{cor: reduced products}
	Let $X$ and $Y$ be $\s$-varieties.
	\begin{enumerate}
		\item If $k$ is perfect, then $(X\times Y)_\red\simeq X_\red\times Y_\red$.
		\item If $k$ is inversive, then $(X\times Y)_\sred\simeq X_\sred\times Y_\sred$.
		\item If $k$ is algebraically closed, then $(X\times Y)_\wm\simeq X_\wm\times Y_\wm$.
		\item If $k$ is inversive and algebraically closed, then $(X\times Y)_\per\simeq X_\per\times Y_\per$.
	\end{enumerate}
\end{cor}
\begin{proof}
	The proof is similar in all four cases.	Exemplarily, let us proof (iv). In terms of \ks\=/algebras, we have to show that the canonical map
	$$k\{X\}/\{0\}\otimes_k k\{Y\}/\{0\}\to (k\{X\}\otimes_k k\{Y\})/\{0\}$$
	is an isomorphism. (Note that here $\{0\}$ denotes the perfect closure of the zero ideal and not the set containing $0$.) As the left hand side is perfectly $\s$-reduced by Lemma \ref{lemma: tensor stays}, we see that $\{0\}=\{0\}\otimes k\{Y\}+k\{X\}\otimes \{0\}$.
\end{proof}

If $\psi\colon R\to S$ is a morphism of $\s$-rings, one can check directly that $\psi^{-1}(\ida)$ is a radical / reflexive / radical mixed / perfect $\s$-ideal if $\ida$ has the corresponding property. This shows that $\psi$ maps the radical / reflexive closure / radical mixed closure / perfect closure of the zero ideal of $R$ into the radical / reflexive closure / radical mixed closure / perfect closure of the zero ideal of $S$.
Therefore, a morphism of $\s$-varieties $X\to Y$ induces a morphism
$$X_\red\to Y_\red \ / \ X_{\sred}\to  Y_{\sred} \ / \ X_\wm\to Y_{\wm} \ / \ X_\per\to Y_\per.$$

\begin{cor} \label{cor: reduced ssubgroups}
	Let $G$ be a $\s$-algebraic group.
	\begin{enumerate}
		\item If $k$ is perfect, then $G_\red$ is a $\s$-closed subgroup of $G$.
		\item If $k$ is inversive, then $G_\sred$ is a $\s$-closed subgroup of $G$.
		\item If $k$ is algebraically closed, then $G_\wm$ is a $\s$-closed subgroup of $G$.
		\item If $k$ is inversive and algebraically closed, then $G_\per$ is a $\s$-closed subgroup of $G$.
	\end{enumerate}
\end{cor}
\begin{proof}
	Again, let us restrict to (iv). The other cases are similar. The multiplication morphism $G\times G\to G$ induces a morphism
	$(G\times G)_\per \to G_\per$. But by Corollary \ref{cor: reduced products}, the $\s$-closed $\s$\=/subvariety $(G\times G)_\per$ of $G\times G$ can be identified with $G_\per\times G_\per\subseteq G\times G$. Therefore, the multiplication maps $G_\per\times G_\per$ into $G_\per$.
	As the inversion \mbox{$G\to G,\ g\mapsto g^{-1}$} also passes to $G_\per$, we see that $G_\per$ is a subgroup of $G$.
\end{proof}
In the following example all the groups $G$, $G_\red$,  $G_\sred$, $G_\wm$ and $G_\per$ are different.

\begin{ex}
	Consider the $\s$-closed subgroup of $\Gm^2$ given as $$G=\left\{(g,h)\in \Gm^2|\  \s^5(g)^2=1,\ h^3=1,\ \s(h)=h^2\right\}$$
	over $k=\overline{\mathbb{F}_2}$, the algebraic closure of the field $\mathbb{F}_2$ with two elements considered as $\s$-field with $\s\colon k\to k$ the identity map. As $\s^5(g)^2-1=(\s^5(g)-1)^2$ over a field of characteristic $2$ we see that 
	$$G_\red=\left\{(g,h)\in \Gm^2|\  \s^5(g)=1,\ h^3=1,\ \s(h)=h^2\right\}.$$
	We have 
	$$G_\sred=\left\{(g,h)\in \Gm^2|\  g^2=1,\ h^3=1,\ \s(h)=h^2\right\}$$
	and we claim that
	$$G_\wm=\left\{(g,h)\in \Gm^2|\  \s^5(g)=1,\ h=1\right\}.$$
	To see the latter, note that if $R$ is a \ks-algebra that is an integral domain and $h\in R^\times$ satisfies $h^3=1$ and $\s(h)=h^2$, then necessarily $h=1$. (This is because the equation $y^3=1$ has only $3$ solutions in an integral domain and they all lie inside $k$. Moreover, $\s$ fixes $k$, so $h=\s(h)=h^2$ and so $h=1$.) The claim then follows from the fact that the radical well-mixed closure of a $\s$-ideal is the intersection of the prime $\s$-ideals it contains (\cite[Lemma 2.10]{Hrushovski:elementarytheoryoffrobenius}).
	Finally, $G_\per=1$, for example, using Remark \ref{rem: no field points}. 

\end{ex}

\begin{ex}
	Let $k$ be a $\s$-field, $\mathsf{G}$ a finite group and $\s\colon \mathsf{G}\to \mathsf{G}$ a group endomorphism. In \cite[Example 2.14]{Wibmer:FinitenessPropertiesOfAffineDifferenceAlgebraicGroups} it is explained how one can associate a $\s$-algebraic group $G$ to this data. There is a one-to-one correspondence between the $\s$-closed subgroups of $G$ and the subgroups of $\GG$ stable under $\s$.

	As $k\{G\}=k^\GG$ is reduced, $G$ is reduced. The $\s$-algebraic group $G$ is $\s$-reduced if and only if $\s\colon\mathsf{G}\to\mathsf{G}$ is an automorphism. Moreover, $G$ is reduced well-mixed if and only if it is perfectly $\s$-reduced if and only if $\s\colon\GG\to\GG$ is the identity map. 
	
	In general, $G_\sred$ corresponds to the $\s$-stable subgroup $\bigcap_{n\in\nn}\s^n(\GG)$ of $\GG$ and $G_\per=G_\wm$ corresponds to the subgroup $\{g\in\GG|\ \s(g)=g\}$ of $\GG$. This follows from the fact that the prime ideals in $k\{G\}=k^\GG$ are in bijection with the elements in $\GG$ and a prime ideal is a $\s$-ideal if only if it is a $\s$-prime $\s$-ideal if and only if the corresponding element of $\GG$ is fixed by $\s$.
\end{ex}


The following example shows that $G_\sred$ need to be a subgroup if $k$ is not inversive.

\begin{ex} \label{ex: Gsred not subgroup}
	Let $k$ be a $\s$-field of characteristic zero which is not inversive. So there exists $\lambda\in k$ with $\lambda\notin \s(k)$. Let $G$ be the $\s$-closed subgroup of the additive group $\Ga$ given by
	$$G(R)=\{g\in R|\ \s^2(g)+\lambda\s(g)=0\}$$
	for any $k$-$\s$-algebra $R$. We will first show that $G$ has no proper, non-trivial $\s$-closed subgroup other than the one defined by the equation $\s(g)=0$. Suppose that $H$ is a proper, non-trivial $\s$-closed subgroup of $G$. By Corollary A.3 in \cite{DiVizioHardouinWibmer:DifferenceAlgebraicRel} every $\s$-closed subgroup of $\Ga$ is of the form $\V(f)$, where $f\in k\{y\}$ is the unique monic linear homogeneous difference polynomial of minimal order in $\I(H)\subseteq k\{\Ga\}=k\{y\}$. As $H$ is non-trivial and properly contained in $G$, $f$ must have order one, i.e., $f=\s(y)+\mu y$ for some $\mu\in k$. But then $\s^2(h)+\s(\mu)\s(h)=0$ and therefore $(\lambda-\s(\mu))\s(h)=0$ for all $h\in H(R)$ for any $k$-$\s$-algebra $R$.
	Because $\lambda\notin\s(k)$ this shows that $\s(h)=0$ for all $h\in H(R)$. Therefore $f=\s(y)$.
	
	Suppose $G_{\sred}$ is a subroup of $G$. By the above, then either $G_{\sred}=G,\ G_{\sred}=1$ or $G_{\sred}=H$, where $H$ is defined by the equation $\s(y)=0$.
	
	Because $\s^n(y)$ does not lie in $[\s^2(y)+\lambda y]$ for $n\in\nn$, the cases $G_{\sred}=1$ and $G_{\sred}=H$ can be excluded. To arrive at a contradiction, it therefore suffices to find a non-zero element in the reflexive closure of the zero ideal of $k\{G\}$.

	Assume that $\lambda^2\in\s(k)$. (For example, we can choose
	$k=\mathbb{C}(\sqrt{x},\sqrt{x+1},\ldots)$ with action of $\s$ determined by $\s(x)=x+1$ and $\lambda=\sqrt{x}$.)
	We have $k\{G\}=k[y,\s(y)]$ and if we choose $\eta\in k$ such that $\s(\eta)=\lambda^2$, then $\s(y)^2-\eta y^2$ lies in the reflexive closure of the zero ideal of $k\{G\}$.
\end{ex}

%
%
%

The following example shows that the $\s$-closed subgroups constructed in Corollary~\ref{cor: reduced ssubgroups} are in general not normal.

\begin{ex} \label{ex: Gsred not normal}
	Let $N$ be the $\s$-closed subgroup of $\Ga$ given by \mbox{$N(R)=\{g\in R|\ \s(g)=0\}$} for any $k$-$\s$-algebra $R$. The $\s$-algebraic group $H=\Gm$ acts on $N$ by group automorphisms
	$$H(R)\times N(R)\to N(R),\ (h,n)\mapsto hn.$$
	So we can form the semidirect product $G=N\rtimes H$ which is the $\s$-variety $N\times H$ with group multiplication given by
	$$(n_1,h_1)\cdot (n_2,h_2)=(n_1+h_1n_2,h_1h_2).$$
	Then $k\{G\}=k\{N\}\otimes_k k\{H\}=k[x]\otimes_k k\{y,y^{-1}\}$ with $\s(x)=0$. The reflexive closure of the zero ideal of $k\{G\}$ is the ideal of $k\{G\}$ generated by $x$. Therefore $G_\sred=H\leq G$. For $h\in H(R)$ and $n\in N(R)$ we have
	$$(n,1)(0,h)(n,1)^{-1}=(n-hn,h),$$ which shows that $G_\sred$ is not normal in $G$.
\end{ex}

In Lemma \ref{lemma: sdim of sssubgroups} we will show that $\sdim(G_\red)$, $\sdim(G_\sred)$, $\sdim(G_\wm)$ and $\sdim(G_\per)$ are all equal to $\sdim(G)$.
The following example shows that the order of $G_\sred$ might be strictly smaller than the order of $G$.

\begin{ex}
	Let $G$ be the $\s$-closed subgroup of $\Ga$ given by $$G(R)=\{g\in R|\ \s^n(g)=0\}$$ for any $k$-$\s$-algebra $R$.
	Then $G$ has order $n$ and $G_\sred$ is the trivial group, which has order $0$.
\end{ex}


\begin{rem} \label{rem: perfectly sreduced and variety}
	Our notion of $\s$-variety (Definition \ref{defi: svariety}) is the difference analog of an affine scheme of finite type over a field in usual algebraic geometry. The affine schemes of finite type over a field that can be recovered from their field-valued points are exactly the reduced ones. The $\s$-varieties that can be recovered from their points in $\s$-fields are exactly the perfectly $\s$\=/reduced ones. Thus, one can argue that perfectly $\s$-reduced $\s$-varieties are the difference analog of reduced affine schemes of finite type over a field. 
	Varieties are commonly assumed to be geometrically reduced, i.e., a reduced affine scheme of finite type over a field is an affine variety if its base change to the algebraic closure is reduced. Therefore, one might argue that the difference analog of an affine variety is a perfectly $\s$-reduced $\s$-variety that remains perfectly $\s$-reduced after base change to an inversive algebraically closed $\s$-field. Or, if we do not mind restrictions on the base field, an affine variety is an affine reduced scheme of finite type over an algebraically closed field and the difference analog of this is a perfectly $\s$-reduced $\s$-variety over an inversive algebraically closed $\s$-field.
\end{rem}

\section{Quotients} \label{sec: Quotients}

In this section we establish the existence of the quotient $G/N$ of a $\s$-algebraic group $G$ modulo a normal $\s$-closed subgroup $N$. Key ingredients for the proof are a result from M. Takeuchi about Hopf-algebras, which is more or less equivalent to the existence of quotients of affine groups schemes (not necessarily of finite type) and our finiteness theorem for \ks\=/Hopf subalgebras  (Theorem \ref{theo: second finiteness}). We also show how to compute $\sdim(G/N)$, $\ord(G/N)$ and $\ld(G/N)$ from the corresponding values for $G$ and $N$.

We do not address the question of the existence of the quotient $G/H$, where $H$ is an arbitrary $\s$-closed subgroup. In this article, we have no need for this more general construction. Moreover, for (affine) algebraic groups, the quotient $G/H$ is in general not affine (but rather quasi-projective). So, addressing this more general question would necessitate the introduction of a more complicated setup for difference algebraic geometry that goes beyond our affine treatment. 

Let $G$ be a $\s$-algebraic group. Recall that a $\s$-closed subgroup $N$ of $G$ is \emph{normal} if $N(R)$ is a normal subgroup of $G(R)$ for any $k$-$\s$-algebra $R$. We write $N\unlhd G$ to express that $N$ is a normal $\s$-closed subgroup of $G$.

If $\f\colon G\to H$ is a morphism of $\s$-algebraic groups, we define the \emph{kernel} $\ker(\f)$ of $\f$ to be the subfunctor of $G$ given by $R\rightsquigarrow\ker(\f_R)$. Then $\ker(\f)$ is a normal $\s$-closed subgroup of $G$. Indeed $\ker(\f)=\f^{-1}(1)$, where $1\leq H$ is the trivial $\s$-closed subgroup of $H$ defined by the kernel $\m_H$ of the counit $k\{H\}\to k$.
Explicitly, we have $\I(\ker(\f))=(\f^*(\m_{H}))\subseteq k\{G\}.$

 The quotient $G/N$ is defined by the following universal property.

\begin{defi}
	Let $G$ be a $\s$-algebraic group and $N\unlhd G$ a normal $\s$-closed subgroup. A morphism of $\s$-algebraic groups $\pi\colon G\to G/N$ such that $N\subseteq\ker(\pi)$ is a \emph{quotient of $G$ mod $N$} if it universal among such maps, i.e., for every morphism of $\s$-algebraic groups $\f\colon G\to H$ with $N\subseteq\ker(\f)$ there exists a unique morphism of $\s$-algebraic groups $\f'\colon G/N\to H$ such that
	\[\xymatrix{
		G \ar^-{\pi}[rr] \ar_-{\f}[rd] & & G/N \ar@{..>}^-{\f'}[ld] \\
		& H &
	}
	\]
	commutes.
\end{defi}

Of course, if a quotient of $G$ mod $N$ exists, it is unique up to a unique isomorphism. We will therefore usually speak of \emph{the} quotient of $G$ mod $N$. Note that for a quotient $\pi\colon G\to G/N$ of $G$ mod $N$ it is a priori not clear that $\ker(\pi)=N$. Allowing ourselves a little abuse of notation we will sometimes also refer to the $\s$-algebraic group $G/N$ as ``the quotient''.

For affine group schemes over a field (not necessarily of finite type), the fundamental theorem on quotients can be formulated in a purely Hopf algebraic manner (\cite{Takeuchi:ACorrespondenceBetweenHopfidealsAndSubHopfalgebras}). Recall that a Hopf ideal $\ida$ in a Hopf algebra $A$ over $k$ is \emph{normal} if, using Sweedler notation,
$$\sum f_{(1)}S(f_{(3)})\otimes f_{(2)}\in A\otimes_k\ida$$ for any $f\in\ida$, where $S$ is the antipode of $A$. Normal Hopf ideals in $A$ correspond to normal closed subgroup schemes (\cite[Lemma 5.1]{Takeuchi:ACorrespondenceBetweenHopfidealsAndSubHopfalgebras}). Similarly, if $G$ is a $\s$-algebraic group, then normal $\s$-Hopf ideals in $k\{G\}$ correspond to normal $\s$-closed subgroups of $G$ (cf. \cite[Lemma~2.4]{Wibmer:FinitenessPropertiesOfAffineDifferenceAlgebraicGroups}).
For a Hopf algebra $A$ over $k$ we denote the kernel of the counit $\varepsilon\colon A\to k$ by $\m_A$.

\begin{theo}[M. Takeuchi] \label{theo: takeuchi}
	Let $A$ be a Hopf algebra over $k$ and $\ida\subseteq A$ a normal Hopf ideal. Then $A(\ida)=\{f\in A|\ \Delta(f)-f\otimes 1\in A\otimes_k\ida\}$ is a Hopf subalgebra of $A$ with $(\m_{A(\ida)})=\ida$. Indeed, $A(\ida)$ is the unique Hopf subalgebra with this property and the largest Hopf subalgebra with the property that $(\m_{A(\ida)})\subseteq\ida$.
\end{theo}
\begin{proof}
	By \cite[Lemma 4.4]{Takeuchi:ACorrespondenceBetweenHopfidealsAndSubHopfalgebras} $A(\ida)$ is a Hopf subalgebra. By \cite[Lemma 4.7]{Takeuchi:ACorrespondenceBetweenHopfidealsAndSubHopfalgebras} it is the largest Hopf subalgebra with  $(\m_{A(\ida)})\subseteq\ida$. Finally, by \cite[Theorem 4.3]{Takeuchi:ACorrespondenceBetweenHopfidealsAndSubHopfalgebras} it is the unique Hopf subalgebra with $(\m_{A(\ida)})=\ida$.
\end{proof}

The existence of the quotient of $G$ mod $N$ can be reduced to Theorem \ref{theo: takeuchi}.
A similar approach was taken in \cite[Section A.9]{DiVizioHardouinWibmer:DifferenceAlgebraicRel}. While the result in \cite{DiVizioHardouinWibmer:DifferenceAlgebraicRel} is formulated in a more general setup (there the $k$-$\s$-Hopf algebras need not be finitely $\s$-generated over $k$) the result we prove here is stronger. Indeed, with the aid of Theorem \ref{theo: second finiteness} we show that $G/N$ is $\s$-algebraic, i.e., $k\{G/N\}$ is finitely $\s$-generated over $k$. This question remained open in \cite{DiVizioHardouinWibmer:DifferenceAlgebraicRel}.

\begin{theo} \label{theo: existence of quotient}
	Let $G$ be a $\s$-algebraic group and $N\unlhd G$ a $\s$-closed subgroup. Then the quotient of $G$ mod $N$ exists. Moreover, a morphism of $\s$-algebraic groups $\pi\colon G\to G/N$ is the quotient of $G$ mod $N$ if and only if $\ker(\pi)=N$ and $\pi^*\colon k\{G/N\}\to k\{G\}$ is injective.
\end{theo}
\begin{proof}
	By Theorem \ref{theo: takeuchi}
	$$k\{G\}(\I(N))=\{f\in k\{G\}|\ \Delta(f)-f\otimes 1\in k\{G\}\otimes_k\I(N)\}$$ is a Hopf subalgebra of $k\{G\}$. Clearly it also is a $k$-$\s$-Hopf subalgebra. From Theorem \ref{theo: second finiteness} we know that $k\{G\}(\I(N))$ is finitely $\s$-generated over $k$. So we can define $G/N$ as the $\s$\=/algebraic group represented by $k\{G\}(\I(N))$, i.e., $k\{G/N\}=k\{G\}(\I(N))$. Let \mbox{$\pi\colon G\to G/N$} be the morphism of $\s$-algebraic groups corresponding to the inclusion $k\{G/N\}\subseteq k\{G\}$ of $k$-$\s$-Hopf algebras.
	
	Let $\f\colon G\to H$ be a morphism of $\s$-algebraic groups such that $N\subseteq\ker(\f)$. As $\ker(\f)=\V(\f^*(\m_{H}))$, the Hopf algebraic meaning of $N\subseteq\ker(\f)$ is $\f^*(\m_{H})\subseteq \I(N)$. To show that $\pi$ has the required universal property, it suffices to show that $\f^*(k\{H\})\subseteq k\{G/N\}$. We know from Theorem \ref{theo: takeuchi} that $k\{G/N\}$ is the largest Hopf subalgebra of $k\{G\}$ such that $\m_{k\{G/N\}}\subseteq\I(N)$. As
	$\m_{\f^*(k\{H\})}=\f^*(\m_{H})\subseteq\I(N)$, we find $\f^*(k\{H\})\subseteq k\{G/N\}$.
	
	Clearly $\pi^*$ is injective. Moreover, $\ker(\pi)=\V(\pi^*(\m_{G/N}))=\V(\I(N))=N$ by Theorem~\ref{theo: takeuchi}.
	
	If $\pi\colon G\to G/N$ is a morphism of $\s$-algebraic groups such that $N=\ker(\pi)$ and $\pi^*\colon k\{G/N\}\to k\{G\}$ is injective, then
	$\pi^*(k\{G/N\})$ is a Hopf subalgebra of $k\{G\}$ such that $(\m_{\pi^*(k\{G/N\})})=\I(N)$. Therefore $\pi^*(k\{G/N\})=k\{G\}(\I(N))$ by Theorem \ref{theo: takeuchi}.
\end{proof}

\begin{cor} \label{cor: quotient embedding}
	Let $\f\colon G\to H$ be a morphism of $\s$-algebraic groups. Then the induced morphism $G/\ker(\f)\to H$ is a $\s$-closed embedding.
\end{cor}
\begin{proof}
	The Hopf subalgebra $\f^*(k\{H\})\subseteq k\{G\}$ satisfies $(\m_{\f^*(k\{H\})})=(\f^*(\m_{H}))=\I(\ker(\f))$.
	Therefore $\f^*(k\{H\})=k\{G\}(\I(\ker(\f)))=k\{G/\ker(\f)\}$ by Theorem \ref{theo: takeuchi}. Consequently the map $k\{H\}\to k\{G/\ker(\f)\}$ is surjective and
	$G/\ker(\f)\to H$ is a $\s$-closed embedding.
\end{proof}

Theorem \ref{theo: existence of quotient} yields a rather practical method for determining the quotient: Given a normal $\s$-closed subgroup $N$ of a $\s$-algebraic group $G$, to determine $G/N$ it suffices to find a morphism $\f\colon G\to H$ with $N=\ker(\f)$ and $\f^*\colon k\{H\}\to k\{G\}$ injective. Let us illustrate this idea with a few examples.

\begin{ex} \label{ex: quotient for sg2}
	Let $G$ be the $\s$-closed subgroup of $\Gm$ given by
	$$G(R)=\{g\in R^\times|\ \s(g)^2=1\}\leq\Gm(R)$$
	and let $N$ be the normal $\s$-closed subgroup of $G$ given by
	$$N(R)=\{g\in R^\times|\ \s(g)=1\}$$
	for any \ks-algebra $R$. We would like to determine the quotient $G/N$. Let $H$ be the $\s$-algebraic group given by
	$H(R)=\{g\in R^\times|\ g^2=1\}$ and let $\f\colon G\to H$ be the morphism given by
	$$\f_R\colon G(R)\to H(R),\ g\mapsto \s(g).$$
	Then $\f$ has kernel $N$ and the dual map $\f^*\colon k\{H\}\to k\{G\}$ is injective. Thus it follows from Theorem \ref{theo: existence of quotient} that $\f$ is the quotient of $G$ mod $N$, i.e., $G/N=H$.
\end{ex}

\begin{ex} \label{ex: quotients for Ga}
	Let $N$ be the $\s$-closed subgroup of the additive group $G=\Ga$ defined by a linear difference equation $\s^n(y)+\lambda_{n-1}\s^{n-1}(y)+\cdots+\lambda_0y=0$. The morphism
	$$\f\colon\Ga\to \Ga,\ g\mapsto \s^n(g)+\lambda_{n-1}\s^{n-1}(g)+\cdots+\lambda_0g$$
	has kernel $N$ and the dual map $\f^*\colon k\{y\}\to k\{y\},\ y\mapsto \s^n(y)+\lambda_{n-1}\s^{n-1}(y)+\cdots+\lambda_0y$ is injective. Therefore, $\f$ is the quotient of $G$ mod $N$, i.e., $G/N=\Ga$.
\end{ex}

\begin{ex} \label{ex: quotients and algebraic groups}
	If $\G$ is an algebraic group with a normal closed subgroup $\N$, then $[\s]_k\N$ is a normal $\s$-closed subgroup of $[\s]_k\G$ and $[\s]_k\G/[\s]_k\N=[\s]_k(\G/\N)$. To verify this, note that the morphism $[\s]_k\pi\colon [\s]_k\G\to [\s]_k(\G/\N)$ induced by $\pi\colon\G\to \G/\N$ has kernel $[\s]_k\N$. Moreover, as $\pi^*\colon k[\G/\N]\to k[\G]$ is injective, it follows that also $[\s]_k(\pi^*)\colon k\{\G/\N\}\to k\{\G\}$ is injective.
\end{ex}

\begin{ex} \label{ex: quotient for finite}
	In \cite[Example 2.14]{Wibmer:FinitenessPropertiesOfAffineDifferenceAlgebraicGroups} it is explained how one can associate a $\s$-algebraic group $G=G(\GG,\s)$ to a finite group $\GG$ equipped with an endomorphism $\s\colon\GG\to \GG$: For a \ks-algebra $R$, the group $G(R)$ consists of all locally constant maps $g\colon \spec(R)\to \GG$ such that
		$$
	\xymatrix{
		\spec(R) \ar_\Sigma[d] \ar^-g[r] & \mathsf{G} \ar^\sigma[d] \\
		\spec(R) \ar^-g[r] & \mathsf{G}		
	}
	$$
	commutes, where $\Sigma(\p)=\s^{-1}(\p)$ (for $\s\colon R\to R$).
	If $\NN$ is a normal subgroup of $\GG$ such that $\s(\NN)\subseteq\NN$, then $N=G(\NN,\s)$ is a normal $\s$-closed subgroup of $G$. As $\s(\NN)\subseteq\NN$ we have an induced endomorphism $\s\colon \GG/\NN\to \GG/\NN$
	and composing $g\colon\spec(R)\to \GG$ with $\GG\to \GG/\NN$ yields a morphism $\pi\colon G\to G(\GG/\NN,\s)$ of $\s$-algebraic groups with kernel $N$. The dual map $\pi^*\colon k^{\GG/\NN}\to k^\GG$ is injective. Thus $\pi$ is the quotient of $G$ mod $N$. In other words, $G(\GG,\s)/G(\NN,\s)=G(\GG/\NN,\s)$. 
\end{ex}

As one may expect, the formation of quotients is compatible with base change:

\begin{lemma} \label{lemma: quotients compatible with base extension}
	Let $N\unlhd G$ be $\s$-algebraic groups and $K$ a $\s$-field extension of $k$. Then $(G/N)_K=G_K/N_K$.
\end{lemma}
\begin{proof}
	It is clear from Theorem \ref{theo: existence of quotient} that the kernel of the morphism $G_K\to (G/N)_K$ obtained from $G\to G/N$ by base change is $N_K$. So by Theorem \ref{theo: existence of quotient} again, it suffices to note that the dual map $k\{G/N\}\otimes_k K\to k\{G\}\otimes_k K$ is injective.
\end{proof}

Now that the existence of the quotient $G/N$ is established, we can start to study its properties. To see how the numerical invariants $\s$-dimension, order and limit degree behave with respect to quotients, we first need to understand how quotients intertwine with Zariski closures.

\begin{lemma} \label{lemma: Zariski closure and normal subgroup}
	Let $\G$ be an algebraic group and let $N\leq G\leq \G$ be $\s$-closed subgroups. For $i\geq 0$ let $G[i]$ and $N[i]$ denote the $i$-th order Zariski closure of $G$ and $N$ in $\G$ respectively. Then $N$ is normal in $G$ if and only if $N[i]$ is normal in $G[i]$ for every $i\geq 0$.
\end{lemma}
\begin{proof}
	As $k\{G\}=\cup_{i\geq 0}k[G[i]]$ is the union of the Hopf subalgebras $k[G[i]]$, we see that $\I(N)$ is a normal Hopf ideal of $k\{G\}$ if and only if $\I(N)\cap k[G[i]]$ is a normal Hopf ideal of $k[G[i]]$ for every $i\geq 0$.
\end{proof}

\begin{prop} \label{prop: Zariskiclosures and quotients}
	Let $\G$ be an algebraic group and $N\unlhd G\leq\G$ $\s$-closed subgroups. For $i\geq 0$ let $G[i]$ and $N[i]$ denote the $i$-th order Zariski closure of $G$ and $N$ in $\G$ respectively. Then there exists an integer $m\geq 0$ such that $G/N$ is a $\s$-closed subgroup of $G[m]/N[m]$ and for $i\geq 0$ the $i$-th order Zariski closure of $G/N$ in $G[m]/N[m]$ is the quotient of $G[i+m]$ mod $N[i+m]$, i.e.,
	$$(G/N)[i]=G[m+i]/N[m+i].$$
\end{prop}
\begin{proof}
	By Theorems \ref{theo: takeuchi} and \ref{theo: existence of quotient} we have
	\begin{align*}
	k\{G/N\} & =\left\{f\in k\{G\}|\ \Delta(f)-f\otimes 1\in k\{G\}\otimes_k\I(N)\right\}\\
	& =\bigcup_{i\geq 0}\{f\in k[G[i]]|\ \Delta(f)-f\otimes 1\in k[G[i]]\otimes_k\I(N[i])\}\\
	& =\bigcup_{i\geq 0} k[G[i]/N[i]].
	\end{align*}
	Moreover, $$k[G[i]/N[i]]\subseteq k[G[i+1]/N[i+1]] \text{ and } \s(k[G[i]/N[i]])\subseteq k[G[i+1]/N[i+1]].$$
	By Theorem \ref{theo: first finiteness}, there exists an integer $m\geq 0$ such that $\I(N[j+1])=\left(\I(N[j]),\s(\I(N[j]))\right)$, i.e.,
	$N[j+1]=(N[j]\times{{}^{\s^j}\!\G})\cap(\G\times{\hs(N[j])})$ for $j\geq m$. 
	We claim that
	\begin{equation}\label{eqn: good quotients}
	k\left[k[G[m]/N[m]],\ldots,\s^i(k[G[m]/N[m]])\right]=k[G[m+i]/N[m+i]] \quad \text{ for } i\geq 0.
	\end{equation}
	The inclusion ``$\subseteq$'' is obvious. To prove the inclusion ``$\supseteq$'' it suffices to show that
	$$\psi_j\colon k[G[j]/N[j]]\otimes_k{\hs(k[G[j]/N[j])}\longrightarrow k[G[j+1]/N[j+1]],\ f_1\otimes(\lambda\otimes f_2)\mapsto f_1\lambda\s(f_2)$$
	is surjective for $j\geq m$. With $$\pi_{j+1}\colon G[j+1]\to G[j],\ (g_0,\ldots,g_{j+1})\mapsto (g_0,\ldots,g_j)$$ and $$\s_{j+1}\colon G[j+1]\to {\hs (G[j])},\ (g_0,\ldots,g_{j+1})\mapsto (g_1,\ldots,g_{j+1}),$$  the morphisms
	$$G[j+1]\xrightarrow{\pi_{j+1}}G[j]\to G[j]/N[j] \text{ and } G[j+1]\xrightarrow{\s_{j+1}}{\hs(G[j])}\to {\hs(G[j]/N[j])}$$ combine to a morphism
	$$G[j+1]\xrightarrow{}(G[j]/N[j])\times {\hs(G[j]/N[j])}$$
	of algebraic groups with kernel $(N[j]\times{{}^{\s^{j}}\!\G})\cap(\G\times{\hs(N[j])})=N[j+1]$.
	Therefore $$G[j+1]/N[j+1]\longrightarrow(G[j]/N[j])\times {\hs(G[j]/N[j])}$$ is a closed embedding and so the dual map is surjective, but the dual map is precisely $\psi_j$. We have thus proved (\ref{eqn: good quotients}). It follows from (\ref{eqn: good quotients}) that $k\{G[m]/N[m]\}\to k\{G/N\}$ is surjective, i.e., $G/N$ is a $\s$-closed subgroup of $G[m]/N[m]$. As the ring to the left hand side of (\ref{eqn: good quotients}) is the coordinate ring of the $i$-th order Zariski closure of $G/N$ in $G[m]/N[m]$, we obtain the required equality of the Zariski closures.
\end{proof}

The following example shows that in general one cannot take $m=0$ in Proposition~\ref{prop: Zariskiclosures and quotients}.
\begin{ex}
	Let $G=\G=\Ga$ and $N\unlhd G$ the $\s$-closed subgroup given by
	$N(R)=\{g\in R|\ \s(g)=0\}$ for any $k$-$\s$-algebra $R$. Then $N[0]=G[0]=\Ga$ and $G[0]/N[0]$ is the trivial group. Therefore $G/N$ cannot be a $\s$-closed subgroup of $G[0]/N[0]$.
\end{ex}

\begin{cor} \label{cor: sdim and ord for quotients}
	Let $G$ be a $\s$-algebraic group and $N\unlhd G$ a normal $\s$-closed subgroup. Then
	\begin{equation} \label{eqn: sdim} \sdim(G)=\sdim(N)+\sdim(G/N)\end{equation} and \begin{equation} \label{eqn: ord} \ord(G)=\ord(N)+\ord(G/N).\end{equation}
\end{cor}
\begin{proof}
	We may assume that $G$ is a $\s$-closed subgroup of some algebraic group $\G$ (Proposition~\ref{prop: exists sclosed embedding}).
	For $i\geq 0$ let $G[i]$ and $N[i]$ denote the $i$-th order Zariski closure of $G$ and $N$ in $\G$ respectively. By Theorem \ref{theo: sdimension} there exist $e_G,e_N\geq 0$ such that $\dim(G[i])=\sdim(G)(i+1)+e_G$ and $\dim(N[i])=\sdim(N)(i+1)+e_N$ for all sufficiently large $i\in\nn$.
	Let $m\geq 0$ be as in Proposition~\ref{prop: Zariskiclosures and quotients} and for $i\geq 0$ let $(G/N)[i]$ denote the $i$-th order Zariski closure of $G/N$ in $G[m]/N[m]$. By Theorem \ref{theo: sdimension} there exist $e_{G/N}\geq 0$ such that $\dim((G/N)[i])=\sdim(G/N)(i+1)+e_{G/N}$. For all sufficiently large $i\in\nn$ we have
	\begin{align*}
	\sdim & (G/N)(i+1)+e_{G/N} =\dim((G/N)[i])=\dim(G[m+i]/N[m+i])= \\
	&=\dim(G[m+i])-\dim(N[m+i])=\\
	&=\sdim(G)(m+i+1)+e_G-\sdim(N)(m+i+1)-e_N=\\
	&=(\sdim(G)-\sdim(N))(i+1)+ (\sdim(G)-\sdim(N))m +e_G-e_N.
	\end{align*}
	This proves (\ref{eqn: sdim}). As $\ord(G)<\infty$ if and only if $\sdim(G)=0$, it follows from (\ref{eqn: sdim}) that (\ref{eqn: ord}) is
	valid if $\sdim(G)>0$. We can therefore assume that $\sdim(G)=0$, and consequently $\sdim(N)=\sdim(G/N)=0$ as well. But then the above formula reduces to
	$\ord(G/N)=e_{G/N}=e_G-e_N=\ord(G)-\ord(N)$.
\end{proof}

Next we will show how to compute $\ld(G/N)$ from $\ld(N)$ and $\ld(G)$. For clarity of the exposition, we single out a lemma on algebraic groups.

\begin{lemma} \label{lemma: kernel for algebraic groups}
	Let $\N_1\unlhd \G_1$ and $\N_2\unlhd \G_2$ be algebraic groups and let $\f\colon \G_2\twoheadrightarrow \G_1$ be a quotient map with kernel $\G$. Assume that the restriction of $\f$ to $\N_2$ has kernel $\N$ and image $\N_1$. Then the kernel of the induced map $\G_2/\N_2\to \G_1/\N_1$ is isomorphic to $\G/\N$.
\end{lemma}
\begin{proof}
	Since $\f$ is a quotient map we may identify $\G_1$ with $\G_2/\G$. Note that the (Noether) isomorphism theorems also hold for algebraic groups. (See e.g., \cite[Chapter 5]{Milne:AlgebraicGroupsTheTheoryOfGroupSchemesOfFiniteTypeOverAField}). We have $\N_1=\N_2/\N=\N_2/\G\cap \N_2=\N_2\G/\G$ and so $\G_1/\N_1=(\G_2/\G)/(\N_2\G/\G)=\G_2/\N_2\G$. This shows that the kernel of
	$\G_2/\N_2\to \G_1/\N_1=\G_2/\N_2\G$ equals $\N_2\G/\N_2=\G/\N_2\cap\G=\G/\N$.
\end{proof}

\begin{cor} \label{cor: multiplicativity of limit degree}
	Let $G$ be a $\s$-algebraic group and $N\unlhd G$ a normal $\s$-closed subgroup.
	Then
	$$\ld(G)=\ld(G/N)\cdot\ld(N).$$
\end{cor}
\begin{proof}
	The limit degree of a $\s$-algebraic groups is finite if and only if the $\s$-dimension is zero.
	So by Corollary \ref{cor: sdim and ord for quotients} the claim is valid if $\sdim(G)>0$. We may thus assume that $\sdim(G)=0$ and therefore $\ld(G),\ \ld(G/N)$ and $\ld(N)$ are all finite. Let $m\geq 0$ be as in Proposition \ref{prop: Zariskiclosures and quotients}. For $i\geq1$ we have commutative diagrams
	$$
	\xymatrix{
		(G/N)[i] \ar@{->>}^{\pi_i}[r] \ar^-\simeq[d]& (G/N)[i-1] \ar^-\simeq[d] \\
		G[m+i]/N[m+i] \ar@{->>}^-{\f_i}[r] & G[m+i-1]/N[m+i-1]
	}
	$$
	where $\f_i$ is induced from the projection $G[m+i]\twoheadrightarrow G[m+i-1]$. For all sufficiently large $i\in\nn$ we have $\ld(G/N)=|\ker(\pi_i)|=|\ker(\f_i)|$.
	Let $\G_{m+i}$ and $\N_{m+i}$ be the kernels of $$G[m+i]\twoheadrightarrow G[m+i-1] \quad \text{ and } \quad N[m+i]\twoheadrightarrow N[m+i-1]$$ respectively.
	It follows from Lemma \ref{lemma: kernel for algebraic groups} that $\ker(\f_i)=\G_{m+i}/\N_{m+i}$. Therefore
	$$\ld(G/N)=|\G_{m+i}/\N_{m+i}|=|\G_{m+i}|/|\N_{m+i}|=\ld(G)/\ld(N).$$
\end{proof}

\section{Morphisms of difference algebraic groups}
\label{sec: Injective and surjective morphisms}

In this section we characterize the morphisms of $\s$-algebraic groups that play a role similar to injective and surjective morphisms in the category of (abstract) groups. These are the $\s$-closed embeddings and the quotient maps. We also show that any morphism of $\s$-algebraic groups factors uniquely as a quotient map followed by a $\s$-closed embedding. Analogous results for algebraic groups can be found in \cite[Chapter 5]{Milne:AlgebraicGroupsTheTheoryOfGroupSchemesOfFiniteTypeOverAField}.

\begin{prop} \label{prop: injective morphism}
	Let $\f\colon G\to H$ be a morphism of $\s$-algebraic groups. Then the following statements are equivalent:
	\begin{enumerate}
		\item The kernel of $\f$ is trivial.
		\item The map $\f_R\colon G(R)\to H(R)$ is injective for every $k$-$\s$-algebra $R$.
		\item The morphism $\f\colon G\to H$ is a $\s$-closed embedding.
		\item The dual map $\f^*\colon k\{H\}\to k\{G\}$ is surjective.
		\item The morphism $\f\colon G\to H$ is a monomorphism in the category of $\s$-algebraic groups, i.e., for every pair $\f_1,\f_2\colon H'\to G$ of morphisms of $\s$-algebraic groups with $\f\f_1=\f\f_2$ we have $\f_1=\f_2$.
	\end{enumerate}
\end{prop}
\begin{proof}
	Clearly (i)$\Leftrightarrow$(ii), (iii)$\Leftrightarrow$(iv), (iii)$\Rightarrow$(ii) and (ii)$\Rightarrow$(v). So it suffices to show that (v) implies (iv). Define $H'=G\times_H G$ by
	$$(G\times_H G)(R)=\{(g_1,g_2)\in G(R)\times G(R)|\ \f(g_1)=\f(g_2)\}$$
	for any $k$-$\s$-algebra $R$. This is a $\s$-closed subgroup of $G\times G$. Indeed, $G\times_H G$ is represented by $k\{G\}\otimes_{k\{H\}}k\{G\}$.
	Let $\f_1$ and $\f_2$ denote the projections onto the first and second coordinate respectively. We have $\f\f_1=\f\f_2$ and so by (iv) we must have $\f_1=\f_2$. This implies that the maps $f\mapsto f\otimes 1$ and $f\mapsto 1\otimes f$ from $k\{G\}\to k\{G\}\otimes_{k\{H\}} k\{G\}$ are equal. As $\f^*(k\{H\})$ is a Hopf subalgebra of $k\{G\}$ we know that $k\{G\}$ is faithfully flat over $\f^*(k\{H\})$ (\cite[Chapter 14]{Waterhouse:IntroductiontoAffineGroupSchemes}). Therefore $f\otimes 1=1\otimes f$ in $k\{G\}\otimes_{k\{H\}} k\{G\}= k\{G\}\otimes_{\f^*(k\{H\})} k\{G\}$ if and only if $f\in \f^*(k\{H\})$ by \cite[Section 13.1, p. 104]{Waterhouse:IntroductiontoAffineGroupSchemes}. Summarily, we find that $\f^*\colon k\{H\}\to k\{G\}$ is surjective.
\end{proof}

We may sometimes write $\f\colon G\hookrightarrow H$ to express that a morphism $\f\colon G\to H$ satisfies the equivalent conditions of Proposition \ref{prop: injective morphism}.

\begin{ex}
	The morphism $\f\colon\Gm\to \Gm$ given by $\f_R(g)=\s(g)$ for any $k$-$\s$-algebra $R$ and $g\in R^\times$ is not a $\s$-closed embedding even though $\f_R$ is injective for every $\s$-field extension $R$ of $k$.
\end{ex}

\begin{ex}
	The morphism $\f\colon \Gm\to\Gm^2$ given by
	$$\f_R\colon \Gm(R)\to\Gm^2(R),\ g\mapsto (g\s(g),\s(g))$$ is a $\s$-closed embedding. 
\end{ex}

\begin{ex}
	If $\G\to\H$ is a closed embedding of algebraic groups, then $[\s]_k\G\to[\s]_k\H$ is a $\s$-closed embedding of $\s$-algebraic groups.
\end{ex}

We next consider morphisms of $\s$-algebraic groups that are analogous to surjective morphisms of (abstract) groups. Note that for a normal $\s$-closed subgroup $N$ of a $\s$-algebraic group $G$, the quotient $\pi\colon G\to G/N$ of $G$ mod $N$ need not be surjective in the blunt sense that $\pi_R\colon G(R)\to (G/N)(R)$ is surjective for every \ks-algebra $R$. Let us illustrate this with an example.

\begin{ex} \label{ex: s not surjective}
	Consider the $\s$-closed subgroup $N=\{g\in \Ga|\ \s(g)=0\}$ of $G=\Ga$. Then $\pi\colon \Ga\to \Ga,\ g\mapsto \s(g)$ is the quotient of $G$ mod $N$ (Example \ref{ex: quotients for Ga}). So $\pi_R\colon (R,+)\to (R,+),\ g\mapsto \s(g)$ is surjective if and only if $\s\colon R\to R$ is surjective (which, depending on $R$, may or may not be the case).
\end{ex}

While the maps $\pi_R\colon G(R)\to (G/N)(R)$ are not surjective on the nose, these maps are in some sense, to be made precise in the following definition, close to being surjective.

\begin{defi}
Let $\psi\colon R\to S$ be a morphism of $k$-$\s$-algebras. Then $\psi$ is \emph{faithfully flat} if the underlying morphism
$\psi^\sharp\colon R^\sharp\to S^\sharp$ of $k$-algebras is faithfully flat. In this case, we also call $S$ is a faithfully flat $R$-$\s$-algebra.

Let $F$ be a functor from the category of \ks-algebras to the category of sets. A subfunctor $D$ of $F$ is \emph{fat} if for every \ks-algebra $R$ and every $g\in F(R)$ there exists a faithfully flat $R$-$\s$-algebra $S$ such that the image of $g$ in $F(S)$ belongs to $D(S)$.
\end{defi}

As we will see in the next section, fat subfunctors are a useful tool for proving the isomorphism theorems for $\s$-algebraic groups.

\begin{ex}
	We continue Example \ref{ex: s not surjective}. While an individual $\pi_R\colon (R,+)\to (R,+),\ g\mapsto \s(g)$ need not be surjective, these maps are close to being surjective in the sense that for every $h\in R$ there exists a faithfully flat $R$-$\s$-algebra $S$ and $g\in S$ such that $\pi_S(g)=h$. For example, we can take $S=R[x]$, a univariate polynomial ring over $R$ with $\s\colon S\to S$ determined by $\s(x)=g$.
\end{ex}

Let $G$ be a $\s$-algebraic group and $N$ a normal $\s$-closed subgroup. By Theorem \ref{theo: existence of quotient} the kernel of $G\to G/N$ equals $N$. We can therefore identify the functor $R\rightsquigarrow G(R)/N(R)$ with a subfunctor of $G/N$. The following lemma provides a useful replacement of the missing surjectivity of the maps $G(R)\to (G/N)(R)$.

%
%
%
%

\begin{lemma} \label{lemma: sheaf surjectivity of quotient}
	Let $G$ be a $\s$-algebraic group and $N\unlhd G$ a $\s$-closed subgroup. Let $R$ be a \ks\=/algebra and $\overline{g}\in (G/N)(R)$. Then there exists a faithfully flat morphism $R\to S$ of $k$-$\s$-algebras and $g\in G(S)$ such that $G(S)\to (G/N)(S)$ maps $g$ to the image of $\overline{g}$ in $(G/N)(S)$.
	
	In other words, the subfunctor $R\rightsquigarrow G(R)/N(R)$ of $G/N$ is a fat subfunctor.
	\end{lemma}
\begin{proof}
	We may use $\overline{g}\in (G/N)(R)=\Hom(k\{G/N\},R)$ to form $S=k\{G\}\otimes_{k\{G/N\}} R$. Since $k\{G\}$ is faithfully flat over $k\{G/N\}$ (\cite[Chapter 14]{Waterhouse:IntroductiontoAffineGroupSchemes}), it follows that \mbox{$R\to S,\ r\mapsto 1\otimes r$} is faithfully flat (\cite[Section 13.3, p. 105]{Waterhouse:IntroductiontoAffineGroupSchemes}). Let us set \mbox{$g\colon k\{G\}\to S,\ f\mapsto f\otimes 1$}. Then
	the maps $k\{G/N\}\xrightarrow{\overline{g}} R\to S$ and $k\{G/N\}\to k\{G\}\xrightarrow{g} S$ are equal. So $g\in G(S)$ has the required property.
\end{proof}

\begin{rem}
	It is possible to understand the quotient $G/N$ as a sheafification of the functor $R\rightsquigarrow G(R)/N(R)$. This is carried out in full detail in \cite[Section 5.1]{Wibmer:Habil}.
\end{rem}

The following proposition characterizes morphisms of $\s$-algebraic groups that are analogous to surjective morphisms of (abstract) groups.

\begin{prop} \label{prop: surjective morphism}
	Let $\f\colon G\to H$ be a morphism of $\s$-algebraic groups. The following statements are equivalent:
	\begin{enumerate}
		\item $\f(G)=H$.
		\item The morphism $\f$ is a quotient, i.e., there exists a normal $\s$-closed subgroup $N$ of $G$ such $\f$ is the quotient of $G$ mod $N$.
		\item The dual map $\f^*\colon k\{H\}\to k\{G\}$ is injective.
		\item For every $k$-$\s$-algebra $R$ and every $h\in H(R)$, there exists a faithfully flat $R$-$\s$-algebra $S$ and $g\in G(S)$ such that the image of $h$ in $H(S)$ equals $\f(g)$, i.e., the subfunctor $R\rightsquigarrow\f_R(G(R))$ of $H$ is fat.
	\end{enumerate}
	\begin{proof}
		As $\f(G)$ is the $\s$-closed $\s$-subvariety of $H$ defined by $\ker(\f^*)$, we see that (i) and (iii) are equivalent. It is clear from Theorem \ref{theo: existence of quotient} that (iii) and (ii) are equivalent. Moreover, (ii) implies (iv) by Lemma \ref{lemma: sheaf surjectivity of quotient}. It thus suffices to show that (iv) implies (iii). Take $R=k\{H\}$ and
		$h=\id_{k\{H\}}\in H(R)=\Hom(k\{H\},k\{H\})$. By (iv) there exists a faithfully flat morphism $\psi\colon k\{H\}\to S$ of $k$-$\s$-algebras and an element $g\in G(S)=\Hom(k\{G\},S)$ such that the image of $h$ in $H(S)=\Hom(k\{H\},S)$ equals $\f(g)=g\f^*$. This means that $\psi=g\f^*$. As any faithfully flat morphism of rings is injective, $\psi$ is injective. Therefore $\f^*$ is injective as well.
	\end{proof}
\end{prop}

\begin{defi}
	A morphism of $\s$-algebraic groups satisfying the equivalent properties of Proposition \ref{prop: surjective morphism} is a \emph{quotient map}. 
\end{defi}
We write $\f\colon G\twoheadrightarrow H$ to indicate that $\f$ is a quotient map.

\begin{ex}
If $\G\to \H$ is a quotient map of algebraic groups, then $[\s]_k\G\to[\s]_k\H$ is a quotient map of $\s$-algebraic groups (as is best seen using point (iii) of Proposition \ref{prop: surjective morphism}).
\end{ex}	

Further examples of quotient maps are in Examples \ref{ex: quotient for sg2}, \ref{ex: quotients for Ga} and \ref{ex: quotient for finite}.

%

\begin{cor} \label{cor: surjective and injective implies isom}
	A morphism of $\s$-algebraic groups that is a $\s$-closed embedding and a quotient map is an isomorphism.
\end{cor}
\begin{proof}
	By Propositions \ref{prop: injective morphism} and \ref{prop: surjective morphism}, such a morphism corresponds to a surjective and injective morphism on the coordinate rings.
\end{proof}

%

\begin{cor} \label{cor: factorisation of morphism}
	Every morphism of $\s$-algebraic groups factors uniquely as a quotient map followed by a $\s$-closed embedding.
\end{cor}
\begin{proof}
	Let $\f\colon G\to H$ be a morphism of $\s$-algebraic groups. The uniqueness in the statement of the corollary means that if $G\twoheadrightarrow H_1\hookrightarrow H$ and $G\twoheadrightarrow H_2\hookrightarrow H$ are two factorizations of $\f$, then there exists an isomorphism $H_1\to H_2$ of $\s$-algebraic groups making
	$$
	\xymatrix{
		G \ar@{=}[d] \ar@{->>}[r] & H_1 \ar^{\simeq}[d] \ar@{^(->}[r] & H \ar@{=}[d] \\
		G \ar@{->>}[r] & H_2 \ar@{^(->}[r] & H
	}
	$$
	commutative. The $k$-$\s$-Hopf subalgebra $\f^*(k\{H\})$ of $k\{G\}$ is finitely \mbox{$\s$-generated} over $k$. So we can define $H_1$ as the $\s$-algebraic group represented by $\f^*(k\{H\})$. The claim of the corollary then follows immediately by dualizing.
\end{proof}
Note that $H_1$ has two interpretations, either as $\f(G)$ or as $G/\ker(\f)$. See Theorem~\ref{theo: isom1} below.

\begin{ex}
	Let $\f\colon\Ga\to \Ga^2$ be the morphism given by
	$$\f_R\colon\Ga(R)\to\Ga^2(R),\ g\mapsto (\s(g),\s^2(g))$$
	for any \ks-algebra $R$. Let us determine the factorization of $\f$ according to Corollary~\ref{cor: factorisation of morphism}. Let $H$ be the $\s$-closed subgroup of $\Ga^2$ given by $$H(R)=\{(g_1,g_2)\in R^2|\ \s(g_1)=g_2\}$$ for any \ks-algebra $R$. Then $H$ is isomorphic to $\Ga$ (via $(g_1,g_2)\mapsto g_1$) and $\f$ maps into $H$. The dual map of $\f\colon \Ga\to H\simeq\Ga$ is given by $k\{y\}\to k\{y\},\ y\mapsto\s(y)$, which is injective. So $\f\colon \Ga\to H$ is a quotient map and $\f\colon \Ga\sar H\iar\Ga^2$ is the searched for factorization of $\f$.
\end{ex}

\section{The isomorphism theorems}
\label{sec: The isomorphism theorems}

In this section we establish the difference analogs of the isomorphism theorems for (abstract) groups. 
These three theorem sometimes also go under the names, homomorphism theorem, isomorphism theorem and correspondence theorem. In any case, these are essential for the proof our Jordan-H\"{o}lder type theorem. Our approach largely follows \cite[Chapter 5]{Milne:AlgebraicGroupsTheTheoryOfGroupSchemesOfFiniteTypeOverAField}.

\begin{lemma} \label{lemma: f(G) subgroup}
	Let $\f\colon G\to H$ be a morphism of $\s$-algebraic groups and let $G_1$ be a $\s$-closed subgroup of $G$. Then $\f(G_1)$ is a $\s$-closed subgroup of $H$.
\end{lemma}
\begin{proof}
	The $\s$-closed $\s$-subvariety $\f(G_1)$ of $H$ is defined by the kernel $\ida$ of $k\{H\}\to k\{G\}\to k\{G_1\}$. Since this is a morphism of \ks-Hopf algebras, it follows that $\ida$ is a $\s$-Hopf ideal. 
%
%
	 So $\f(G)$ is a $\s$-closed subgroup of $H$.
\end{proof}

The following theorem is the difference analog of the first isomorphism theorem for (abstract) groups.
\begin{theo} \label{theo: isom1}
	Let $\f\colon G\to H$ be a morphism of $\s$-algebraic groups. Then $\f(G)$ is a $\s$-closed subgroup of $H$ and the induced morphism $G/\ker(\f)\to\f(G)$ is an isomorphism.
\end{theo}
\begin{proof}
	We already observed in Lemma \ref{lemma: f(G) subgroup} that $\f(G)$ is a $\s$-closed subgroup. Since $\ker(\f)$ is the kernel of $G\to\f(G)$, the induced morphism $G/\ker(\f)\to\f(G)$ is a \mbox{$\s$-closed} embedding  by Corollary \ref{cor: quotient embedding} and so we can identify $G/\ker(\f)$ with a $\s$-closed $\s$-subvariety of $\f(G)$. Since $\f$ factors through $G/\ker(\f)$ it follows from the definition of $\f(G)$ that $G/\ker(\f)=\f(G)$.
\end{proof}

For the proof of the second and third isomorphism theorem we need a little preparation.
The fact that for a quotient map $\pi\colon G\to G/N$ of $\s$-algebraic groups, the maps $\pi_R\colon G(R)\to (G/N)(R)$ need not be surjective, makes it difficult to transfer proofs in the category of (abstract) groups, that would usually be carried out by a diagram chase, to proofs in the category of $\s$-algebraic groups. The following lemmas are useful for overcoming this difficulty. See, for example, the proof of Lemma \ref{lemma: Dedekind law}.

Let $X$ be a $\s$-variety. If $R\to S$ is an injective morphism of $k$-$\s$-algebras (e.g., $S$ is a faithfully flat $R$-$\s$-algebra), then
$$X(R)=\Hom(k\{X\},R)\to\Hom(k\{X\},S)=X(S)$$ is injective. To simplify the notation, we will, in the sequel, often identify $X(R)$ with its image in $X(S)$.

\begin{lemma} \label{lemma: faithfullyflat intersection}
	Let $Y$ be a $\s$-closed $\s$-subvariety of a $\s$-variety $X$ and let $R\to S$ be an injective morphism of $k$-$\s$-algebras (e.g., $R\to S$ is faithfully flat). Then
	$$
	Y(R)=X(R)\cap Y(S),
	$$
	where, using the above described identification, the intersection is understood to take place in $X(S)$.
\end{lemma}
\begin{proof}
	The inclusion ``$\subseteq$'' is obvious. To prove ``$\supseteq$'' it suffices to note that for
	a morphism $k\{X\}\to S$ with factorizations $k\{X\}\twoheadrightarrow k\{Y\}\to S$ and $k\{X\}\to R\hookrightarrow S$, one has an arrow $k\{Y\}\to R$ such that
	$$\xymatrix{ & k\{Y\} \ar@{..>}[dd] \ar[rd]  & \\
		k\{X\} \ar@{->>}[ur] \ar[rd] & & S \\
		& R \ar@{^(->}[ur] &
	}
	$$
	commutes.
\end{proof}

The following lemma provides an explicit description of the points of $\f(G)$.

\begin{lemma} \label{lemma: f(G) faithfully flat}
	Let $\f\colon G\to H$ be a morphism of $\s$-algebraic groups and $G_1\leq G$ a $\s$-closed subgroup. Let $R$ be a $k$-$\s$-algebra. Then
	$\f(G_1)(R)$ equals the set of all $h\in H(R)$ such that there exists a faithfully flat $R$-$\s$-algebra $S$ and $g_1\in G_1(S)$ with $\f(g_1)=h$.
\end{lemma}
\begin{proof}
	The induced morphism $G_1\to\f(G_1)$ is a quotient map. So it follows from Proposition~\ref{prop: surjective morphism}~(iv) that for $h\in\f(G_1)(R)$ there exists a faithfully flat $R$-$\s$-algebra $S$ and $g_1\in G_1(S)$ with $\f(g_1)=h$.
	
	Conversely, if $h=\f(g_1)$, then $h\in\f(G_1(S))\subseteq\f(G_1)(S)$ and we can deduce from
	$\f(G_1)(S)\cap H(R)=\f(G_1)(R)$ (Lemma \ref{lemma: faithfullyflat intersection}) that $h\in\f(G_1)(R)$.
\end{proof}

The following lemma provides one of the reasons why fat subfunctors are useful for us. 

\begin{lemma}
	Let $X$ and $Y$ be $\s$-varieties and let $D$ be a fat subfunctor of $X$. Then any morphism $D\to Y$ (of functors) extends uniquely to a morphism $X\to Y$.
\end{lemma}
\begin{proof}
	The key property of faithfully flat algebras we need is the following: If $S$ is a faithfully flat $R$-algebra then the sequence $R\to S\rightrightarrows S\otimes_R S$ is exact, i.e, if $s\in S$ is such that $s\otimes 1=1\otimes s\in S\otimes_R S$, then there exists a unique $r\in R$ mapping to $s$ under $R\to S$. See e.g., \cite[Chapter I, \S 1, Lemma 2.7]{DemazureGabriel:GroupesAlgebriques} (with $M=C$). It follows that for any \ks-algebra $R$ and faithfully flat $R$-$\s$-algebra $S$, the sequence $X(R)\to X(S)\rightrightarrows X(S\otimes_R S)$ is exact and similarly for $Y$ in place of $X$.
	
	Let us first show the uniqueness of an extension $\widetilde{\f}\colon X\to Y$ of $\f\colon D\to Y$. Let $R$ be a \ks\=/algebra and $x\in X(R)$. Since $D\subseteq X$ is fat, there exists a faithfully flat $R$-$\s$-algebra $S$ such that the image $d$ of $x$ in $X(S)$ lies in $D(S)$.
	The commutative diagram
	$$
	\xymatrix{
	X(R) \ar[r] & X(S) \ar@<2pt>[r]
	\ar@<-2pt>[r] & X(S\otimes_R S) \\
	D(R) \ar[r] \ar[d] \ar@{^(->}[u] & D(S) \ar[d] \ar@{^(->}[u] \ar@<2pt>[r]
	\ar@<-2pt>[r] & D(S\otimes_ R S) \ar[d] \ar@{^(->}[u] \\
	 Y(R) \ar[r] & Y(S) \ar@<2pt>[r]
	 \ar@<-2pt>[r] & Y(S\otimes_R S)
}
$$
with exact top and bottom rows, shows that $\widetilde{\f}_R(x)$ is the unique element of $Y(R)$ that maps to $\f_S(d)\in Y(S)$. 

To establish the existence, we need to show that $\widetilde{\f}$, when constructed as above, is well-defined, i.e., if $S_1$ and $S_2$ are faithfully flat $R$-$\s$-algebras and $d_1\in D(S_1)$ and $d_2\in D(S_2)$ are the image of $x\in X(R)$ in $X(S_1)$ and $X(S_2)$ respectively, then $\f_{S_1}(d_1)\in Y(S_1)$ and $\f_{S_2}(d_2)\in Y(S_2)$ are both the image of the same $y\in Y(R)$. 

The images of $d_1$ and $d_2$ in $D(S_1\otimes_R S_2)$ agree, since they both are the image of $x\in X(R)$. Let us denote this common image by $d\in D(S_1\otimes_R S_2)$. Then $\f_{S_1}(d_1)\in Y(S_1)$ and $\f_{S_2}(d_2)\in Y(S_2)$ have the same image, namely $\f_{S_1\otimes_R S_2}(d)$, in $Y(S_1\otimes_R S_2)$.
Thus, identifying elements along the inclusions 
$$
\xymatrix{
& Y(S_1\otimes_R S_2) & \\
Y(S_1) \ar@{^(->}[ru] & & Y(S_2)  \ar@{_(->}[lu] \\
& Y(R)  \ar@{^(->}[ru]  \ar@{_(->}[lu]&  
}
$$
we see that $\f_{S_1}(d_1)$, $\f_{S_2}(d_2)$ and $\f_{S_1\otimes_R S_2}(d)$ all agree with the same element $y$ of $Y(R)$. 

So we have for every \ks-algebra $R$ a map $\widetilde{\f}_R\colon X(R)\to Y(R)$ that extends $\f_R\colon D(R)\to Y(R)$. These $\widetilde{\f}_R$ form a morphism $\widetilde{\f}\colon X\to Y$ of functors extending $\f\colon D\to Y$.	
\end{proof}

As an immediate consequence of this lemma we obtain:
\begin{cor} \label{cor: extends from fat}
	Let $D$ and $D'$ be fat subfunctors of the $\s$-varieties $X$ and $X'$ respectively. Then any isomorphism $D\to D'$ uniquely extends to a morphism $X\to X'$ and this morphism is an isomorphism. \qed
\end{cor}

The next lemma is the $\s$-analog of the basic fact that surjective morphism of (abstract) groups preserve normal subgroups.

\begin{lemma} \label{lemma: normal stays normal under surjective map}
	Let $\f\colon G\to H$ be a quotient map of $\s$-algebraic groups. If $N$ is a normal $\s$-closed subgroup of $G$, then $\f(N)$ is a normal $\s$-closed subgroup of $H$.
\end{lemma}
\begin{proof}
	We already know from Lemma \ref{lemma: f(G) subgroup} that $\f(N)$ is a $\s$-closed subgroup of $H$.
	Let $R$ be a $k$-$\s$-algebra, $h\in\f(N)(R)$ and $h_1\in H(R)$. We have to show that $h_1hh_1^{-1}\in\f(N)(R)$. By Proposition \ref{prop: surjective morphism}, there exists a faithfully flat $R$-$\s$-algebra $S$ and \mbox{$g\in N(S)$} with $\f(g)=h$. Similarly, there exists a faithfully flat $R$-$\s$-algebra $S_1$ and $g_1\in G(S_1)$ with $\f(g_1)=h_1$. Then $S'=S\otimes_R S_1$ is a faithfully flat $R$-$\s$-algebra (\cite[Section 13.3, p. 106]{Waterhouse:IntroductiontoAffineGroupSchemes}) and we can consider $G(S)$ and $G(S_1)$ as subgroups of $G(S')$. Since $N(S')\unlhd G(S')$ we see that $g_1gg_1^{-1}\in N(S')$. Therefore
	$$h_1hh_1^{-1}=\f(g_1gg_1^{-1})\in\f(N(S'))\subseteq\f(N)(S').$$ As $\f(N)(S')\cap H(R)=\f(N)(R)$ by Lemma \ref{lemma: faithfullyflat intersection}, this shows that $h_1hh_1^{-1}\in\f(N)(R)$.
\end{proof}

Let $N$ and $H$ be $\s$-closed subgroups of a $\s$-algebraic group $G$ such that $H$ normalizes $N$, i.e., $H(R)$ normalizes $N(R)$ for any $k$-$\s$-algebra $R$. Then we can form the semidirect product $N\rtimes H$: This is a $\s$-algebraic group with underlying $\s$-variety $N\times H$ and multiplication given by
$((n_1,h_1),(n_2,h_2))\mapsto(n_1 h_1n_2h_1^{-1}, h_1 h_2)$
for any $k$-$\s$-algebra $R$ and $n_1,n_2\in N(R)$, $h_1,h_2\in H(R)$.
The map
$$m\colon N\rtimes H\to G,\ (n,h)\mapsto nh$$
for any $k$-$\s$-algebra $R$ and $n\in N(R)$, $h\in H(R)$ is a morphism of $\s$-algebraic groups.
We define
$$HN:=NH:=m(N\rtimes H).$$
Then $HN$ is a $\s$-closed subgroup of $G$ (Lemma \ref{lemma: f(G) subgroup}). In fact, $HN$ is the smallest $\s$-closed subgroup of $G$ that contains $N$ and $H$.
Since $N\rtimes H\to HN$ is a quotient map, it follows from Proposition \ref{prop: surjective morphism} that the functor $R\rightsquigarrow N(R)H(R)=H(R)N(R)$ is a fat subfuntor of $HN$. 
Moreover, by Lemma~\ref{lemma: f(G) faithfully flat} we have
$$(HN)(R)=\{g\in G(R)|\ \exists \text{ faithfully flat $R$-$\s$-algebra $S$ such that } g\in N(S)H(S)=H(S)N(S)\}$$
for any $k$-$\s$-algebra $R$. Since $N$ is normal in $N\rtimes H$ we know from Lemma \ref{lemma: normal stays normal under surjective map} that $N=m(N)$ is normal in $HN$.

The following theorem is the analog of the second isomorphism theorem for groups.

\begin{theo} \label{theo: isom2}
	Let $H$ and $N$ be $\s$-closed subgroups of a $\s$-algebraic group $G$ such that $H$ normalizes $N$. Then the canonical morphism
	$$H/(H\cap N)\to HN/N$$
	is an isomorphism.
\end{theo}
\begin{proof}
	By Lemma \ref{lemma: sheaf surjectivity of quotient} the functor $R\rightsquigarrow H(R)/H(R)\cap N(R)$ is a fat subfunctor of $H/(H\cap N)$. Similarly, since $R\rightsquigarrow (HN)(R)/N(R)$ is a fat subfunctor of $HN/N$ and $R\rightsquigarrow H(R)N(R)$ is a fat subfunctor of $HN$, it follows that $R\rightsquigarrow  H(R)N(R)/N(R)$ is a fat subfunctor of $HN/N$.

	For every $k$-$\s$-algebra $R$ we have an isomorphism
	\begin{equation} \label{eq: isom on fat}
	H(R)/(H(R)\cap N(R))\to H(R)N(R)/N(R).
	\end{equation}
	 So the canonical morphism $H/(H\cap N)\to HN/N$ restricts to an isomorphism between the fat subfunctors on the left and right hand side of (\ref{eq: isom on fat}). Corollary \ref{cor: extends from fat} implies that the canonical morphism must be an isomorphism itself.
\end{proof}


The following theorem is the $\s$-analog of the third isomorphism theorem for (abstract) groups.

\begin{theo} \label{theo: isom3}
	Let $G$ be a $\s$-algebraic group, $N\unlhd G$ a normal $\s$-closed subgroup and $\pi\colon G\to G/N$ the quotient. The map $H\mapsto \pi(H)=H/N$ defines a bijection between the $\s$-closed subgroups $H$ of $G$ containing $N$ and the $\s$-closed subgroups $H'$ of $G/N$. The inverse map is $H'\mapsto\pi^{-1}(H')$. A $\s$-closed subgroup $H$ of $G$ containing $N$ is normal in $G$ if and only if $H/N$ is normal in $G/N$, in which case the canonical morphism
	$$G/H\to (G/N)/(H/N)$$
	is an isomorphism.
\end{theo}
\begin{proof}
	Theorem \ref{theo: isom1} applied to $H\to G/N$ shows that $H/N=\pi(H)$.
	Let us show that $\pi^{-1}(\pi(H))=H$ for every $\s$-closed subgroup $H$ of $G$ containing $N$.
	Let $R$ be a $k$-$\s$-algebra and $g\in \pi^{-1}(\pi(H))(R)$, i.e., $\pi(g)\in\pi(H)(R)$. By Lemma \ref{lemma: f(G) faithfully flat} there exists a faithfully flat $R$\=/$\s$\=/algebra $S$ and $h\in H(S)$ with $\pi(h)=\pi(g)\in (G/N)(S)$. As $\ker(\pi)=N$ by Theorem \ref{theo: existence of quotient}, this implies that $gh^{-1}\in N(S)\leq H(S)$. Therefore $g\in H(S)$ and $g\in H(S)\cap G(R)=H(R)$ by Lemma \ref{lemma: faithfullyflat intersection}. Thus $\pi^{-1}(\pi(H))\subseteq H$. The reverse inclusion is obvious.
	
	Let us next show that $\pi(\pi^{-1}(H'))=H'$ for a $\s$-closed subgroup $H'$ of $G/N$. As $\pi$ maps $\pi^{-1}(H')$ into $H'$, it is clear form the definition of $\pi(\pi^{-1}(H'))$ that $\pi(\pi^{-1}(H'))\subseteq H'$.
	
	Conversely, let $R$ be a $k$-$\s$-algebra and $h'\in H'(R)$. There exists a faithfully flat $R$-$\s$-algebra $S$ and $g\in G(S)$ such that $\pi(g)=h'$.
	So $g\in\pi^{-1}(H')(S)$ and $h'=\pi(g)\in\pi(\pi^{-1}(H')(S))\subseteq\pi(\pi^{-1}(H'))(S)$. Thus
	$h'\in \pi(\pi^{-1}(H'))(S)\cap H'(R)=\pi(\pi^{-1}(H'))(R)$. Hence $\pi(\pi^{-1}(H'))=H'$.
	
	If $H$ is normal in $G$, then $\pi(H)$ is normal in $G/N$ by Lemma \ref{lemma: normal stays normal under surjective map}. Clearly $\pi^{-1}(H')$ is normal if $H'$ is normal.

Note that $R\rightsquigarrow (G/N)(R)/(H/N)(R)$ is a fat subfunctor of $(G/N)/(H/N)$ by Lemma~\ref{lemma: sheaf surjectivity of quotient}. Moreover, for any \ks-algebra $R$, the map $G(R)/N(R)\to (G/N)(R)/(H/N)(R)$ has kernel $H(R)/N(R)$ (as $\pi^{-1}(\pi(H))=H$). So, using Lemma \ref{lemma: sheaf surjectivity of quotient} again, we see that the functor $R\rightsquigarrow(G(R)/N(R))/(H(R)/N(R))$ is a fat subfunctor of $R\rightsquigarrow (G/N)(R)/(H/N)(R)$. It follows that $R\rightsquigarrow (G(R)/N(R))/(H(R)/N(R))$ is a fat subfunctor of $(G/N)/(H/N)$.
	For every $k$-$\s$-algebra $R$ we have an isomorphism
	$$G(R)/H(R)\to (G(R)/N(R))/(H(R)/N(R)).$$
	In other words, the canonical morphism $G/H\to (G/N)/(H/N)$ restricts to an isomorphism between the fat subfunctors
	$R\rightsquigarrow G(R)/H(R)$ and $R\rightsquigarrow (G(R)/N(R))/(H(R)/N(R))$ of $G/N$ and $(G/N)/(H/N)$ respectively. By Corollary \ref{cor: extends from fat} the canonical morphism must be an isomorphism itself.
\end{proof}

\section{Components} \label{sec: components}

In this section we study the \emph{identity component} $G^o$ and the \emph{strong identity component} $G^{so}$ of a $\s$-algebraic group $G$. The identity component $G^o$ is the analog of the usual identity component $\G^o$ of an algebraic group $\G$. Indeed, the underlying group scheme $(G^o)^\sharp$ of the identity component $G^o$ of $G$ is the identity component $(G^\sharp)^o$ of the underlying group scheme $G^\sharp$ of $G$.

Recall that the strong identity component $\G^{so}$ of an algebraic group $\G$ can be defined as the smallest closed subgroup with the same dimension as $\G$ (see \cite[Def. 6.9 and Prop.~6.10]{Milne:AlgebraicGroupsTheTheoryOfGroupSchemesOfFiniteTypeOverAField}). The strong identity component $G^{so}$ of a $\s$-algebraic group $G$ is defined similarly: it is the smallest $\s$-closed subgroup with the same $\s$-dimension as $G$. The strong identity component and the related notion of being strongly connected are fundamental for establishing our Jordan-H\"{o}lder type theorem for $\s$-algebraic groups.

\subsection{The identity component}

Rather than defining the identity component $G^o$ of a $\s$\=/algebraic group directly, it turns out to be more convenient to first define the quotient $G/G^o$ through a universal property. Our approach is analogous to the approach taken in \cite[Chapter 6]{Waterhouse:IntroductiontoAffineGroupSchemes}. We begin by recalling some definitions and results from \cite[Section 6]{Wibmer:FinitenessPropertiesOfAffineDifferenceAlgebraicGroups}.

\begin{defi}
	A finitely $\s$-generated \ks-algebra $R$ is \emph{$\s$-\'{e}tale} (over $k$) if $R$ is integral over $k$ and separable as a $k$-algebra. A $\s$-algebraic group $G$ is $\s$-\'{e}tale if $k\{G\}$ is a $\s$-\'{e}tale \ks-algebra.
\end{defi}

Recall that a $k$-algebra $A$ is \'{e}tale if $A\otimes_k\overline{k}$ is isomorphic (as a $\overline{k}$-algebra) to a finite direct product of copies of $\overline{k}$, where $\overline{k}$ denotes the algebraic closure of $k$.
Other equivalent ways to express that a finitely $\s$-generated \ks-algebra $R$ is $\s$-\'{e}tale are:
\begin{itemize}
	\item Every $r\in R$ satisfies a separable polynomial over $k$.
	\item The $k$-algebra $R$ is a union of \'{e}tale $k$-algebras.
\end{itemize}

\begin{ex}
	If $\G$ is an \'{e}tale algebraic group, then $[\s]_k\G$ is a $\s$-\'{e}tale $\s$-algebraic group. The $\s$-algebraic group from \cite[Example 2.14]{Wibmer:FinitenessPropertiesOfAffineDifferenceAlgebraicGroups} is also $\s$-\'{e}tale.
\end{ex}

For a $k$-algebra $A$, we let $\pi_0(A)$ denote the union of all \'{e}tale $k$-subalgebras of $A$.
That is, $\pi_0(A)$ consists of all elements of $A$ that annul a separable polynomial over $k$. Then $\pi_0(A)$ is a $k$-subalgebra of $A$. (Cf. Section~6.7 in \cite{Waterhouse:IntroductiontoAffineGroupSchemes}.) Clearly, a $\s$-algebraic group $G$ is $\s$-\'{e}tale if and only if $\pi_0(k\{G\})=k\{G\}$.
\begin{lemma} \label{lemma: pi0 for tensor}
	Let $A$ and $B$ be $k$-algebras. Then $\pi_0(A\otimes_k B)=\pi_0(A)\otimes_k \pi_0(B)$.
\end{lemma}
\begin{proof}
	We may assume that $A$ and $B$ are finitely generated as $k$-algebras. In this case the statement is proved in \cite[Section 6.7, p. 50]{Waterhouse:IntroductiontoAffineGroupSchemes}.
\end{proof}

If $R$ is a \ks-algebra, one can show that $\pi_0(R)$ is a \ks-subalgebra. Moreover, for a $\s$\=/algebraic group $G$, one can use Lemma \ref{lemma: pi0 for tensor}, to show that $\pi_0(k\{G\})$ is a \ks-Hopf subalgebra of $k\{G\}$, which, by Theorem \ref{theo: second finiteness}, is finitely $\s$-generated and therefore represents a $\s$\=/\'{e}tale $\s$\=/algebraic group $\pi_0(G)$. The quotient map $G\to \pi_0(G)$ corresponding to the inclusion $k\{\pi_0(G)\}=\pi_0(k\{G\})\subseteq k\{G\}$ of \ks-Hopf algebras satisfies a universal property detailed in the following proposition. See \cite[Prop. 6.13]{Wibmer:FinitenessPropertiesOfAffineDifferenceAlgebraicGroups}.

\begin{prop} \label{prop: pi0}
	Let $G$ be a $\s$-algebraic group. There exists a $\s$-\'{e}tale $\s$-algebraic group $\pi_0(G)$ and a morphism $G\to \pi_0(G)$ of $\s$-algebraic groups satisfying the following universal property: If $G\to H$ is a morphism of $\s$-algebraic groups with $H$ $\s$-\'{e}tale, then there exists a unique morphism $\pi_0(G)\to H$ such that
	\[\xymatrix{
		G \ar[rr] \ar[rd] & & \pi_0(G) \ar@{..>}[ld] \\
		& H &
	}
	\]
	commutes. 
\end{prop}

Of course $\pi_0(G)$ is uniquely characterized by the above universal property. 

\begin{defi}
	Let $G$ be a $\s$-algebraic group. The $\s$-\'{e}tale $\s$-algebraic group $\pi_0(G)$ defined by the universal property in Proposition \ref{prop: pi0} is the \emph{group of connected components} of $G$. The kernel $G^o$ of $G\twoheadrightarrow\pi_0(G)$ is the \emph{identity component} of $G$.
\end{defi}
So $G/G^o=\pi_0(G)$. For an ideal $\ida$ of a ring $R$ we denote with $\vv(\ida)$ the closed subset of $\spec(R)$ consisting of all prime ideals of $R$ that contain $\ida$. The following lemma combines Lemmas 6.7 and 6.15 from \cite{Wibmer:FinitenessPropertiesOfAffineDifferenceAlgebraicGroups}.

\begin{lemma} \label{lemma: summarize connected}
	Let $G$ be a $\s$-algebraic group. Then:
	\begin{enumerate}
		\item The connected components and the irreducible components of $\spec(k\{G\})$ coincide. 
		\item For a prime ideal $\p$ of $k\{G\}$, the connected component of $\spec(k\{G\})$ containing $\p$ equals $\vv(\ida)$, where $\ida$ is the ideal generated by all idempotent elements of $k\{G\}$ contained in $\p$.
		\item The connected components of $\spec(k\{G\})$ are in bijection with the connected components of $\spec(k\{\pi_0(G)\})$.
		\item Every connected component of $\spec(k\{\pi_0(G)\})$ consists of a single point.
	\end{enumerate}
\end{lemma}

The following lemma characterizes connected $\s$-algebraic groups.

\begin{lemma} \label{lemma: characterize connected}
	The following four conditions on a $\s$-algebraic group $G$ are equivalent:
	\begin{enumerate}
		\item $G^o=G$.
		\item $\pi_0(G)=1$.
		\item  $\spec(k\{G\})$ is connected.
		\item The nilradical of $k\{G\}$ is a prime ideal.
	\end{enumerate}
\end{lemma}
\begin{proof}
	Clearly, (i)$\Leftrightarrow$(ii). We have (ii)$\Leftrightarrow$(iii) by Lemma \ref{lemma: summarize connected} (iii) and (iv). 
	Point (iv) is equivalent to $\spec(k\{G\})$ being irreducible. Thus (iii)$\Leftrightarrow$(iv) by Lemma \ref{lemma: summarize connected} (i).
%
%
\end{proof}

\begin{defi}
	A $\s$-algebraic group is \emph{connected} if is satisfies the equivalent conditions of Lemma \ref{lemma: characterize connected}.
\end{defi}

\begin{ex}
	Let $G$ be the $\s$-closed subgroup of $\Ga$ defined the linear difference equation $\s^n(y)+\lambda_{n-1}\s^{n-1}(y)+\cdots+\lambda_0y=0$. Then
	$k\{G\}=k[y,\s(y),\ldots,\s^{n-1}(y)]$ is an integral domain. Therefore $G$ is connected.
\end{ex}

\begin{ex}
	Let $G$ be the unitary $\s$-algebraic group, i.e.,
	$$G(R)=\{g\in\Gl_n(R)|\ g\s(g)^{\operatorname{T}}=\s(g)^{\operatorname{T}}g=I_n\}\leq\Gl_n(R)$$
	for any \ks-algebra $R$. The defining equations may be written as $\s(g)=(g^{-1})^{\operatorname{T}}$. So the coordinate ring $k\{G\}=k\big[x_{ij},\frac{1}{\det(x)}\big]$ is an integral domain. Therefore $G$ is connected. 
\end{ex}

%

It is not obvious from the definition that the identity component $G^o$ of a $\s$-algebraic group is connected. The following lemma closes this gap. The proof also shows that $\spec(k\{G^o\})$ is homeomorphic to the connected component of $\spec(k\{G\})$ containing the ``identity'' $\m_G$. (Recall that $\m_G$ is the kernel of the counit $k\{G\}\to k$.)

\begin{lemma} \label{lemma: Go connected}
	Let $G$ be a $\s$-algebraic group. Then $G^o$ is connected. 
\end{lemma}
\begin{proof}
An \'{e}tale $k$-algebra is a finite direct product of finite separable field extensions of $k$. Thus every ideal in an \'{e}tale $k$-algebra is generated by idempotent elements. Thus, also every ideal of $k\{\pi_0(G)\}$ is generated by idempotent elements. As every idempotent element of $k\{G\}$ lies in $k\{\pi_0(G)\}$, it follows that $\m_{\pi_0(G)}=\m_{G}\cap k\{\pi_0(G)\}$ is generated by all idempotent elements contained in $\m_{G}$. Therefore, $\I(G^o)=(\m_{\pi_0(G)})$ is the ideal of $k\{G\}$ generated by all idempotent elements of $k\{G\}$ contained in $\m_{G}$. In other words, by Lemma \ref{lemma: summarize connected} (ii), $\vv(\I(G^o))$ is the connected component of $\spec(k\{G\})$ that contains $\m_G$. As $\VV(\I(G^o))$ and $\spec(k\{G^o\})=\spec(k\{G\}/\I(G^o))$ are homeomorphic, this implies that $G^o$ is connected.
\end{proof}

The formation of Zariski closures is compatible with taking the identity component.
\begin{lemma} \label{lemma: connected component and zariski closure}
	Let $G$ be a $\s$-closed subgroup of an algebraic group $\G$ and for $i\geq 0$ let $G[i]$ and $G^o[i]$ denote the $i$-th order Zariski closure of $G$ and $G^o$ in $\G$ respectively. Then
	$$G^o[i]=G[i]^o.$$
	In particular, $G$ is connected if and only if all its Zariski closures are connected.
\end{lemma}
\begin{proof} Both groups are defined by the ideal of $k[G[i]]\subseteq k\{G\}$ that is generated by all idempotent elements of $k[G[i]]$ contained in the kernel of the counit $k[G[i]]\to k$.
\end{proof}


\begin{cor} \label{cor: sdim of Go}
	Let $G$ be a $\s$-algebraic group. Then $\sdim(G^o)=\sdim(G)$ and $\ord(G^o)=\ord(G)$.
\end{cor}
\begin{proof}
	Let $\G$ be a $\s$-algebraic group containing $G$ as a $\s$-closed subgroup. Then for $i\geq 0$ we have
	$\dim(G[i])=\dim(G[i]^o)=\dim(G^o[i])$ by Lemma \ref{lemma: connected component and zariski closure}. Thus the claim follows from Theorem~\ref{theo: sdimension}.
\end{proof}

The limit degree of $G$ and $G^o$ are in general distinct. Indeed $\ld(G)=\ld(\pi_0(G))\cdot\ld(G^o)$ by Corollary \ref{cor: multiplicativity of limit degree}.

As announced in Section \ref{sec: Subgroups defined by ideal closures}, we can now show that $G_\red$, $G_\sred$, $G_\wm$ and $G_\per$ all have the same $\s$-dimension as $G$.
\begin{lemma} \label{lemma: sdim of sssubgroups}
	Let $G$ be a $\s$-algebraic group. Assume that $k$ has the relevant properties as stated in Corollary~\ref{cor: reduced ssubgroups} (so that we are dealing with $\s$-closed subgroups).
	Then $\sdim(G_\red)$, $\sdim(G_\sred)$, $\sdim(G_\wm)$ and $\sdim(G_\per)$ are all equal to $\sdim(G)$.
\end{lemma}
\begin{proof}
	As the dimension of a finitely generated $k$-algebra remains invariant if we factor by the nilradical, it follows that $\sdim(G_\red)=\sdim(G)$.
	
	Since $(G^o)_\sred\leq G_\sred$, $(G^o)_\wm\leq G_\wm$ and $(G^o)_\per\leq G_\per$ we may assume that $G$ is connected by Lemma \ref{lemma: Go connected} and Corollary \ref{cor: sdim of Go}.
	But then the nilradical of $k\{G\}$ is a prime $\s$\=/ideal (Lemma \ref{lemma: characterize connected}) and therefore $G_\wm=G_\red$ and thus $\sdim(G_\wm)=\sdim(G)$ also in this case.
	
	To prove $\sdim(G_\per)=\sdim(G)$ we may assume that $G$ is reduced. Then the zero ideal of $k\{G\}$ is prime and therefore its reflexive closure $\cup_{i\geq 1}\s^{-i}(0)$ is a $\s$-prime $\s$-ideal. This shows that $G_\per=G_\sred$.
	
	It thus suffices to show that $\sdim(G_\sred)=\sdim(G)$. Let $\G$ be an algebraic group containing $G$ as a $\s$-closed subgroup and for $i\geq 0$ let $G[i]$ and $G_\sred[i]$ denote the $i$-th order Zariski closure of $G$ and $G_\sred$ in $\G$ respectively. By Theorem \ref{theo: first finiteness} we have $$\I(G_\sred[i])=\big(\I(G_\sred[i-1]), \s(\I(G_\sred[i-1]))\big)\subseteq k[G[i]]\subseteq k\{G\}$$ for all sufficiently large $i\in\nn$. But $\I(G_\sred)=\{f\in k\{G\}|\ \exists\ n\geq 1 : \s^n(f)=0\}$.
	This shows that there exist $f_1,\ldots,f_m$ in $k\{G\}$ such that $\I(G_\sred[i])=(f_1,\ldots,f_m)\subseteq k[G[i]]$ for all sufficiently large $i$. Therefore $\dim(G[i])-\dim(G_\sred[i])\leq m$ for all sufficiently large $i\in\nn$ and consequently $\sdim(G)=\sdim(G_\sred)$.
\end{proof}

A normal subgroup of a normal subgroup of an (abstract) group $G$ need not be a normal subgroup of $G$. However, a characteristic subgroup of a normal subgroup of a group $G$ is a normal subgroup of $G$. The following definition, analogous to \cite[Def. 1.51]{Milne:AlgebraicGroupsTheTheoryOfGroupSchemesOfFiniteTypeOverAField}, allows us to transfer this kind of reasoning to $\s$-algebraic groups.

\begin{defi} \label{defi: characteristic subgroup}
	A $\s$-closed subgroup $H$ of a $\s$-algebraic group $G$ is a \emph{characteristic subgroup} of $G$ if for every \ks-algebra $R$, every automorphism of $G_R$ induces an automorphism of $H_R$.
\end{defi}
To be precise, here an automorphism $\f$ of $G_R$ is an isomorphism $\f\colon G_R\to G_R$ of functors form the category of $R$-$\s$-algebras to the category of groups. In particular, $\f_{R'}\colon G(R')\to G(R')$ is an isomorphism of groups for every $R$\=/$\s$\=/algebra $R'$, and the requirement is that $\f_{R'}(H(R'))=H(R')$. Since conjugation with $g\in G(R)$ induces an automorphism of $G_R$, we see that a characteristic subgroup is normal.

Our next goal is to show that $G^o$ is a characteristic subgroup of $G$. To this end we record a practical criterion to test if a normal $\s$-closed subgroup is characteristic.

\begin{lemma} \label{lemma: criterion to be characteristic}
	Let $G$ be a $\s$-algebraic group and $N\unlhd G$ a normal $\s$-closed subgroup. If for every \ks-algebra $R$, every automorphism of the $R$-$\s$-Hopf algebra $k\{G\}\otimes_k R$ maps $k\{G/N\}\otimes_k R$ into $k\{G/N\}\otimes_k R$, then $N$ is a characteristic subgroup of $G$.
\end{lemma}
\begin{proof}
	Let $\f$ be an automorphism of $G_R$. We have to show that $\f$ induces an automorphism of $N_R$, i.e., for every $R$-$\s$-algebra $R'$ the map $\f_{R'}$ maps $N_R(R')=N(R')$ bijectively onto $N(R')$. The automorphism $\f$ of $G_R$ corresponds to an automorphism $\f^*$ of the $R$-$\s$-Hopf algebra $k\{G\}\otimes_k R$. By assumption $\f^*$ and the inverse of $\f^*$ map $k\{G/N\}\otimes_k R$ into itself. Thus, $\f^*$ induces an automorphism of $k\{G/N\}\otimes_k R$. This yields a commutative diagram
	\[\xymatrix{
		G(R') \ar[r] \ar[d]_-{\f_{R'}} &  (G/N)(R') \ar[d] \\
		G(R') \ar[r]  & (G/N)(R')
	}
	\]
	where the vertical arrows are isomorphisms. The claim now follows from the fact that the kernel of $G(R')\to (G/N)(R')$ is $N(R')$ (Theorem \ref{theo: existence of quotient}).
\end{proof}

\begin{prop} \label{prop: Go is characteristic}
	Let $G$ be a $\s$-algebraic group. Then $G^o$ is a characteristic subgroup of $G$.
\end{prop}
\begin{proof}
	By Lemma \ref{lemma: criterion to be characteristic} it suffices to show that for every $k$-$\s$-algebra $R$, every automorphism $\psi$ of the $R$-$\s$-Hopf algebra $k\{G\}\otimes_k R$ maps $\pi_0(k\{G\})\otimes_k R$ into $\pi_0(k\{G\})\otimes_k R$. Using Lemma~\ref{lemma: pi0 for tensor}, we have
	\begin{align*}
	\psi(\pi_0(k\{G\})\otimes 1) & \subseteq \psi(\pi_0(k\{G\}\otimes_k R))\subseteq \pi_0(k\{G\}\otimes_k R)= \\
	& =\pi_0(k\{G\})\otimes_k\pi_0(R)\subseteq \pi_0(k\{G\})\otimes_k R.
	\end{align*}
	Thus  $\psi(\pi_0(k\{G\})\otimes_k R)\subseteq\pi_0(k\{G\})\otimes_k R$ as required.
\end{proof}

\subsection{The strong identity component}

The following lemma facilitates the definition of the strong identity component.

\begin{lemma}
	Let $G$ be a $\s$-algebraic group with $\sdim(G)>0$. Among the $\s$-closed subgroups $H$ of $G$ with $\sdim(H)=\sdim(G)$, there exists a unique smallest one.
\end{lemma}
\begin{proof}
	Let $H_1$ and $H_2$ be $\s$-closed subgroups of $G$ with $\sdim(H_1)=\sdim(H_2)=\sdim(G)$. By Theorem \ref{theo: dimension theore} we have $\sdim(H_1\cap H_2)=\sdim(G)$. Thus the claim follows from Theorem~\ref{theo: first finiteness}.
\end{proof}

\begin{defi}
	Let $G$ be a $\s$-algebraic group with $\sdim(G)>0$. The \emph{strong identity component} $G^{so}$ of $G$ is the  smallest $\s$-closed subgroup of $G$ with $\s$-dimension equal to the $\s$\=/dimension of $G$. A $\s$-algebraic group is \emph{strongly connected} if it has positive $\s$-dimension and equals its strong identity component.
\end{defi}
Thus a $\s$-algebraic group $G$ with $\sdim(G)>0$ is strongly connected if and only if it has no proper $\s$-closed subgroup of the same $\s$-dimension. The strong identity component of a $\s$\=/algebraic group is strongly connected.
 As $G^o$ is a $\s$-closed subgroup with $\sdim(G^o)=\sdim(G)$ (Corollary \ref{cor: sdim of Go}), we see that a strongly connected $\s$-algebraic group is connected.

 \begin{lemma} \label{lemma: strongly connected is sintegral}
 	Assume that $k$ is perfect and inversive. Then a strongly connected \mbox{$\s$-algebraic} group is $\s$-integral (in particular, perfectly $\s$-reduced).
 \end{lemma}
 \begin{proof}
 	Let $G$ be a strongly connected $\s$-algebraic group. Then $G$ is connected and because $\sdim(G)=\sdim(G_\red)$ by Lemma \ref{lemma: sdim of sssubgroups}, we must have $G=G_\red$. So $G$ is reduced and hence integral by Lemma \ref{lemma: characterize connected}. Similarly, $G=G_\sred$ by Lemma \ref{lemma: sdim of sssubgroups}. Thus $G$ is $\s$-integral.
 \end{proof}

\begin{ex} \label{ex: strong component}
	If $\G$ is a smooth, connected algebraic group with $\dim(\G)>0$, then $G=[\s]_k\G$ is strongly connected. Indeed, as $\G$ is smooth and connected, the same holds for $G[i]=\G\times\ldots\times {\hsi \G}$ for every $i\geq 0$. So if $H$ is a proper $\s$-closed subgroup of $G$, then $\dim(H[i])<\dim(G[i])$ for all sufficiently large $i\in\nn$. But $\dim(G[i])=\dim(\G)(i+1)$ and so it follows from Theorem \ref{theo: sdimension} that $\sdim(H)<\sdim(G)$.
	
	If $\G$ is not smooth or not connected, then $[\s]_k\G$ need not be strongly connected.
\end{ex}

We next give an example of a $\s$-integral $\s$-algebraic group that is not strongly connected.

\begin{ex}
	Let $G$ be the $\s$-closed subgroup of $\Ga^2$ given by
	$$G(R)=\{(g_1, g_2)\in R^2|\ \s(g_1)=g_1\}\leq \Ga^2(R)$$ for any $k$-$\s$-algebra $R$. As $k\{G\}=k[y_1]\{y_2\}$ with $\s(y_1)=y_1$ we see that $G$ is $\s$-integral. We have $\sdim(G)=1$. The $\s$-closed subgroup $H$ of $G$ given by
	$H(R)=\{(0,g)\in R^2\}$ is isomorphic to $\Ga$ and therefore also has $\s$-dimension one. Using Example \ref{ex: strong component} we see that $G^{so}=H$.
\end{ex}

The following example shows that Lemma \ref{lemma: strongly connected is sintegral} fails over an arbitrary base $\s$-field. There exists a strongly connected $\s$-algebraic group that is not $\s$-reduced.

\begin{ex} \label{ex: strongly connected but not sreduced}
	Let $k$ be a non-inversive $\s$-field of characteristic zero. So there exists $\lambda\in k$ with $\lambda\notin \s(k)$. Let $G$ be the $\s$-closed subgroup of $\Ga^2$ given by
	$$G(R)=\left\{(g_1,g_2)\in R^2|\ \s(g_1)=\lambda\s(g_2)\right\}$$
	for any $k$-$\s$-algebra $R$.
	Then $k\{G\}=k[y_1,y_2,\s(y_2),\ldots]$ with $\s(y_1)=\lambda \s(y_2)$. For $i\geq 0$ let $G[i]$ denote the $i$-th order Zariski closure of $G$ in $\Ga^2$. Then $k[G[i]]=k[y_1,y_2,\ldots,\s^i(y_2)]$ and therefore $\dim(G[i])=1\cdot(i+1)+1$, in particular, $\sdim(G)=1$.
	
	We claim that $G$ is strongly connected. Suppose that $H\leq G$ is a proper $\s$-closed subgroup with $\sdim(H)=\sdim(G)$.
	Let $a_1$ and $a_2$ denote the image of $y_1$ and $y_2$ in $k\{H\}$ respectively.
	By \cite[Corollary A.3]{DiVizioHardouinWibmer:DifferenceAlgebraicRel} the $\s$-ideal $\I(H)\subseteq k\{\Ga^2\}$ is $\s$-generated by homogenous linear $\s$-polynomials. Thus there exists a non-trivial $k$-linear relation between $a_1, a_2, \s(a_2),\ldots$. If that relation would properly involve $\s^i(a_2)$ for $i\geq 1$, then $\sdim(H)=0$. Thus there exists a non-trivial $k$-linear relation between $a_1$ and $a_2$. We have $a_1\neq 0$ and $a_2\neq 0$ because otherwise $\sdim(H)=0$. So there exists $\mu\in k$ with $a_1-\mu a_2=0$. Consequently $$0=\s(a_1)-\s(\mu)\s(a_2)=\lambda\s(a_2)-\s(\mu)\s(a_2)=(\lambda-\s(\mu))\s(a_2).$$ Since $\lambda\notin\s(k)$ this implies $\s(a_2)=0$. But then $\sdim(H)=0$; a contradiction.
	
	Now assume that $\lambda^2\in \s(k)$. (For example, we can choose
	$k=\mathbb{C}(\sqrt{x},\sqrt{x+1},\ldots)$ with action of $\s$ determined by $\s(x)=x+1$ and $\lambda=\sqrt{x}$.) If $\mu\in k$ with $\s(\mu)=\lambda^2$ then $\s(y_1^2-\mu y_2^2)=0$. Thus $G$ is not $\s$-reduced.

	%
	%
	%
	%
\end{ex}

The strong identity component is essential for the proof of our Jordan-H\"{o}lder type theorem (Theorem~A from the introduction). The idea for the proof of the existence part of this theorem is easy to explain: Starting with a strongly connected $\s$-algebraic group $G$, we can a choose among all proper normal $\s$-closed subgroups of positive $\s$-dimension one, say $G_1$, of maximal $\s$-dimension. Since $G$ is strongly connected, $\sdim(G_1)<\sdim(G)$. Moreover, $G/G_1$ is almost-simple by choice of $G_1$. To conclude the proof by induction on the $\s$-dimension, one would like to replace $G_1$ by its strong identity component $G_1^{so}$. However, for this to work one needs to know that $G_1^{so}$ is normal in $G$. The latter would be true if we knew that $G_1^{so}$ is a characteristic subgroup of $G_1$.

It is clear that every automorphism of a $\s$-algebraic group $G$ of positive $\s$-dimension, induces an automorphism of $G^{so}$. However, this is weaker than Definition \ref{defi: characteristic subgroup} and indeed, in general, $G^{so}$ need not be a characteristic subgroup of $G$. In fact, the following example illustrates that $G^{so}$ need not even be normal in $G$. (This is similar to the situation with algebraic groups. Cf. \cite[6.11]{Milne:AlgebraicGroupsTheTheoryOfGroupSchemesOfFiniteTypeOverAField}.)

\begin{ex} \label{ex: Gso not normal}
	Let $G=N\rtimes H$ be the $\s$-algebraic group from Example \ref{ex: Gsred not normal}. Then $\sdim(G)=1$.
	The $\s$-closed subgroup $H=\Gm$ of $G$ has $\s$-dimension one. Since $H$ is strongly connected (Example \ref{ex: strong component}) we see that $H=G^{so}$.
	We already noted in Example \ref{ex: Gsred not normal} that $H$ is not normal in $G$.
\end{ex}

The following proposition salvages the above plan to establish the existence part of our Jordan-H\"{o}lder type theorem.

\begin{prop} \label{prop: normality for strong comp}
	Assume that $k$ is algebraically closed and inversive. Let $G$ be a perfectly $\s$\=/reduced $\s$-algebraic group and $H$ a normal $\s$-closed subgroup of $G$ with $\sdim(H)>0$. Then $H^{so}$ is normal in $G$. (In particular, $H^{so}$ is normal in $H$.)
\end{prop}

As an immediate corollary to Proposition \ref{prop: normality for strong comp} we obtain:

\begin{cor} \label{cor: perfectly sreduced implies GSo normal}
	Let $G$ be a perfectly $\s$-reduced $\s$-algebraic group over an inversive algebraically closed $\s$-field with $\sdim(G)>0$. Then $G^{so}$ is a normal $\s$-closed subgroup of $G$.
\end{cor}

For the proof of Proposition \ref{prop: normality for strong comp} we need two preparatory lemmas.

\begin{lemma} \label{lemma: exists good extension}
	Assume that $k$ is algebraically closed and inversive. Let $K$ be a $\s$-field extension of $k$. Then there exists a $\s$-field extension $L$ of $K$ such that only the elements of $k$ are fixed by all $\s$-field automorphisms of $L/k$, i.e., $L^{\Aut(L/k)}=k$.
\end{lemma}
\begin{proof}
	Let us start with proving the following claim: There exists a $\s$-field extension $L$ of $K$ such that for all $a\in K\smallsetminus k$ there exists a $\s$-field automorphism $\tau$ of $L/k$ with $\tau(a)\neq a$ and such that every $\s$-field automorphism of $K/k$ extends to a $\s$-field automorphism of $L/k$.
	
	Since $k$ is algebraically closed, $K\otimes_k K$ is an integral domain. Since $k$ is inversive, $K\otimes_k K$ is $\s$-reduced (Lemma \ref{lemma: tensor stays} (ii)). Therefore the field of fractions $L$ of $K\otimes_k K$ is naturally a $\s$-field. Consider $L$ as a $\s$-field extension of $K$ via the embedding $a\mapsto a\otimes 1$. The $\s$-field automorphism $\tau$ of $L/k$ determined by $\tau(a\otimes b)=b\otimes a$ moves every element of $K\smallsetminus k$. Moreover, every $\s$-field automorphism $\tau'$ of $K/k$ extends to $L/k$, for example, by $\tau'(a\otimes b)=\tau'(a)\otimes b$.
	
	Now let us prove the lemma. By the above claim, there exists a $\s$-field extension $L_1/K$ such that every element of $K\smallsetminus k$ can be moved by a $\s$-field automorphism of $L_1/K$ and every $\s$-field automorphism of $K/k$ extends to a $\s$-field automorphism of $L_1/k$. Now apply the claim again to $L_1/k$ to find a $\s$-field extension $L_2/L_1$ such that every element of $L_1\smallsetminus k$ can be moved by a $\s$-field automorphism of $L_2/k$ and every $\s$-field automorphism of $L_1/k$ extends to a $\s$-field automorphism of $L_2/k$. Continuing like this we obtain a chain of $\s$-field extensions $k\subseteq K\subseteq L_1\subseteq L_2\subseteq\ldots$. The union $L=\cup L_i$ has the required property.
\end{proof}

The formation of the strong identity component is compatible with base change under certain assumptions:
\begin{lemma} \label{lemma: strong component stable under baseext}
	Assume that $k$ is algebraically closed and inversive. Let $G$ be a \mbox{$\s$-algebraic} group with $\sdim(G)>0$ and let $K$ be a $\s$-field extension of $k$.
	Then $(G_K)^{so}=(G^{so})_K.$
\end{lemma}
\begin{proof}
	As the $\s$-dimension is invariant under base extension (Lemma \ref{lemma: sdim and base change}),
	$$\sdim((G^{so})_K)=\sdim(G^{so})=\sdim(G)=\sdim(G_K).$$ Therefore $(G_K)^{so}\leq (G^{so})_K$.
	
	Let us now next show that $(G_K)^{so}$ descends to $k$, i.e., there exists a $\s$-closed subgroup $H$ of $G$ with $(G_K)^{so}=H_K$. By Lemma \ref{lemma: exists good extension} there exists a $\s$-field extension $L$ of $K$ such that $L^{\Aut(L|k)}=k$, where $\Aut(L|k)$ is the group of all $\s$-field automorphisms of $L|k$. The group $\Aut(L|k)$ acts on $L\{G_L\}=k\{G\}\otimes_k L$ by $k$-$\s$-algebra automorphisms via the right factor. Let $H'$ be a $\s$-closed subgroup of $G_L$. Since the Hopf algebra structure maps commute with the $\Aut(L|k)$-action, $\tau(\I(H'))$ is a $\s$-Hopf ideal of $k\{G\}\otimes_k L$ for every $\tau\in \Aut(L|k)$. Moreover, the $\s$-dimension of the $\s$-closed subgroup of $G_L$ defined by $\tau(\I(H'))$ is equal to the $\s$-dimension of $H'$. Since $\I((G_L)^{so})$ is the unique maximal $\s$-Hopf ideal of $k\{G\}\otimes_k L$ such that $\sdim((G_L)^{so})=\sdim(G_L)$, we see that $\tau(\I((G_L)^{so}))=\I((G_L)^{so})$ for every $\tau\in \Aut(L|k)$. Let
	$$\ida=\{f\in \I((G_L)^{so})|\ \tau(f)=f\ \forall\ \tau\in \Aut(L|k)\}=\I((G_L)^{so})\cap k\{G\}.$$
	Since the action of $\Aut(L|k)$ commutes with the Hopf algebra structure maps, $\ida$ is \mbox{$\s$-Hopf} ideal of $k\{G\}$ and therefore corresponds to a $\s$-closed subgroup $H$ of $k\{G\}$. We have $\ida\otimes_k L=\I((G_L)^{so})$. (See \cite[Corollary to Proposition 6, Chapter V, \S 10.4, A.V.63]{Bourbaki:Algebra2}.) So $H_L=(G_L)^{so}$.
	As $\sdim(H)=\sdim((G_L)^{so})=\sdim(G_L)=\sdim(G)$ we see that $G^{so}\leq H$, therefore $(G^{so})_L\leq H_L=(G_L)^{so}$. Hence also
	$$((G^{so})_K)_L=(G^{so})_L\leq (G_L)^{so}\leq ((G_K)^{so})_L.$$
	Thus $(G^{so})_K\leq (G_K)^{so}$.
\end{proof}

The following example shows that the formation of the strong identity component is in general not compatible with base change.

\begin{ex}
	Let $G$ be the strongly connected $\s$-algebraic group from Example \ref{ex: strongly connected but not sreduced}. Let $K=k^*$ be the inversive closure of $k$ (see (\cite[Definition 2.1.6]{Levin:difference})) and let $\mu\in K$ with $\s(\mu)=\lambda$. Then $G_K$ is not strongly connected since it has the $\s$-closed subgroup $H$ of $\sdim(H)=1=\sdim(G)$ given by
	$$H(R)=\{(g_1,g_2)\in R^2|\ g_1=\mu g_2\}$$ for any $k$-$\s$-algebra $R$. So $(G_K)^{so}$ is properly contained in $(G^{so})_K=G_K$.
\end{ex}

We are now prepared to prove Proposition \ref{prop: normality for strong comp}.

\begin{proof}[Proof of Proposition \ref{prop: normality for strong comp}]
	We have to show that the morphism of $\s$-varieties
	$$\f\colon G\times H^{so}\to G,\ (g,h)\mapsto ghg^{-1}$$ maps into $H^{so}$.
	We know from Lemma \ref{lemma: strongly connected is sintegral} that $H^{so}$ is perfectly $\s$-reduced. By assumption also $G$ is perfectly $\s$-reduced. Therefore, by Lemma~\ref{lemma: tensor stays}~(iv), also the product $G\times H^{so}$ is perfectly $\s$-reduced. So, by Lemma \ref{lemma: morphism with perfectly sreduced}, it suffices to show that $\f_K((G\times H^{so})(K))\subseteq H^{so}(K)$ for every $\s$-field extension $K$ of $k$. Let $g\in G(K)$. Then $g$ induces an automorphism of $G_K$ by conjugation. Since $H$ is normal in $G$ we have an induced automorphism on $H_K$. This automorphism maps $(H_K)^{so}$ into $(H_K)^{so}$. But $(H_K)^{so}=(H^{so})_K$ by Lemma \ref{lemma: strong component stable under baseext}. This shows that conjugation by $g$ maps $H^{so}(K)$ into $H^{so}(K)$. Thus $\f_K((G\times H^{so})(K))\subseteq H^{so}(K)$ as required.
\end{proof}

\section{Jordan-H\"{o}lder theorem} \label{sec: JordanHoelder}

In this section we apply the results from the previous sections to establish our Jordan-H\"{o}lder type theorem for $\s$-algebraic groups. 
A Jordan-H\"{o}lder type theorem for algebraic groups can be found in \cite{Rosenlicht:SomeBasicTheoremsOnAlgebraicGroups}, while a Jordan-H\"{o}lder type theorem for differential algebraic groups has been proved in \cite{CassidySinger:AJordanHoelderTheoremForDifferentialAlgebraicGroups}.

As we will show, the Schreier refinement theorem also holds for $\s$-algebraic groups (Theorem~\ref{theo: schreier refinement}). In particular, any two decomposition series of a $\s$-algebraic group are equivalent. Here a decomposition series is a subnormal series such that the quotient groups have no proper non-trivial normal $\s$-closed subgroups.

However, a $\s$-algebraic group rarely has a decomposition series. It is therefore useful to consider more general subnormal series and to relax the condition that the quotient groups should have no proper non-trivial normal $\s$-closed subgroups. This is where the almost-simple $\s$-algebraic groups enter into the picture.

The basic idea is to consider $\s$-algebraic groups up to quotients by zero $\s$-dimensional normal subgroups. Formally this is realized by replacing in the uniqueness statement of the classical Jordan-H\"{o}lder theorem the notion of isomorphism by the notion of isogeny.

\medskip

Our first aim is to prove the analog of the Schreier refinement theorem, which plays a key role in the proof of the uniqueness part of our Jordan-H\"{o}lder type theorem. We follow along the lines of the well-known proof via the Butterfly lemma. (Cf. \cite[Section 1.3]{Lang:Algebra} and \cite[Section~6~a]{Milne:AlgebraicGroupsTheTheoryOfGroupSchemesOfFiniteTypeOverAField}.) We will need two analogs of elementary statements about groups.

\begin{lemma} \label{lemma: Dedekind law}
	Let $N$, $G$ and $H$ be $\s$-closed subgroups of a $\s$-algebraic group $G'$ such that $N\unlhd G$ and $N$ normalizes $H$. Then $G\cap NH=N(G\cap H)$.
\end{lemma}
\begin{proof}
	As $N\leq G\cap NH$ and $G\cap H\leq G\cap NH$ it is clear that $N(G\cap H)\subseteq G\cap NH$.
	
	Conversely, let $R$ be a $k$-$\s$-algebra and
	$g\in (G\cap NH)(R)$. There exists a faithfully flat $R$\=/$\s$\=/algebra $S$ and $n\in N(S)$, $h\in H(S)$ such that $g=nh$ in $G'(S)$. But then $h=n^{-1}g\in G(S)$ and therefore $g=nh\in N(S)(G(S)\cap H(S))\subseteq (N(G\cap H))(S)$. It follows from Lemma \ref{lemma: faithfullyflat intersection} that $g\in (N(G\cap H))(R)$.
\end{proof}

\begin{lemma} \label{lemma: normal normalizes}
	Let $H_1\unlhd H_2$ be $\s$-closed subgroups of a $\s$-algebraic group $G$. Assume that $H_2$ normalizes $N\leq G$. Then $NH_1\unlhd NH_2$.
\end{lemma}
\begin{proof}
	Clearly $N\rtimes H_1$ is a normal $\s$-closed subgroup of $N\rtimes H_2$. Therefore $NH_1=m(N\rtimes H_1)$ is a normal $\s$-closed subgroup of $NH_2=m(N\rtimes H_2)$ by Lemma \ref{lemma: normal stays normal under surjective map}.
\end{proof}

The following lemma is the analog of the Butterfly (or Zassenhaus) lemma.

\begin{lemma} \label{lemma: butterfly}
	Let $N_1\unlhd H_1$ and $N_2\unlhd H_2$ be $\s$-closed subgroups of a $\s$-algebraic group $G$. Then $N_1(H_1\cap N_2)\unlhd N_1(H_1\cap H_2)$,
	$N_2(N_1\cap H_2)\unlhd N_2(H_1\cap H_2)$ and
	$$\frac{N_1(H_1\cap H_2)}{N_1(H_1\cap N_2)}\simeq \frac{N_2(H_1\cap H_2)}{N_2(N_1\cap H_2)}.$$
\end{lemma}
\begin{proof}
	Since $H_1\cap N_2$ is normal in $H_1\cap H_2$ it follows from Lemma \ref{lemma: normal normalizes} that $N_1(H_1\cap N_2)$ is normal in $N_1(H_1\cap H_2)$. Similarly, $N_2(N_1\cap H_2)\unlhd N_2(H_1\cap H_2)$. As $H_1\cap H_2$ normalizes $N_1(H_1\cap N_2)$ it follows from Theorem \ref{theo: isom2} that
	\begin{equation} \label{eqn: butterfly}
	\frac{H_1\cap H_2}{(H_1\cap H_2)\cap N_1(H_1\cap N_2)}\simeq\frac{(H_1\cap H_2)N_1(H_1\cap N_2)}{N_1(H_1\cap N_2)}.
	\end{equation}
	Lemma \ref{lemma: Dedekind law} with $N=H_1\cap N_2$, $G=H_1\cap H_2$ and $H=N_1$ shows that
	$$(H_1\cap H_2)\cap N_1(H_1\cap N_2)=(H_1\cap N_2)(H_1\cap H_2\cap N_1)= (H_1\cap N_2)(N_1\cap H_2).$$
	Because $H_1\cap N_2\subseteq H_1\cap H_2$ we find $(H_1\cap H_2)N_1(H_1\cap N_2)=N_1(H_1\cap H_2)$. Therefore (\ref{eqn: butterfly}) becomes
	$$ \frac{H_1\cap H_2}{(H_1\cap N_2)(N_1\cap H_2)}\simeq\frac{N_1(H_1\cap H_2)}{N_1(H_1\cap N_2)}.$$ By symmetry
	$$ \frac{H_1\cap H_2}{(H_1\cap N_2)(N_1\cap H_2)}\simeq\frac{N_2(H_1\cap H_2)}{N_2(N_1\cap H_2)}.$$
\end{proof}

\begin{defi}
	Let $G$ be a $\s$-algebraic group. A \emph{subnormal series}\index{subnormal series} of $G$ is a sequence
	\begin{equation} \label{eqn: subnormal series} G=G_0\supseteq G_1\supseteq\cdots\supseteq G_n=1 \end{equation}
	of $\s$-closed subgroups of $G$ such that $G_{i+1}$ is normal in $G_i$ for $i=0,\ldots,n-1$. Another subnormal series
	\begin{equation} \label{eqn: subnormal series2} G=H_0\supseteq H_1\supseteq\cdots\supseteq H_m=1 \end{equation}
	of $G$ is a \emph{refinement} of (\ref{eqn: subnormal series}) if $\{G_0,\ldots,G_n\}\subseteq\{H_1,\ldots,H_m\}$. The subnormal series (\ref{eqn: subnormal series}) and (\ref{eqn: subnormal series2}) are \emph{equivalent} if $m=n$ and there exists a permutation $\pi$ such that
	the quotient groups $G_i/G_{i+1}$ and $H_{\pi(i)}/H_{\pi(i)+1}$ are isomorphic for $i=0,\ldots,n-1$. 
\end{defi}

The following theorem is the analog of the Schreier refinement theorem.

\begin{theo} \label{theo: schreier refinement}
	Any two subnormal series of a $\s$-algebraic group have equivalent refinements.
\end{theo}
\begin{proof}
	Let $$G=G_0\supseteq G_1\supseteq\cdots\supseteq G_n=1$$ and
	$$G=H_0\supseteq H_1\supseteq\cdots\supseteq H_m=1$$ be subnormal series of a $\s$-algebraic group $G$. Set $G_{i,j}=G_{i+1}(H_j\cap G_i)$ for $i=0,\ldots,n-1$ and $j=0,\ldots,m$. Then
	$$G=G_0=G_{0,0}\supseteq G_{0,1}\supseteq G_{0,2}\supseteq\cdots\supseteq G_{0,m}=G_1=G_{1,0}\supseteq G_{1,1}\supseteq\cdots\supseteq G_{n-1,m}=1$$
	is a subnormal series for $G$. Similarly, setting $H_{j,i}=H_{j+1}(G_i\cap H_j)$ for $j=0,\ldots,m-1$ and $i=0,\ldots,n$, defines a subnormal series for $G$.
	By Lemma \ref{lemma: butterfly}
	$$G_{i,j}/G_{i,j+1}\simeq H_{j,i}/H_{j,i+1}.$$
\end{proof}

See Example \ref{ex: Jordan Hoelder} for an example illustrating Theorem \ref{theo: schreier refinement}.
The following definition is crucial for the uniqueness part of our Jordan-H\"{o}lder type theorem.

\begin{defi}
	Let $G$ and $H$ be strongly connected $\s$-algebraic groups. A morphism $\f\colon G\to H$ is an \emph{isogeny} if $\f$ is a quotient map and $\sdim(\ker(\f))=0$. Two strongly connected $\s$-algebraic groups $H_1$ and $H_2$ are \emph{isogenous}\index{isogenous} if there exists a strongly connected $\s$-algebraic group $G$ and isogenies $G\twoheadrightarrow H_1$ and $G\twoheadrightarrow H_2$.
\end{defi}
By Theorem \ref{theo: isom1} and  Corollary \ref{cor: sdim and ord for quotients} a quotient map $\f\colon G\twoheadrightarrow H$ is an isogeny, if and only if $\sdim(G)=\sdim(H)$. In particular, isogenous $\s$-algebraic groups have the same $\s$-dimension.
\begin{lemma} \label{lemma: composition of isogenies}
	The composition of two isogenies is an isogeny.
\end{lemma}
\begin{proof}
	Clearly the composition of two quotient maps is a quotient map. If $G_1\to G_2$ and $G_2\to G_3$ are isogenies, then $\sdim(G_1)=\sdim(G_2)$ and $\sdim(G_2)=\sdim(G_3)$. Therefore $\sdim(G_1)=\sdim(G_3)$.
\end{proof}

\begin{lemma} \label{lemma: isogeny is equivalence relation}
	Isogeny is an equivalence relation on the class of strongly connected $\s$-algebraic groups.
\end{lemma}
\begin{proof}
	Reflexivity and symmetry are obvious. Let us prove the transitivity. So let \mbox{$\f_1\colon G\twoheadrightarrow H_1$}, $\f_2\colon G\twoheadrightarrow H_2$ and $\f_2'\colon G'\twoheadrightarrow H_2$, $\f_3'\colon G'\twoheadrightarrow H_3$ be isogenies. The morphism $\f_2\times\f_2'\colon G\times G'\to H_2\times H_2$
	is a quotient map with kernel $\ker(\f_2)\times\ker(\f_2')$, which has $\s$-dimension zero by Lemma~\ref{lemma: sdim and ord for product}. The diagonal
	$D\leq H_2\times H_2$ given by $D(R)=\{(h_2,h_2)|\ h_2\in H_2(R)\}$ for any $k$-$\s$-algebra $R$ is a $\s$-closed subgroup of $H_2\times H_2$ isomorphic to $H_2$. Therefore $G''=((\f_2\times\f_2')^{-1}(D))^{so}$ is a $\s$-closed subgroup of $G\times G'$ with $\sdim(G'')=\sdim(H_2)$. Let $\pi\colon G''\to G$ and $\pi'\colon G''\to G'$ denote the projections onto the first and second factor respectively. We have the following diagram:
	$$
	\xymatrix{
		&  & G'' \ar_\pi[ld] \ar^{\pi'}[rd] & & \\
		& G \ar_-{\f_1}[ld] \ar^-{\f_2}[rd] &  & G' \ar_-{\f_2'}[ld] \ar^-{\f_3'}[rd] & \\
		H_1 &  & H_2 & & H_3 \\
	}
	$$

	We claim that $\pi$ and $\pi'$ are isogenies. We have $\ker(\pi)\leq 1\times \ker(\f_2')$. Therefore $\sdim(\ker(\pi))=0$ and consequently $$\sdim(\pi(G''))=\sdim(G'')=\sdim(H_2)=\sdim(G).$$ Since $G$ is strongly connected, this shows that $\pi(G'')=G$, so $\pi$ is a quotient map. Hence $\pi$ is an isogeny. Similarly, it follows that $\pi'$ is an isogeny. The isogenies $\f_1\pi$ and $\f_3'\pi'$ (Lemma~\ref{lemma: composition of isogenies}) show that $H_1$ and $H_3$ are isogenous.
\end{proof}

Difference algebraic groups rarely possess decomposition series, i.e., a subnormal series such that the quotient groups have no proper, non-trivial normal $\s$-closed subgroups. This is illustrated in the following example.  

%

\begin{ex} \label{ex: no decomposition series for Ga} Let $k$ be a $\s$-field of characteristic zero.
	The $\s$-algebraic group $G=\Ga$ does not have a decomposition series. Indeed, by \cite[Corollary A.3]{DiVizioHardouinWibmer:DifferenceAlgebraicRel}, every proper $\s$-closed subgroup $G$ of $\Ga$ is of the form
	$G(R)=\{g\in R|\ f(g)=0\}$ for some non-zero homogeneous linear difference equation $f=\s^n(y)+\lambda_{n-1}\s^{n-1}(y)+\cdots+\lambda_0y$.
	If $h$ is another non-trivial linear homogeneous difference equation, then the product $h*f$ in the sense of linear difference operators (see \cite[Section 3.1]{Levin:difference}) defines a $\s$-closed subgroup $H$ of $\Ga$ with $G\subsetneqq H\subsetneqq\Ga$. For example, for $h=\s(y)$ we have
	$$H(R)=\{g\in R| \ \s^{n+1}(g)+\sigma(\lambda_{n-1})\s^{n}(y)+\cdots+\sigma(\lambda_0)\s(g)=0 \}.$$
	
\end{ex}

To remedy this shortcoming we need to relax the condition that the quotient groups of a decomposition series should have no proper non-trivial normal $\s$-closed subgroups. This leads to the following definition:

\begin{defi} \label{defi: almost-simple salgebraic group}
	A $\s$-algebraic group over an algebraically closed, inversive $\s$-field is \emph{almost-simple} if it is perfectly $\s$-reduced, has positive $\s$-dimension and every normal proper $\s$\=/closed subgroup has $\s$-dimension zero.
\end{defi}

Recall that almost-simple algebraic groups are, by definition, required to be geometrically reduced. As argued in Remark \ref{rem: perfectly sreduced and variety}, the assumption to be perfectly $\s$-reduced over an algebraically closed inversive $\s$-field can be seen as an analog of this requirement. The structure of almost-simple $\s$-algebraic groups is investigated in the next section. In particular, it is shown there that for an almost-simple algebraic group $\G$, the $\s$-algebraic group $[\s]_k\G$ is almost-simple. See Examples \ref{ex: isogenous to Gm} and \ref{ex: almost-simple} for further examples of almost-simple $\s$-algebraic groups.

\begin{lemma} \label{lemma: almost-simple is sintegral}
	Assume that $k$ is algebraically closed and inversive. An almost-simple $\s$\=/algebraic group is strongly connected and $\s$-integral.
\end{lemma}
\begin{proof}
	Clear from Corollary \ref{cor: perfectly sreduced implies GSo normal} and Lemma \ref{lemma: strongly connected is sintegral}.
\end{proof}

Almost-simplicity is preserved under isogeny:

\begin{lemma} \label{lemma: almost-simple and isogeny}
	Assume that $k$ is algebraically closed and inversive. Let $G$ and $H$ be strongly connected isogenous $\s$-algebraic groups. Then $G$ is almost-simple if and only if $H$ is almost-simple. 
\end{lemma}
\begin{proof}
	We may assume, without loss of generality, that there exists an isogeny $\f\colon G\twoheadrightarrow H$. Recall from Lemma \ref{lemma: strongly connected is sintegral} that $G$ and $H$ are perfectly $\s$-reduced.
	If $G$ is almost-simple, then $H$ is almost-simple by Theorem \ref{theo: isom3} and Corollary \ref{cor: sdim and ord for quotients}.
	
	Conversely, assume that $H$ is almost-simple and let $N$ be a proper normal $\s$-closed subgroup of $G$. Then $\f(N)$ is a normal $\s$-closed subgroup of $H$. There are two cases: either $\f(N)=H$ or $\f(N)$ has $\s$-dimension zero. In the latter case it follows that $N$ has $\s$-dimension zero and we are done. So it suffices to show that the case $\f(N)=H$ cannot occur. Suppose, for a contradiction, that $\f(N)=H$. Then $N\ker(\f)=G$ by Theorem \ref{theo: isom3}. Using Theorem \ref{theo: isom2} we find
	$$H\simeq G/\ker(\f)=N\ker(\f)/\ker(\f)\simeq N/(N\cap\ker(\f)).$$
	This implies $\sdim(N)=\sdim(H)=\sdim(G)$. As $G$ is strongly connected, we arrive at the contradiction $N=G$.
\end{proof}

We are now prepared to prove our Jordan-H\"{o}lder type theorem.

\begin{theo} \label{theo: Jordan Hoelder}
	Assume that $k$ is algebraically closed and inversive. Let $G$ be a strongly connected $\s$-algebraic group. Then there exists a subnormal series \begin{equation} \label{eqn: G subnormal series} G=G_0\supseteq G_1\supseteq\cdots\supseteq G_n=1 \end{equation}
	of strongly connected $\s$-closed subgroups such that $G_i/G_{i+1}$ is almost-simple for $i=0,\ldots,n-1$. If
	\begin{equation} \label{eqn: H subnormal series} G=H_0\supseteq H_1\supseteq\cdots\supseteq H_m=1 \end{equation}
	is another such subnormal series, then $m=n$ and there exists a permutation $\pi$ such that
	$G_{i}/G_{i+1}$ and $H_{\pi(i)}/H_{\pi(i)+1}$ are isogenous for $i=0,\ldots,n-1$.
\end{theo}
\begin{proof}
	Let us first prove the existence statement.
	Among all normal proper $\s$-closed subgroups of $G$ choose one, say $H$, with maximal $\s$-dimension. If $\sdim(H)=0$, then $G$ is almost-simple and we are done. So let us assume that $\sdim(H)>0$. Since $G$ is strongly connected, $\sdim(H)<\sdim(G)$. By construction $G/H$ is almost-simple (Theorem~\ref{theo: isom3}). Let $G_1=H^{so}$. Then $G_1$ is normal in $G$ by Proposition \ref{prop: normality for strong comp} and \mbox{$\sdim(G/G_1)>0$}. Let us show that $G/G_1$ is almost-simple. Let $N$ be a proper normal $\s$-closed subgroup of $G$ containing $G_1$. By choice of $H$, $\sdim(N)\leq\sdim(H)$, but since $\sdim(H)=\sdim(G_1)\leq\sdim(N)$ we have $\sdim(N)=\sdim(G_1)$. So $G/G_1$ is almost-simple (Theorem \ref{theo: isom3}). As $\sdim(G_1)<\sdim(G)$ the claim follows by induction on $\sdim(G)$.
	
	Now let us prove the uniqueness statement. It follows from Theorem \ref{theo: schreier refinement} that (\ref{eqn: G subnormal series}) and (\ref{eqn: H subnormal series}) have equivalent refinements. Let
	\begin{equation} \label{eqn: G refinement} G=G_0\supsetneqq G_{0,1}\supsetneqq G_{0,2}\supsetneqq\cdots\supsetneqq G_{0,n_0}\supsetneqq G_1 \supsetneqq G_{1,1}\supsetneqq\cdots\supsetneqq G_{1,n_1}\supsetneqq G_2\supsetneqq\cdots\supsetneqq G_n=1 \end{equation} be such a refinement of (\ref{eqn: G subnormal series}). For $i=0,\ldots,n-1$, as $G_i$ is strongly connected and $G_i/G_{i+1}$ is almost-simple, $\sdim(G_i/G_{i,1})=\sdim(G_i/G_{i+1})>0$ and
	$\sdim(G_{i,j}/G_{i,j+1})=0$ for $j=1,\ldots,n_i-1$, also $\sdim(G_{i,n_i}/G_{i+1})=0$.
	The kernel of $G_i/G_{i+1}\twoheadrightarrow G_i/G_{i,1}$ has $\s$\=/dimension zero, so $G_i/G_{i+1}$ and $G_i/G_{i,1}$ are isogenous. In summary, we find that among the quotient groups of the subnormal series (\ref{eqn: G refinement}), there are precisely $n$ of positive $\s$-dimension, (namely $G_i/G_{i,1}$, $i=0,\ldots,n-1$). A similar statement applies to the equivalent refinement of (\ref{eqn: H subnormal series}). Therefore $n=m$ and, using Lemma \ref{lemma: isogeny is equivalence relation}, we see that there exists a permutation $\pi$ such that
	$G_{i}/G_{i+1}$ and $H_{\pi(i)}/H_{\pi(i)+1}$ are isogenous for $i=0,\ldots,n-1$.
\end{proof}

\begin{rem}
	It is clear from the proof that the uniqueness statement in Theorem~\ref{theo: Jordan Hoelder} is valid without any restriction on the base $\s$-field $k$.
\end{rem}

See Examples  \ref{ex: Jordan Hoelder algebraic} and \ref{ex: Jordan Hoelder} for examples illustrating Theorem \ref{theo: Jordan Hoelder}.

\section{Almost-simple difference algebraic groups}

Roughly speaking, Theorem \ref{theo: Jordan Hoelder} shows that any difference algebraic group of positive $\s$\=/dimension can be decomposed into almost-simple $\s$-algebraic groups. This begs the question that we address in this final section: what are the almost-simple $\s$-algebraic groups?

We first show that for an almost-simple algebraic group $\G$, the difference algebraic group $[\s]_k\G$ is almost-simple (Proposition \ref{prop: almost simple algebrais is almost simple salgebraic}). Then, we show that, up to $\s$-dimension zero, every almost-simple $\s$-algebraic group is of this form. More precisely, if $G$ is an almost-simple $\s$\=/algebraic group, then there exists a normal $\s$-closed subgroup $N$ of $G$ with $\sdim(N)=0$ such that $G/N$ is isomorphic to $[\s]_k\G$, for some almost-simple $\s$-algebraic group $\G$ (Corollary~\ref{cor: almost simple is isogenous to almost simple}). It follows that a strongly connected $\s$-algebraic group is almost-simple if and only if it is isogenous to $[\s]_k\G$ for some almost-simple algebraic group $\G$. These results, at least to some extend, parallel results for differential algebraic groups. See \cite{CassidySinger:AJordanHoelderTheoremForDifferentialAlgebraicGroups}, \cite{Minchenko:OnCentralExtensionsOfSimpleDifferentialAlgebraicGroups}, \cite{Freitag:IndecomposabilityForDifferentialAlgebraicGroups}. However, the ideas and structure of the proofs are quite different.

\subsection{Almost-simple algebraic groups are almost-simple $\s$-algebraic groups}

We begin by recalling some definitions from the theory of algebraic groups. See, e.g., \cite{Milne:AlgebraicGroupsTheTheoryOfGroupSchemesOfFiniteTypeOverAField}.
A \emph{semisimple} algebraic group over an algebraically closed field $k$ is a smooth connected algebraic group whose radical (i.e., the maximal smooth, connected normal solvable closed subgroup) is trivial. Almost simple algebraic groups play a central role in the structure theory of semisimple algebraic groups. They are commonly defined as follows:

\begin{defi} \label{defi: almost simple}
	An algebraic group over an algebraically closed field is \emph{almost simple} if it is non-trivial, semisimple and every proper normal closed subgroup is finite, i.e., has dimension zero.
\end{defi}

Alternatively, an algebraic group is almost simple if and only if it is smooth, connected, non-commutative and every proper normal closed subgroup is finite.
 In particular, the algebraic groups $\Ga$ and $\Gm$ are not considered to be almost simple, even though they are smooth, connected and every proper normal closed subgroup is finite.  
%
%
%
For our purposes it will be convenient to have a more inclusive definition that also encompasses the multiplicative and the additive group:

\begin{defi} \label{defi: almost-simple}
	An algebraic group over an algebraically closed field is \emph{almost-simple} if it is non-trivial, smooth, connected and  every proper normal closed subgroup is finite, i.e., has dimension zero.
\end{defi}

Alternatively, an algebraic group is almost-simple if and only if it is smooth, has positive dimension and every proper normal closed subgroup has dimension zero. Arguably, Definition~\ref{defi: almost-simple salgebraic group} is the exact analog of the latter characterization.

\medskip

\noindent {\large \bf Caution:} There is a difference between ``almost simple'' and ``almost-simple''. See Definitions~\ref{defi: almost simple} and \ref{defi: almost-simple} above. As detailed in the following remark, the almost-simple algebraic groups are exactly the almost simple algebraic groups plus $\Ga$ and $\Gm$.

\begin{rem} \label{rem: classification of almost-simple algebraic groups}
	The almost-simple algebraic groups over an algebraically closed field are classified: A commutative almost-simple algebraic group is isomorphic to either the additive group $\Ga$ or the multiplicative group $\Gm$. A non-commutative almost-simple algebraic group is an almost simple algebraic group and these are classified by their root data (see, e.g., \cite[Section~24]{Milne:AlgebraicGroupsTheTheoryOfGroupSchemesOfFiniteTypeOverAField}).
\end{rem}
\begin{proof}
	A smooth connected commutative algebraic group $\G$ over an algebraically closed field is a direct product of a torus with a smooth connected unipotent algebraic group (\cite[Cor.~16.15]{Milne:AlgebraicGroupsTheTheoryOfGroupSchemesOfFiniteTypeOverAField}). Thus if $\G$ is almost-simple, then $\G$ must be isomorphic to a torus, and in this case $\G\simeq \Gm$, or $\G$ is a smooth connected unipotent group. In the later case $\G\simeq \Ga$, because a non-trivial smooth connected unipotent algebraic group contains a copy of $\Ga$ (\cite[Cor.~14.55]{Milne:AlgebraicGroupsTheTheoryOfGroupSchemesOfFiniteTypeOverAField}).
\end{proof}
We will need the following lemma.
\begin{lemma} \label{lemma: Nred normal in G}
	Assume that $k$ is perfect. Let $G$ be a reduced $\s$-algebraic group and let $N$ be a normal $\s$-closed subgroup of $G$. Then $N_\red$ is normal in $G$.
\end{lemma}

For the proof of Lemma \ref{lemma: Nred normal in G} we will use the following lemma on algebraic groups that is also used in the proof of Proposition \ref{prop: if no normal then benign}. Note that in general, even over an algebraically closed field, $\G_\red$ need not be a normal closed subgroup of $\G$. However: 
\begin{lemma} \label{lemma: Nred normal}
	Assume that $k$ is perfect.
	Let $\G$ be a smooth algebraic group and $\N$ a normal closed subgroup of $\G$. Then $\N_\red$ is normal in $\G$.
\end{lemma}
\begin{proof}
This follows from \cite[Lemma 3.2]{SnowdenWiles:BignessInCompatibleSystems}.
\end{proof}

\begin{proof}[Proof of Lemma \ref{lemma: Nred normal in G}] Let $\G$ be an algebraic group containing $G$ as a $\s$-closed subgroup. For $i\in\nn$ let $N[i]$ and $G[i]$ be the $i$-th order Zariski closure of $N$ and $G$ in $\G$ respectively. Then $N[i]$ is a normal closed subgroup of $G[i]$ (Lemma \ref{lemma: Zariski closure and normal subgroup}). As $G$ is reduced, also $G[i]$ is reduced and therefore smooth. So it follows from Lemma \ref{lemma: Nred normal} that $N[i]_\red$ is normal in $G[i]$. As $N[i]_\red=N_\red[i]$, the $i$-th order Zariski closure of $N_\red$ in $\G$, we deduce that $N_\red$ is normal in $G$ by Lemma \ref{lemma: Zariski closure and normal subgroup}.	
\end{proof}

\begin{prop} \label{prop: almost simple algebrais is almost simple salgebraic}
	Assume that $k$ is algebraically closed and inversive. Let $\G$ be an almost-simple algebraic group. Then $G=[\s]_k\G$ is an almost-simple $\s$-algebraic group. 
\end{prop}
\begin{proof}
	We know from Example \ref{ex: smooth connected implies sintegral} that $G=[\s]_k\G$ is perfectly $\s$-reduced. Moreover $\sdim(G)=\dim(\G)>0$ by Example \ref{ex: sdim=dim}. So it remains to show that for a proper normal $\s$-closed subgroup $N$ of $G$ one has $\sdim(N)=0$. For $i\in\nn$ let $N[i]$ denote the $i$-th order Zariski closure of $N$ in $\G$.
%
%

Let us first get the commutative case out of the way. So we assume $\G=\Ga$ or $\G=\Gm$ (Remark~\ref{rem: classification of almost-simple algebraic groups}).
By Theorem \ref{theo: sdimension} there exists an $e\in\nn$ such that  $\dim(N[i])=\sdim(N)(i+1)+e$ for all sufficiently large $i\in\nn$. Since $N[i]\leq \G\times\ldots\times{\hsi\G}$ and $\dim(\G\times\ldots\times{\hsi\G})=i+1$, the assumption $\sdim(N)=1$ would imply $N=G$. Thus $\sdim(N)=0$ as desired.

We now assume that $\G$ is non-commutative, i.e., almost simple.
By Lemmas \ref{lemma: Nred normal} and \ref{lemma: sdim of sssubgroups} we may assume that $N$ is reduced. It follows from Proposition \ref{prop: Go is characteristic} that $N^o$ is a normal $\s$-closed subgroup of $G$. By Corollary \ref{cor: sdim of Go} we may assume that $N$ is connected. 
Thus $N[i]$ is a smooth connected normal closed subgroup of $G[i]=\G\times\ldots\times{\hsi\G}$ for all $i\in\nn$ (Lemmas~\ref{lemma: Zariski closure and normal subgroup}  and \ref{lemma: connected component and zariski closure}). Note that $G[i]$ is semisimple and that $\G,{\hs\G},\ldots,{\hsi\G}$ can be identified with the almost simple factors of $G[i]$. As $N[i]$ is smooth and connected, it follows from \cite[Theorem 21.51]{Milne:AlgebraicGroupsTheTheoryOfGroupSchemesOfFiniteTypeOverAField} that $N[i]$ is a product of some of the almost simple factors. Let $i_0\in\nn$ be minimal with the property that $N[i_0]$ is properly contained in $G[i]$. As $N[i_0-1]=\G\times\ldots\times{^{\s^{i_0-1}}}\!\G$ and $N[i_0]$ is a product of some almost simple factors, we must have $N[i_0]=\G\times\ldots\times{^{\s^{i_0-1}}}\!\G\times 1\leq G[i_0]$. But then
 $N[i]=\G\times\ldots\times{^{\s^{i_0-1}}}\!\G\times 1\times\ldots\times 1\leq \G\times\ldots\times{\hsi\G}$ for $i\geq i_0$. Consequently $\sdim(N)=0$ as desired.
\end{proof}

\subsection{Almost-simple $\s$-algebraic groups are isogenous to almost-simple algebraic groups}

The idea of the following definition is fundamental for proving the claim made in the above headline.

\begin{defi} \label{def: emb}
Let $G$ be a $\s$-algebraic group. We denote with $\operatorname{Emb}(G)$ the collection of all morphisms $\f\colon G\to [\s]_k\G$ of $\s$-algebraic groups such that
\begin{itemize}
	\item $\G$ is an algebraic group,
	\item $\ker(\f)$ has $\s$-dimension zero,
	\item $\f(G)$ is Zariski dense in $\G$ and
	\item the kernel of $\pi_1\colon \f(G)[1]\to \f(G)[0],\ (g_0,g_1)\mapsto g_0$ has dimension $\sdim(G)$, where $\f(G)[1]$ and $\f(G)[0]$ denote the first order Zariski closure and the Zariski closure of $\f(G)$ in $\G$ respectively.
\end{itemize} 
\end{defi}

Note that if $\G$ is an algebraic group and $\f\colon G\to[\s]_k\G$ is a morphism with $\sdim(\ker(\f))=0$, then $\sdim(\f(G))=\sdim(G)$. By Proposition \ref{prop: if no normal then benign} the sequence $(\dim(\ker(\pi_i))_{i\in\nn}$, where $\pi_i\colon \f(G)[i]\to\f(G)[i-1]$, is non-increasing and stabilizes with value $\sdim(G)$. The last condition of Definition \ref{def: emb} thus signifies that the sequence $(\dim(\ker(\pi_i))_{i\in\nn}$ already stabilizes at $i=1$.

The following lemma shows that $\Emb(G)$ is non-empty. The key idea to showing that every almost-simple $\s$-algebraic group $G$ is isogenous to $[\s]_k\G$ for some almost-simple algebraic group $\G$, is to consider an element $\f\colon G\to[\s]_k\G$ of $\Emb(G)$ with $\dim(\G)$ minimal. We will eventually show that any such $\f$ is an isogeny.



\begin{lemma}
	Let $G$ be a $\s$-algebraic group. Then there exists an algebraic group $\G$ and a $\s$\=/closed embedding $\f\colon G\hookrightarrow \G$ such that $\f(G)$ is Zariski dense in $\G$ and the kernel of $\pi_1\colon \f(G)[1]\to \f(G)[0]$ has dimension $\sdim(G)$. In particular, $\Emb(G)$ is non-empty.
\end{lemma}
\begin{proof}
	By Proposition \ref{prop: exists sclosed embedding} there exists an algebraic group $\G'$ and a $\s$-closed embedding $\f'\colon G\hookrightarrow\G'$. For $i\in\nn$ let $\f'(G)[i]$ denote the $i$-th order Zariski closure of $\f'(G)$ in $\G'$ and let $\G_i'$ denote the kernel of $\pi_i\colon \f'(G)[i]\to \f'(G)[i-1]$. (By definition $\G_0=G[0]$.) By Proposition \ref{prop: limit of Gi} (i) the sequence $(\dim(\G'_i))_{i\in\nn}$ is non-increasing and stabilizes with value $\sdim(\f(G))=\sdim(G)$. Let $n\in\nn$ be minimal with the property that $\dim(\G_n')=\sdim(G)$ and set $\G=\f'(G)[n]$. By Lemma \ref{lemma: universal property with Hopf} the morphism $\f'(G)^\sharp\to \f'(G)[n]=\G$ of group schemes, corresponding to the inclusion $k[\f'(G)[n]]\subseteq k\{\f'(G)\}$, induces a morphism $\f''\colon\f'(G)\to [\s]_k\G$ of $\s$-algebraic groups. We claim that $\f\colon G\xrightarrow{\f'}\f'(G)\xrightarrow{\f''}[\s]_k\G$ has the required properties.
	Note that the dual $(\f'')^*\colon k\{\G\}\to k\{\f'(G)\}$ of $\f''$ is surjective because $k\{\f'(G)\}$ is $\s$-generated by $k[\f'(G)[0]]\subseteq k[\f'(G)[n]]=k[\G]\subseteq k\{\f'(G)\}$. Thus $\f''$ is a $\s$-closed embedding and so also $\f$ is a $\s$-closed embedding. As $k[\G]\subseteq k\{\f'(G)\}\simeq k\{G\}$ we see that $\f(G)$ is Zariski dense in $\G$. 
	
	Note that $k[\f(G)[1]]=k[k[\G],\s(k[\G])]=k[\f'(G)[n+1]]\subseteq k\{\f'(G)\}$. So we have a commutative diagram 
	$$
	\xymatrix{
	\f'(G)[n+1] \ar^-{\pi_n'}[r] \ar_\simeq[d] & \f'(G)[n] \ar^\simeq[d]\\
	\f(G)[1] \ar^{\pi_1}[r] & \f(G)[0]	
	}
	$$ 
	As the kernel of $\pi_n'$ has dimension $\sdim(G)$, we see that the kernel of $\pi_1$ has dimension $\sdim(G)$ as desired.
\end{proof}

We need a few preparatory results.

\begin{lemma} \label{lemma: embed sreduced}
	Let $G$ be a $\s$-algebraic group and let $\f\colon G\to [\s]_k\G$ be an element of $\operatorname{Emb}(G)$ such that $\dim(\G)$ is minimal. Let $\f(G)[1]\leq \G\times{\hs\G}$ denote the first order Zariski closure of $\f(G)$ in $\G$. Then the image of the projection $\s_1\colon \f(G)[1]\to {\hs\G},\ (g_0,g_1)\mapsto g_1$ has dimension $\dim(\G)$.
\end{lemma}
\begin{proof}
	Let $\H\leq {\hs\G}$ denote the image of $\s_1$.
	According to Lemma \ref{lemma: universal property with Hopf} the morphism $\f(G)^\sharp\to\f(G)[1]\xrightarrow{\s_1}\H$  of group schemes induces a morphism $\f'\colon \f(G)\to [\s]_k\H$ of $\s$-algebraic groups. We will show that $\f''=\f'\f\in\operatorname{Emb}(G)$.
	
	Let $F$ be a finite set such that $k[\f(G)[0]]=k[F]\subseteq k\{\f(G)\}$. Then $k\{\f(G)\}=k\{F\}$, $k[\H]=k[\s(F)]$ and $\f'$ corresponds to the inclusion $k\{\f''(G)\}=k\{\f'(\f(G))\}=k\{\s(F)\}\subseteq k\{\f(G)\}$. Moreover, $k[\f(G)[i]]=k[F,\ldots,\s^i(F)]$ and $k[\f''(G)[i]]=k[\f'(\f(G))[i]]=k[\s(F),\ldots,\s^{i+1}(F)]$ for $i\in\nn$.
	
	We have a surjective map $k[F]\otimes_k k[\s(F),\ldots,\s^{i+1}(F)]\to k[F,\ldots,\s^{i+1}(F)]$. Therefore $\dim(\f(G)[i+1])\leq \dim(\f(G)[0])+\dim(\f''(G)[i])$. By Theorem \ref{theo: sdimension} we have
	$\dim(\f(G)[i])=\sdim(\f'(G))(i+1)+e=\sdim(G)(i+1)+e$ for some $e\in\nn$ for all sufficiently large $i\in\nn$.
	
	Therefore
	\begin{align*}
	\dim(\f''(G)[i]) & \geq \dim(\f(G)[i+1])-\dim(\f(G)[0])=\sdim(G)(i+2)+e-\dim(\f(G)[0])= \\
	&=\sdim(G)(i+1)+e'
\end{align*}
	for all sufficiently large $i$. It follows from Theorem \ref{theo: sdimension} that $\sdim(\f''(G))\geq \sdim(G)$. So $\sdim(\f''(G))=\sdim(G)$ and $\sdim(\ker(\f''))=0$.
	
	As the map $k[\H]\to k\{\f''(G)\}$ is injective, we see that $\f''(G)$ is Zariski dense in $\H$.	
	We have a commutative diagram of $k$-Hopf algebras
	$$
	\xymatrix{
	{\hs(k[F])} \ar@{->>}[d] \ar@{^(->}[r] & {\hs(k[F,\s(F)])} \ar@{->>}[d] \\
		k[\s(F)] \ar@{^(->}[r]  & k[\s(F),\s^2(F)]	
	}
	$$
	where the vertical maps are given by applying $\s$, corresponding to the commutative diagram
	$$
	\xymatrix{
	\f''(G)[1] \ar@{->>}[r] \ar@{^(->}[d] & \f''(G)[0] \ar@{^(->}[d]  \\
	{\hs (\f(G)[1])} \ar@{->>}[r] & {\hs(\f(G)[0])}
	}
	$$
	of algebraic groups.
	As $\f''(G)[1]\hookrightarrow {\hs(\f(G)[1])}$ maps the kernel of $\f''(G)[1]\to\f''(G)[0]$ injectively into the kernel of ${\hs (\f(G)[1])}\to {\hs(\f(G)[0])}$, which has dimension $\sdim(G)$ since $\f\in\Emb(G)$, we see that the dimension of the kernel of $\f''(G)[1]\to\f''(G)[0]$ is at most $\sdim(G)$.
	By Proposition~\ref{prop: limit of Gi} the sequence $\dim(\ker(\f''(G)[i] \to \f''(G)[i-1]))_{i\geq 1}$ is non-increasing and stabilizes with value $\sdim(\f''(G))=\sdim(G)$. 
	Therefore, the dimension of $\ker(\f''(G)[1] \to \f''(G)[0])$ equals $\sdim(G)$.
	
	In summary, we find that $\f''\in\operatorname{Emb}(G)$. Thus the minimality of $\dim(\G)$ implies $\dim(\H)\geq\dim(\G)$.
\end{proof}

\begin{lemma} \label{lemma: fix with finite}
Let $G$ be a $\s$-closed subgroup of an algebraic group $\G$ such that the dimension of the kernel of $\pi_1\colon G[1]\to G[0]$ equals $\sdim(G)$. Furthermore, let $\H\leq \G$ be a smooth, connected, closed subgroup such that $\H\times{\hs \H}\subseteq G[1]$. Then $[\s]_k\H\subseteq G$. 
\end{lemma}
\begin{proof}
	Let us abbreviate $d=\sdim(G)$. It follows from the assumption that $\dim(G[1])=\dim(G[0])+d$. Using the assumption together with Proposition \ref{prop: limit of Gi} (i), we see that $\dim(G[i])=\dim(G[0])+id$ for all $i\in\nn$.

	For $i\geq 1$ let $\H_i$ denote the closed subgroup of $\G\times{\hs\G}\times\ldots\times{\hsi\G}$ given as
	$$\left(G[1]\times {^{\s^2}\!}\G\times\ldots\times {\hsi\G}\right)\cap\left(\G\times{\hs( G[1])}\times{^{\s^3}\!}\G\times\ldots\times{\hsi\G}\right)\cap\ldots\cap\left(\G\times\ldots\times{^{\s^{i-2}\!}}\G\times{^{\s^{i-1}}\!}(G[1])\right). $$ 
	Note that $\I(\H_i)=\big(\I(G[1]),\s(\I(G[1])),\ldots,\s^{i-1}(\I(G[1]))\big)\subseteq k[\G\times{\hs\G}\times\ldots\times{\hsi\G}]$. Thus $\I(\H_i)\subseteq \I(G)$ and $G[i]\subseteq \H_i$.
	
	 We will show by induction on $i$ that $\dim(\H_i)\leq\dim(G[0])+id$. As $\H_1=G[1]$ the statement is true for $i=1$. So we assume $i>1$. Note that the projection $$\pi_i\colon \G\times{\hs\G}\times\ldots\times{\hsi\G}\to \G\times{\hs\G}\times\ldots\times{^{\s^{i-1}}\G}$$ maps $\H_i$ into $\H_{i-1}$. The kernel of $\pi_i$ on $\H_i$ has dimension at most $d$ because, if $(h_0,\ldots,h_i)\in \H_i$ lies in the kernel of $\pi_i$, i.e., $(h_0,\ldots,h_i)=(1,\ldots,1,h_i)$, then $(1,h_i)$ lies in the kernel of ${^{\s^{i-1}}\!(\pi_1)}\colon{^{\s^{i-1}}\!(G[1])}\to {^{\s^{i-1}}\!(G[0])}$, which has dimension $d$. It follows that $\dim(\H_i)\leq \dim(\H_{i-1})+d\leq \dim(G[0])+id$ by the induction hypotheses. As $G[i]\subseteq \H_i$ we have indeed $\dim(\H_i)=\dim(G[0])+id$.

	So $G[i]\subseteq \H_i$ and $\dim(G[i])=\dim(\H_i)$ for all $i\geq 1$. This implies that $(\H_i^o)_\red\subseteq G[i].$ Since $\H\times {\hs\H}\subseteq G[1]$ it is clear that $\H\times{\hs\H}\times\ldots\times{\hsi\H}\subseteq \H_i$. As $\H$ is smooth and connected, the same holds for $\H\times{\hs\H}\times\ldots\times{\hsi\H}$. Thus  $\H\times{\hs\H}\times\ldots\times{\hsi\H}\subseteq (\H_i^o)_\red\subseteq G[i]$ for all $i\geq 1$. Therefore $[\s]_k\H\subseteq G$.
\end{proof}

The following proposition is the main step towards showing that every almost-simple $\s$\=/algebraic group is isogenous to an almost-simple algebraic group.

\begin{prop} \label{prop: if no normal then benign}
	Assume that $k$ is algebraically closed and inversive.
	Let $G$ be an integral $\s$-algebraic group with $\sdim(G)>0$ and let $\f\colon G\to [\s]_k\G$ be an element of $\operatorname{Emb}(G)$ such that $\dim(\G)$ is minimal. Then there exists a normal closed subgroup $\mathcal{N}$ of $\G$ with $\dim(\mathcal{N})>0$ such that $[\s]_k\mathcal{N}\subseteq\f(G)$.
	%
\end{prop}
\begin{proof}
	%
	%
	%
	As $G$ is integral also $\f(G)$ is integral. Since $\f(G)$ is Zariski dense in $\G$ (i.e., $k[\G]\to k\{\f(G)\}$ is injective) it follows that $\G$ is integral (i.e., connected and smooth).
	
	We consider the first order Zariski closure $\f(G)[1]\leq \G\times{\hs\G}$ of $\f(G)$ in $\G$ and $\pi_1\colon \f(G)[1]\to\f(G)[0]=\G,\ (g_0,g_1)\mapsto g_0$. We also have a morphism $\s_1\colon \f(G)[1]\to {\hs\G},\ (g_0,g_1)\mapsto g_1$ of algebraic groups. From Lemma \ref{lemma: embed sreduced} we know that the image of $\s_1$ has dimension $\dim(\G)$. As $\hs\G$ is integral we can conclude that $\s_1$ is a quotient map.
	Let $\G_1\leq {\hs\G}$ be such that $\ker(\pi_1)=1\times\G_1$. Similarly, let $\G_0\leq \G$ be such that $\ker(\s_1)=\G_0\times 1$. Then $\G_0\times \G_1$ is a normal subgroup of $\f(G)[1]$.
	Moreover, as $\pi_1$ is a quotient map, $\G_0$ is normal in $\G$ and because $\s_1$ is a quotient map, $\G_1$ is normal in ${\hs\G}$.
	Because $k$ is assumed to be inversive, there exists a normal closed subgroup $\G_1'$ of $\G$ with ${\hs (\G_1')}=\G_1$.  
	 Then $\mathcal{M}=\G_0\cap\G_1'$ is normal in $\G$. 
	
	Suppose $\dim(\mathcal{M})=0$, i.e., $\mathcal{M}$ is finite. Define
	$\H=\f(G)[1]/(\G_0\times \G_1)$ and 
	%
	%
	consider the morphism $\f'\colon \f(G)\to [\s]_k\H$ of $\s$-algebraic groups induced from the morphism $$\f(G)^\sharp\to \f(G)[1]\to \H$$ of group schemes as in Lemma \ref{lemma: universal property with Hopf}.
	
	We will show that $\f''=\f'\f\in\operatorname{Emb}(G)$. 
	The maps in the sequence $G^\sharp\to\f(G)^\sharp\to \f(G)[1]\to \H$ are all quotient maps. So $\f''(G)$ is Zariski dense in $\H$.
		Note that $\f'^\sharp$ is given by 
 	$$\f(G)^\sharp\to ([\s]_k\H)^\sharp,\ (g_0,g_1,\ldots)\mapsto \left(\overline{(g_i,g_{i+1})}\right)_{i\in\nn},$$
 	where $\overline{(g_i,g_{i+1})}$ denotes the image of $(g_i,g_{i+1})$ under the quotient map $${\hsi(\f(G)[1])}\to {\hsi(\f(G)[1])}/{\hsi(\G_0\times\G_1)}.$$ 
 	So if $(g_0,g_1,\ldots)\in\f(G)^\sharp\leq \G\times{\hs\G}\times\ldots$ lies in the kernel of $\f'^\sharp$, then $g_0\in \G_0,\ g_1\in\G_1,\ g_1\in {\hs\G_0},\ g_2\in {\hs\G_1},\ g_2\in{^{\s^2}\!\G_0}$ and so on.
 	As $\G_1\cap{\hs\G_0}={\hs\G_1'}\cap{\hs\G_0}={\hs(\G_1'\cap\G_0)}={\hs\mathcal{M}}$, it follows that the kernel of $\f'^\sharp$ is contained in $\G_0\times {\hs\mathcal{M}}\times {^{\s^2}\!\mathcal{M}}\times\ldots.$ Since $\mathcal{M}$ is finite, this implies that $\sdim(\ker(\f'))=0$. Using Corollary \ref{cor: sdim and ord for quotients} it follows that also $\sdim(\ker(\f''))=0$.
 		
	The quotient map $\f(G)[1]\to \G/\G_0,\ (g_0,g_1)\mapsto \overline{g_0}$ has kernel $\G_0\times \G_1$ and therefore induces an isomorphism $\eta\colon \H\to \G/\G_0$ of algebraic groups. The isomorphism ${\hs\eta}\colon{\hs\H}\to{\hs\G}/{\hs\G_0}$ has an inverse $({\hs\eta})^{-1}\colon {\hs\G}/{\hs\G_0}\to {\hs\H}$. We claim that the image of the morphism
	$$\xi\colon \f(G)[1]\to \H\times{\hs\H},\ (g_0,g_1)\mapsto \left(\overline{(g_0,g_1)}, ({\hs\eta})^{-1}(\overline{g_1})\right)$$ of algebraic groups, contains $\f''(G)[1]=\f'(\f(G))[1]$. An element of $\f'(\f(G))[1]$ is of the form $\left(\overline{(g_0,g_1)},\overline{(g_1,g_2)}\right)$ with $(g_0,g_1,g_2)\in\f(G)[2]\leq \G\times{\hs\G}\times{^{\s^2}\!}\G$, so in particular $(g_0,g_1)\in\f(G)[1]$ and $(g_1,g_2)\in{\hs(\f(G)[1])}$. As ${\hs\eta}\left(\overline{(g_1,g_2)}\right)=\overline{g_1}$, we see that $\overline{(g_1,g_2)}= ({\hs\eta})^{-1}(\overline{g_1})$. Thus $\f''(G)[1]\subseteq\xi(\f(G)[1])$ as claimed.	
	The kernel of $\xi$ is $$\G_0\times({\hs\G_0}\cap\G_1)=\G_0\times({\hs\G_0}\cap{\hs(\G'_1)})=\G_0\times{\hs\mathcal{M}}.$$ It follows that $$\dim(\f''(G)[1])\leq\dim(\xi(\f(G)[1]))=\dim(\f(G)[1])-\dim(\ker(\xi))=\dim(\f(G)[1])-\dim(\G_0).$$ Therefore
	\begin{align*}
	 \dim(\ker(\f''(G)[1]\to\f''(G)[0])) & =\dim(\f''(G)[1])-\dim(\f''(G)[0])\leq \\
	 &\leq\dim(\f(G)[1])-\dim(\G_0)-\dim(\H)= \\
	 &= \dim(\f(G)[1])-\dim(\G_0)-(\dim(\f(G)[1])-(\dim(\G_0)+\dim(\G_1))=\\ 
	 &=\dim(\G_1)=\sdim(G),
	 \end{align*}
	where the last equality above holds because $\f\in\operatorname{Emb}(G)$.
	We already know that $\sdim(\f''(G))=\sdim(G)$ (because $\sdim(\ker(\f''))=0$). It thus follows from Proposition \ref{prop: limit of Gi} (i) that the kernel of $\f''(G)[1]\to \f''(G)[0]$ has dimension $\sdim(G)$. In summary, we conclude that $\f''\in\operatorname{Emb}(G)$.

	By the minimality of $\dim(\G)$, we have $\dim(\H)\geq\dim(\G)$. But $\H\simeq \G/\G_0$ and so we must have $\dim(\G_0)=0$. As $\s_1\colon \f(G)[1]\to {\hs\G}$ is a quotient map with kernel $\G_0\times 1$, it follows that
	$\dim(\f(G)[1])=\dim(\G)$. As $\f(G)[0]|=\G$, we find $$\sdim(G)=\dim(\G_1)=\dim(\f(G)[1])-\dim(\f(G)[0])=0.$$
	This contradicts our assumption that $\sdim(G)>0$. Thus $\dim(\mathcal{M)}\geq1$.

	By construction $\mathcal{M}\times{\hs \mathcal{M}}\subseteq\f(G)[1]$.
	The identity component $\mathcal{M}^o$ of $\mathcal{M}$ is a characteristic subgroup of $\mathcal{M}$ (\cite[Prop. 1.52]{Milne:AlgebraicGroupsTheTheoryOfGroupSchemesOfFiniteTypeOverAField}). Therefore $\mathcal{M}^o$ is a normal closed subgroup of $\G$. Moreover, as $\G$ is smooth, it follows from Lemma \ref{lemma: Nred normal} that $(\mathcal{M}^o)_\red$ is normal in $\G$. Clearly $\dim((\mathcal{M}^o)_\red)=\dim(\mathcal{M})$ and so $\mathcal{N}=(\mathcal{M}^o)_\red$ is a connected, smooth, normal subgroup of $\G$ with $\dim(\N)>0$ and $\mathcal{N}\times{\hs \mathcal{N}}\subseteq \mathcal{M}\times{\hs\mathcal{M}}\subseteq \f(G)[1]$. So it follows from Lemma \ref{lemma: fix with finite} that $[\s]_k\N\subseteq \f(G)$.	
\end{proof}

%

\begin{cor} \label{cor: almost simple is isogenous to almost simple}
	Assume that $k$ is algebraically closed and inversive. Let $G$ be an almost-simple $\s$-algebraic group. Then there exists an almost-simple algebraic group $\G$ and an isogeny $G\twoheadrightarrow[\s]_k\G$.
\end{cor}
\begin{proof}
	We know from Lemma \ref{lemma: almost-simple is sintegral} that $G$ is integral. By Proposition \ref{prop: if no normal then benign} there exists an algebraic group $\G$, a normal closed subgroup $\N$ of $\G$ with $\dim(\N)>0$ and a morphism $\f\colon G\to [\s]_k\G$ such that $\sdim(\ker(\f))=0$ and $[\s]_k\N\subseteq \f(G)$. As $\N$ is normal in $\G$, $[\s]_k\N$ is normal in $[\s]_k\G$ and therefore $[\s]_k\N$ is also normal in $\f(G)$. Thus $N=\f^{-1}([\s]_k\N)$ is a normal $\s$\=/closed subgroup of $G$. As $N/\ker(\f)\simeq [\s]_k\N$ and $\sdim(\ker(\f))=0$, we see that $\sdim(N)=\sdim([\s]_k\N)=\dim(\N)>0$. So $N=G$.
	
	This implies that $\f(G)=[\s]_k\N$. So $\f\colon G\to [\s]_k\N$ is an isogeny. It remains to see that $\N$ is almost-simple. As $G$ is integral it follows that also $\mathcal{N}$ is integral, i.e., smooth and connected. Assume $\N'$ is a normal closed subgroup of $\N$ with $\dim(\N')>0$. Then $N'=\f^{-1}([\s]_k\N')$ is a normal $\s$-closed subgroup of $G$ with $\sdim(N')=\sdim(N'/\ker(\f))=\sdim([\s]_k\N')=\dim(\N')>0$. Thus $N'=G$. This implies $[\s]_k\N'=[\s]_k\N$ and so $\N=\N'$. Therefore $\N$ is almost-simple.
\end{proof}

Combining the above corollary with Proposition \ref{prop: almost simple algebrais is almost simple salgebraic}, we obtain a characterization of almost-simple $\s$-algebraic groups:

\begin{theo} \label{theo: main almost-simple}
	Assume that $k$ is algebraically closed and inversive. Let $G$ be a strongly connected $\s$-algebraic group.
	Then $G$ is almost-simple if and only if $G$ is isogenous to $[\s]_k\G$ for some almost-simple algebraic group $\G$.
\end{theo}
\begin{proof}
	An almost-simple $\s$-algebraic group is isogenous to $[\s]_k\G$ for some almost-simple algebraic group $\G$ by Corollary \ref{cor: almost simple is isogenous to almost simple}.
	If $\G$ is an almost-simple $\s$-algebraic group, then $[\s]_k\G$ is an almost-simple $\s$-algebraic group by Proposition \ref{prop: almost simple algebrais is almost simple salgebraic}. Thus the claim follows from Lemma~\ref{lemma: almost-simple and isogeny}.	
\end{proof}

While Theorem \ref{theo: main almost-simple} and Corollary \ref{cor: almost simple is isogenous to almost simple} elucidate the structure of the almost simple $\s$\=/algebraic groups, a full classification of the almost-simple $\s$-algebraic groups up to isomorphism remains a topic for future research. A natural approach to this question is to investigate how the isogeny class of an almost-simple $\s$-algebraic group splits into isomorphism classes. The following example shows that the isogeny class of an almost-simple $\s$-algebraic group may contain infinitely many isomorphism classes.

\begin{ex} \label{ex: isogenous to Gm}
		We give an example of an infinite family of pairwise non-isomorphic almost-simple $\s$-algebraic groups isogenous to $[\s]_k\Gm$. Assume that $k$ is algebraically closed and inversive.
	For $n\geq 1$ let
	$$G_n=\{(h,g)\in\Gm^2|\ \s(h)=g^n\}.$$
	Then $k\{G_n\}=k[x,x^{-1},y,y^{-1},\s(y),\s(y)^{-1},\ldots]$, with $\s(x)=y^n$, which is a $\s$-domain. So $G_n$ is $\s$-integral. In particular, $G_n$ is perfectly $\s$-reduced.

 Let $H$ be a $\s$-closed subgroup of $G_n$. The $\s$-closed subgroups of $\Gm^2$ are defined by multiplicative functions (\cite[Lemma A.40]{DiVizioHardouinWibmer:DifferenceGaloisTheoryOfLinearDifferentialEquations}). So there exist $\alpha_0,\ldots,\alpha_r,\beta_0,\ldots,\beta_s\in\mathbb{Z}$ such that
	\begin{equation} \label{eqn: multiplicative function}
	h^{\alpha_0}\s(h)^{\alpha_1}\ldots\s^r(h)^{\alpha_r}g^{\beta_0}\ldots\s^s(g)^{\beta_s}=1
	\end{equation}
	 for all $(h,g)\in H(R)$ for all \ks-algebras $R$. Using the defining equation $\s(h)=g^n$ of $H$, equation (\ref{eqn: multiplicative function}) can be transformed to an equation of the form
	 \begin{equation} \label{eqn: ex simplified}
	 h^{\alpha_0}g^{\gamma_1}\s(g)^{\gamma_2}\ldots \s^t(g)^{\gamma_t}=1.
	 \end{equation}
	If $H$ is properly contained in $G_n$, then there exists a non-trivial such relation, i.e., not all of $\alpha_0,\gamma_1,\ldots,\gamma_t$ are zero. Applying $\s$ to equation (\ref{eqn: ex simplified}) yields $g^{\alpha_0n}\s(g)^{\gamma_1}\ldots\s^{t+1}(g)^{\gamma_t}=1$. So we have a non-trivial relation $g^{\delta_0}\s(g)^{\delta_1}\ldots\s^m(g)^{\delta_m}=1$ satisfied by all $(h,g)\in G(R)$ for all \ks\=/algebras $R$. Raising this equation to the $n$-th power and replacing $g^n$ with $\s^n(h)$, we see that $\s(h)^{\delta_0}\s^2(h)^{\delta_1}\ldots\s^{m+1}(h)^{\delta_m}=1$. So, with 
	\begin{align*}
		H_1 & =\big\{h\in\Gm|\ \s(h)^{\delta_0}\s^2(h)^{\delta_1}\ldots\s^{m+1}(h)^{\delta_m}=1 \big\},\\
	H_2 & =\big\{g\in\Gm|\ g^{\delta_0}\s(g)^{\delta_1}\ldots\s^m(g)^{\delta_m}=1\big\},
	\end{align*}
	we have $H\subseteq H_1\times H_2$. As every proper $\s$-closed subgroup of $\Gm$ has $\s$-dimension zero (Proposition \ref{prop: almost simple algebrais is almost simple salgebraic}), we see, using Lemma \ref{lemma: sdim and ord for product} that
	$$\sdim(H)\leq \sdim(H_1\times H_2)=\sdim(H_1)+\sdim(H_2)=0.$$
	This shows that $G_n$ is almost-simple. The morphism $G_n\to [\s]_k\Gm,\ (h,g)\mapsto h$ is an isogeny. (The projection onto the second component is also an isogeny.)
	
	It remain to see that $G_n$ and $G_m$ are not isomorphic for $n\neq m$. To this end we consider the kernel $N_n$ of the morphism $\f\colon G_n\to {\hs G_n},\ (h,g)\mapsto (\s(h),\s(g))$. Here ${\hs G_n}$ is the $\s$\=/algebraic group over $k$ obtained from $G_n$ by base change via $\s\colon k\to k$. (In this example, in fact, ${\hs G_n}\simeq G_n$ as $G_n$ is defined over the prime field.) 
	The dual map of $\f$ is ${\hs(k\{G_n\})}=k\{G_n\}\otimes_k k\to k\{G_n\},\ f\otimes\lambda\mapsto \s(f)\lambda$, which is an invariant under isomorphism. It follows that also the kernel $N_n$ of $\f$ is an invariant under isomorphism. 
	
	So, assuming that $G_n$ and $G_m$ are isomorphic, it follows that also $N_n$ and $N_m$ are isomorphic. We have $N_n=\{(h,g)\in\Gm^2|\ \s(h)=1,\ \s(g)=1,\  g^n=1\}$ and consequently $(N_n)^\sharp=\Gm\times\mu_n$, where $\mu_n$ denotes the group of $n$-th roots of unity. As $(N_n)^\sharp/((N_n)^\sharp)^o\simeq \mu_n$, we find $\mu_n\simeq\mu_m$. Thus $m=n$.
%
	
\end{ex}

The following example exhibits a fairly general construction of almost-simple $\s$-algebraic groups that are not isomorphic to almost-simple algebraic groups, considered as $\s$-algebraic groups. We note that Example \ref{ex: isogenous to Gm} can be seen as a special case of Example \ref{ex: almost-simple} (choose $\G=\H=\Gm$ and $\pi\colon\Gm\to \Gm,\ g\mapsto g^n$).

\begin{ex} \label{ex: almost-simple}
	Assume that $k$ is algebraically closed and inversive.
	Let $\G$ be an almost-simple algebraic group, $\H$ an algebraic group, and $\pi\colon \G\to {\hs\H}$ a quotient map.
	Set $$G=\{(h,g)\in\H\times\G|\ \s(h)=\pi(g)\}.$$ Since $\s\colon [\s]_k\H\to{\hs}([\s]_k\H)$ and $[\s]_k\pi\colon [\s]_k\G\to {\hs([\s]_k\H)}$ are morphisms of $\s$-algebraic groups, we see that $G$ is a $\s$-closed subgroup of $\H\times\G$. We will show that $G$ is almost-simple and isogenous to $[\s]_k\G$. Moreover, if $\ker(\pi)$ is non-trivial, then $G$ is not isomorphic to an almost-simple algebraic group (considered as a $\s$-algebraic group).
	
	The morphism $\pi\colon \G\to {\hs\H}$ corresponds to a morphism $\pi^*\colon {\hs(k[\H])}=k[\H]\otimes_k k\to k[\G]$. The coordinate ring of $G$ is $k\{G\}=k[\H]\otimes_k k\{\G\}$ with $\s\colon k\{G\}\to k\{G\}$ given by $\s(f_1\otimes f_2)=1\otimes \pi^*(f_1\otimes 1)\s(f_2)\in k[\H]\otimes_k k\{\G\}$ for $f_1\otimes f_2\in k[\H]\otimes_k k\{\G\}$. (In particular, $G$ is $\s$-integral.)

	The projection $\f\colon G\to [\s]_k\G$ onto the second factor is a quotient map (corresponding to the inclusion $k\{\G\}\hookrightarrow k[\H]\otimes_k k\{\G\} $) and $\ker(\f)=\{(h,1)\in\H\times \G|\ \s(h)=1\}$, which has $\s$-dimension zero. To show that $G$ is almost-simple it thus suffices, by Lemma \ref{lemma: almost-simple and isogeny}, to show that $G$ is strongly connected.
	
	So let $H$ be a $\s$-closed subgroup of $G$ with $\sdim(H)=\sdim(G)$. We have to show that $H=G$. As $\f(H)$ is a $\s$-closed subgroup of $[\s]_k\G$ with $\s$-dimension $\sdim(H)=\sdim(G)=\sdim([\s]_k\G)$ and $[\s]_k\G$ is strongly connected, we see that $\f(H)=[\s]_k\G$. This signifies that the defining ideal $\I(H)\subseteq k[\H]\otimes_k k\{\G\}$ of $H$ has zero intersection with $k\{\G\}$. 
	
	Suppose $H$ is properly contained in $G$, i.e., $\I(H)\subseteq k\{G\}=k[\H]\otimes_k k\{\G\}$ contains a non-zero element $f$. As $\I(H)$ is a $\s$-ideal, it follows that $\s(f)$ is a non-zero element of $\I(H)\cap k\{\G\}$; a contradiction. Thus $G$ is almost-simple and isogenous to $[\s]_k\G$.
	
	To show that $G$ is not isomorphic to an almost-simple algebraic group, consider the kernel $N=\{(h,g)|\  \s(h)=1,\ \s(g)=1,\ \pi(g)=1 \}$ of $\s\colon G\to {\hs G}$. Then $N^\sharp=\H\times\ker(\pi)$.
	
	For an almost-simple algebraic group $\G'$, the kernel $N'$ of $\s\colon [\s]_k\G'\to [\s]_k\G'$ satisfies $(N')^\sharp=\G'$. Thus, if $\ker(\pi)$ is non-trivial, $G$ cannot be isomorphic to $[\s]_k\G'$, because otherwise $\H\times\ker(\pi)$ would be isomorphic to $\G'$ (which is impossible since $\G'$ is connected by $\H\times\ker(\pi)$ is not).
\end{ex}

We conclude with two examples illustrating Theorem \ref{theo: Jordan Hoelder}. The following example shows that in a certain sense Theorem \ref{theo: Jordan Hoelder} generalizes the Jordan-H\"{o}lder theorem for algebraic groups.

\begin{ex} \label{ex: Jordan Hoelder algebraic}
	Let $k$ be algebraically closed and inversive.
	Let $\G$ be a smooth connected algebraic group of positive dimension. Then there exists a subnormal series $\G=\G_0\supseteq\ldots\supseteq \G_n=1$ of smooth connected closed subgroups of $\G$ such that $\G_i/\G_{i+1}$ is almost-simple. By Example \ref{ex: strong component} the $\s$-algebraic groups $[\s]_k\G_i$ are strongly connected. Moreover, $[\s]_k\G_i/[\s]_k\G_{i+1}=[\s]_k(\G_i/\G_{i+1})$ is almost-simple (Example \ref{ex: quotients and algebraic groups} and Proposition \ref{prop: almost simple algebrais is almost simple salgebraic}). Thus 
	$$[\s]_k\G=[\s]_k\G_0\supseteq [\s]_k\G_1\supseteq \ldots\supseteq [\s]_k\G_n=1$$
	is a subnormal series of $G=[\s]_k\G$ as in Theorem \ref{theo: Jordan Hoelder}. 
\end{ex}

The following example shows that in Theorem \ref{theo: Jordan Hoelder} the word ``isogenous'' cannot be replaced with the word ``isomorphic''.

\begin{ex} \label{ex: Jordan Hoelder}
		Let $k$ be algebraically closed and inversive. Let $G$ be the $\s$-closed subgroup of $\Gm^3$ given by $G=\left\{(a,b,c)\in\Gm^3|\ \s(a)=bc^2 \right\}$.
		
		 Let us first show that $G$ is strongly connected. We have $$k\{G\}=k[x,x^{-1},y,y^{-1},z,z^{-1},\s(y),\s(y)^{-1},\s(z),\s(z)^{-1},\ldots]$$ with $\s(x)=yz^2$, which is a $\s$-domain. The morphism $\f\colon G\to [\s]_k\Gm^2,\ (a,b,c)\mapsto (b,c)$ is a quotient map, corresponding to the inclusion $k\{\Gm^2\}=k\{y,y^{-1},z,\s(z)^{-1},\ldots\}\subseteq k\{G\}$. The kernel $\ker(\f)=\left\{(a,1,1)\in\Gm^3|\ \s(a)=1\right\}$ has $\s$-dimension $0$. Let $H$ be a $\s$-closed subgroup of $G$ with $\sdim(H)=\sdim(G)=2$. Then $\f(H)\leq [\s]_k\Gm^2$ also has $\s$-dimension $2$. Because $[\s]_k\Gm^2$ is strongly $\s$-connected, it follows that $\f(H)=[\s]_k\Gm^2$. This signifies that the intersection $\I(H)\cap k\{y,y^{-1},z,z^{-1},\ldots\}$ is zero. However, if $f\in\I(H)$ is non-zero, then $\s(f)$ is a non-zero element of $\I(H)\cap k\{y,y^{-1},z,z^{-1},\ldots\}$. This shows that $\I(H)=0$, i.e., $H=G$ and $G$ is strongly connected.  
		 
		 Set $G_1=\left\{(1,b,c)\in\Gm^3|\ bc^2=1\right\}\leq G$. As $G\simeq[\s]_k\Gm$ (via $(1,b,c)\mapsto c$), we see that $G_1$ is strongly connected and in fact almost-simple. The quotient $G/G_1$ is isomorphic to $[\s]_k\Gm$, indeed, $G\to [\s]_k\Gm,\ (a,b,c)\mapsto a$ is a quotient map with kernel $G_1$. Thus
		 \begin{equation} \label{eq: first subnormal series}
		 G\supseteq G_1\supseteq 1
		 \end{equation}
		 is a subnormal series as in Theorem \ref{theo: Jordan Hoelder}. The quotient groups are $G/G_1\simeq [\s]_k\Gm$ and $G_1\simeq[\s]_k\Gm$.
		 
		 Set $H_1=\{(a,1,c)|\ \s(a)=c^2\}\leq G$. Then $H_1$ is isomorphic to the group labeled $G_2$ in Example \ref{ex: isogenous to Gm}. So $H_1$ is strongly connected and almost-simple. The quotient $G/H_1$ is isomorphic to $[\s]_k\Gm$. Indeed $G\to [\s]_k\Gm,\ (a,b,c)\mapsto b$ is a quotient map with kernel $H_1$.
		 Thus 
		  \begin{equation} \label{eq: second subnormal series}
		 G\supseteq H_1\supseteq 1
		 \end{equation}
		 is a subnormal series as in Theorem \ref{theo: Jordan Hoelder}. The quotient groups are $G/H_1\simeq [\s]_k\Gm$ and $H_1$. As predicted by Theorem \ref{theo: Jordan Hoelder} and verified through Example \ref{ex: isogenous to Gm}, the quotient groups $[\s]_k\Gm$, $[\s]_k\Gm$ of (\ref{eq: first subnormal series}) are isogenous with the quotient groups $[\s]_k\Gm$ and $H_1$ of (\ref{eq: second subnormal series}). Note, however, that $H_1$ is not isomorphic to $[\s]_k\Gm$ (Example \ref{ex: isogenous to Gm}).
		 
		 Theorem \ref{theo: schreier refinement} predicts that the subnormal series (\ref{eq: first subnormal series}) and (\ref{eq: second subnormal series}) have equivalent refinements. So let us find such refinements. We set $G_2=H_2=\{(1,1,c)\in\Gm^3|\ c^2=1\}\simeq [\s]_k\mu_2$. Then
		 \begin{equation} \label{eq: subnormal refined 1}
		 G\supseteq G_1\supseteq G_2\supseteq 1
		 \end{equation}
		 and 
		 	 \begin{equation} \label{eq: subnormal refined 2}
		 G\supseteq H_1\supseteq H_2\supseteq 1
		 \end{equation}
are equivalent refinements of (\ref{eq: first subnormal series}) and (\ref{eq: second subnormal series}) respectively. Indeed, the quotient groups for both subnormal series are $[\s]_k\Gm$, $[\s]_k\Gm$ and $[\s]_k\mu_2$. Note that $G_1\to[\s]_k\Gm,\ (1,b,c)\mapsto c^2$ is a quotient map with kernel $G_2$ and that $H_1\to [\s]_k\Gm,\ (a,1,c)\mapsto a$ is a quotient map with kernel $H_2$.
\end{ex}

\bibliographystyle{alpha}
\bibliography{bibdata}
\end{document}